\documentclass[12pt,a4paper]{amsart}
\usepackage{esint}
\usepackage{amssymb}
\usepackage{MnSymbol}
\usepackage{esint}
\usepackage{comment}
\usepackage{graphicx}
\usepackage{enumerate}

\usepackage[
hypertexnames=false, colorlinks, 
linkcolor=black,citecolor=black,urlcolor=black, linktocpage=true
]%
{hyperref} 
\hypersetup{bookmarksdepth=3}

	\normalbaselines						  %

\newtheorem{thm}{Theorem}[section]
\newtheorem{lm}[thm]{Lemma}
\newtheorem{cor}[thm]{Corollary}
\newtheorem{prop}[thm]{Proposition}

\newtheorem{quest}[thm]{Question}

\theoremstyle{definition}
\newtheorem{df}[thm]{Definition}
\newtheorem*{df*}{Definition}

\theoremstyle{remark}
\newtheorem{rem}[thm]{Remark}
\newtheorem*{rem*}{Remark}

\numberwithin{equation}{section}

\newcommand{\ci}[1]{_{ {}_{\scriptstyle #1}}}
\newcommand{\ti}[1]{_{\scriptstyle \text{\rm #1}}}

\newcommand{\dd}{{\mathrm{d}}}

\newcommand{\mathd}{\mathrm{d}}

\newcommand{\cB}{\mathcal{B}}

\newcommand{\cD}{\mathcal{D}}
\newcommand{\cC}{\mathcal{C}}
\newcommand{\cX}{\mathcal{X}}

\newcommand{\cM}{\mathcal{M}}

\newcommand{\cS}{\mathcal{S}}

\newcommand{\cE}{\mathcal{E}}
\newcommand{\cF}{\mathcal{F}}

\newcommand{\cK}{\mathcal{K}}

\newcommand{\cR}{\mathcal{R}}
\newcommand{\cW}{\mathcal{W}}
\newcommand{\cV}{\mathcal{V}}
\newcommand{\cA}{\mathcal{A}}

\newcommand{\T}{\mathbb{T}}

\newcommand{\f}{\varphi}

\newcommand{\e}{\varepsilon}
\newcommand{\s}{\sigma}

\newcommand{\C}{\mathbb{C}}
\newcommand{\R}{\mathbb{R}}
\newcommand{\Z}{\mathbb{Z}}

\newcommand{\Q}{\mathbb{Q}}
\newcommand{\B}{\mathbb{B}}
\newcommand{\E}{\mathbb{E}}
\newcommand{\F}{\mathbb{F}}
\newcommand{\bP}{\mathbb{P}}
\newcommand{\shi}{\mathbb{S}}
\newcommand{\cU}{\mathcal{U}}

\newcommand{\ch}{\operatorname{ch}}

\newcommand{\1}{\mathbf{1}}

\newcommand{\wt}{\widetilde}

\newcommand{\cz}{Calder\'{o}n--Zygmund\ }
\newcommand{\jr}{Journ\'{e}\ }

\newcommand{\La}{\langle }
\newcommand{\Ra}{\rangle }

\newcommand{\bs}[1]{\boldsymbol{#1}}
\newcommand{\bfD}{\bs{\mathcal{D}}}

\newcommand{\bee}{\boldsymbol{\e}}

\newcommand{\ddoto}{\"{o}}
\newcommand{\ddotu}{\"{u}}
\newcommand{\ddota}{\"{a}}

\newcommand{\fdot}{\,\cdot\,}

\newcommand{\supp}{\operatorname{supp}}

\DeclareMathOperator*{\esssup}{ess\,sup}

\newcommand{\BMOprod}{\ensuremath\text{BMO}\ci{\text{prod}}}
\newcommand{\BMOprodD}{\ensuremath\text{BMO}\ci{\text{prod},\bfD}}

\newcommand{\rp}{\text{Re}}
\newcommand{\ip}{\text{Im}}

\newcommand{\bi}{\mathbf{i}}
\newcommand{\bj}{\mathbf{j}}

\def\cyr{\fontencoding{OT2}\fontfamily{wncyr}\selectfont}
\DeclareTextFontCommand{\textcyr}{\cyr}

\newenvironment{entry}
{\begin{list}{X}%
		{%
			\setlength{\labelwidth}{55pt}%
			\setlength{\leftmargin}{\labelwidth}
			\addtolength{\leftmargin}{\labelsep}%
			\setlength{\itemsep}{.4pc}
	}%
}%
{\end{list}}

\newcounter{vremennyj}

\newcommand{\Addresses}{{
  \bigskip
  \footnotesize

  \noindent Komla~Domelevo\\
  \textsc{Julius-Maximilians-Universität Würzburg,\\
          Campus Hubland Nord,\\
          Emil-Fischer-Stra{\ss}e 40,\\
          97074 W\"{u}rzburg, Germany}\\
  \textit{E-mail address: } \texttt{komla.domelevo@mathematik.uni-wuerzburg.de}

  \bigskip
  
  \noindent Spyridon~Kakaroumpas\\
  \textsc{Julius-Maximilians-Universität Würzburg,\\
          Campus Hubland Nord,\\
          Emil-Fischer-Stra{\ss}e 40,\\
          97074 W\"{u}rzburg, Germany}\\
  \textit{E-mail address: } \texttt{spyridon.kakaroumpas@mathematik.uni-wuerzburg.de}
  
  \bigskip
  
  \noindent Stefanie~Petermichl\\
  \textsc{Julius-Maximilians-Universität Würzburg,\\
          Campus Hubland Nord,\\
          Emil-Fischer-Stra{\ss}e 40,\\
          97074 W\"{u}rzburg, Germany}\\
  \textit{E-mail address: } \texttt{stefanie.petermichl@mathematik.uni-wuerzburg.de}
  
  \bigskip

  \noindent Odí~Soler~i~Gibert\\
  \textsc{Julius-Maximilians-Universität Würzburg,\\
          Campus Hubland Nord,\\
          Emil-Fischer-Stra{\ss}e 40,\\
          97074 W\"{u}rzburg, Germany}\\
  \textit{E-mail address: }\texttt{odi.solerigibert@mathematik.uni-wuerzburg.de}

}}

\begin{document}

\title{Boundedness of Journ\'{e} operators with matrix weights}

\author{K. Domelevo, S. Kakaroumpas, S. Petermichl, O. Soler i Gibert}

\thanks{The second and third authors are partially supported by the Alexander von Humboldt Stiftung.
        The third and fourth authors are partially supported by the ERC project CHRiSHarMa no.~DLV-682402.}

\maketitle

\begin{abstract}
We develop a biparameter theory for matrix weights and provide various biparameter matrix-weighted bounds for \jr operators as well as other central operators under the assumption of the product matrix Muckenhoupt condition.
In particular, we provide a complete theory for biparameter \jr operator bounds on matrix-weighted $L^2$ spaces.
We also achieve bounds in the general case of matrix-weighted $L^p$ spaces, for $1 < p < \infty$ for paraproduct-free \jr operators.
Finally, we expose an open problem involving a matrix-weighted Fefferman--Stein inequality, on which our methods rely in the general setting of matrix-weighted bounds for arbitrary \jr operators and $p \neq 2.$ 
                   \textit{Secondary:} 42B25, 
                                       42B15.
\end{abstract}

\setcounter{tocdepth}{1}
\tableofcontents
\setcounter{tocdepth}{2}

\section*{Notation}

\begin{entry}
\item[$\Z_{+}$] the set of non-negative integers;

\item[$\1\ci{E}$] characteristic function of a set $E$;

\item[$\mathd x$] integration with respect to Lebesgue measure; 

\item[$|E|$] $n$-dimensional Lebesgue measure of a measurable set $E\subseteq\R^n$;

\item[$\La f\Ra\ci{E}$] average with respect to Lebesgue measure, $\La f\Ra\ci{E}:=\frac{1}{|E|}\int_{E}f(x)\mathd x$;

\item[$\strokedint_{E} \mathd x$] integral with respect to the normalized Lebesgue measure on a set $E$ of positive finite measure, $\strokedint_{E} f(x)\mathd x := \frac{1}{|E|} \int_{E} f(x)\mathd x = \La f\Ra\ci{E}$;

\item[$\La e,f\Ra$] usual hermitian pairing of vectors $e,f\in\C^d$ that is conjugate-linear in the second variable, $\La e,f\Ra:=\sum_{i=1}^{d}e_i\bar{f}_i$;

\item[$|e|$] usual Euclidean norm of a vector $e\in\C^d$, $|e|:=\La e,e\Ra^{1/2}$;

\item[$M_{d}(\C)$] the set of all $d\times d$-matrices with complex entries;

\item[$|A|$] usual matrix norm (i.e. largest singular value) of a matrix $A\in M_{d}(\C)$;

\item[$I_{d}$] the identity $d\times d$-matrix;

\item[$\Vert f \Vert\ci{\cX}$] norm of the element $f$ in a Banach space $\cX;$

\item[$\cX'$] topological dual of a topological vector space $\cX;$

\item[$L^{\infty}\ti{c}$] the space of compactly supported $L^{\infty}$ functions;

\item[$\cC^{\infty}\ti{c}$] the space of compactly supported smooth functions;

\item[$L^{p}(W)$] matrix-weighted Lebesgue space, $\|f\|\ci{L^p(W)}^p := \int|W(x)^{1/p}f(x)|^p \mathd x$; 

\item[$(f,g)$] usual (unweighted) $L^2$-pairing, $(f,g) := \int \La f(x),g(x)\Ra \mathd x$, where $f,g$ take values in $\C^d$;

\item[$p'$] H\"{o}lder conjugate exponent to $p$, $1/p+1/p'=1$; 

\item[$I_{-},\,I_{+}$] left and right halves, respectively, of an interval $I\subseteq\R$;

\item[$h^{0}\ci{I},~h^{1}\ci{I}$] $L^{2}$-normalized \emph{cancellative} and \emph{non-cancellative}, respectively, Haar functions for an interval $I\subseteq \R$, $h^{0}\ci{I}:=\frac{\1\ci{I_{+}}-\1\ci{I_{-}}}{\sqrt{|I|}}$, $h^{1}\ci{I}:=\frac{\1\ci{I}}{\sqrt{|I|}}$; for simplicity we denote $h\ci{I}:=h^{0}\ci{I}$;

\item[$f\ci{I}$] usual Haar coefficient of a function $f\in L^1\ti{loc}(\R;\C^d)$, $f\ci{I}:=\int_{\R}h\ci{I}(x)f(x)\mathd x$;

\item[$h\ci{R}^{\e_1\e_2}$] any of the four $L^2$-normalized Haar functions for a rectangle $R$ in $\R^2$, $h\ci{R}^{\e_1\e_2} := h\ci{I}^{\e_1}\otimes h\ci{J}^{\e_2}$, where $R = I \times J$ and $\e_1,\e_2 \in \lbrace 0,1 \rbrace$; for simplicity we denote $h\ci{R} := h\ci{R}^{00}$;

\item[$f\ci{R}$] usual (biparameter) Haar coefficient of a function $f\in L^1\ti{loc}(\R^2;\C^d)$, $f\ci{R} :=\int_{\R^2}h\ci{R}(x)f(x)\mathd x$.
\end{entry}

The notation $x\lesssim\ci{a,b,\ldots}  y$ means $x\leq Cy$ with a constant $0<C<\infty$ depending \emph{only} on the quantities $a, b, \ldots$; the notation $x\gtrsim\ci{a,b,\ldots} y$ means $y\lesssim\ci{a,b,\ldots} x.$
We use  $x\sim\ci{a,b,\ldots} y$ if \emph{both} $x\lesssim\ci{a,b,\ldots} y$ and $x\gtrsim\ci{a,b,\ldots} y$ hold.

\section{Introduction}
The theory of weights has drawn much attention in recent years. In the scalar-valued setting, we say that a non-negative locally integrable function $w$ on $\R$ is an $A_p$ weight if $$\sup_I \La w\Ra\ci I \La w^{-p'/p}\Ra\ci I^{p/p'}<\infty ,$$
where the supremum runs over intervals $I\subseteq\R$ and $\La\Ra\ci I$ returns the average of a function over the interval $I$ (here, $1<p<\infty$ is fixed, and $p':=p/(p-1)$). It is classical that this is a necessary and sufficient condition for square functions, Hardy--Littlewood maximal function, Hilbert transform and \cz operators to be bounded in the weighted Lebesgue space $L^p(w)$, $1<p<\infty$. 

Inspired by applications to multivariate stationary stochastic processes, a theory of matrix $A_2$ weights was developed by S.~Treil and A.~Volberg \cite{angle-past-future}, where a necessary and sufficient condition for the boundedness of the Hilbert transform on $L^2$ matrix-weighted Lebesgue spaces was found, the finiteness of the matrix $A_2$ characteristic. This was a difficult task at that time and required a number of new ideas. Under considerable extra effort, the result was extended to the $L^p$ case, $1<p<\infty$ by F.~Nazarov and Treil \cite{nazarov-treil}, with an alternative proof given by Volberg \cite{volberg}.  M.~Christ and M.~Goldberg \cite{Michael Christ and Goldberg}, \cite{goldberg} made highly interesting contributions to the field with their work on matrix-weighted maximal functions. The field spiked recent interest in the light of the fashion of sharp weighted norm estimates. In the scalar case, the classical questions were settled between 2000 and 2012 \cite{hukovic-treil-volberg}, \cite{hytonen}, \cite{petermichl weight} but mostly remain an unsolved mystery in the matrix case, even in the case $p=2$. For example, the dyadic square function was finally shown to have the same bound as in the scalar case \cite{hytonen-petermichl-volberg}, \cite{treil-nonhomogeneous_square_function}, but the bound for the Hilbert transform or even the Haar multiplier is still open, with the current best estimates missing that of the scalar case by a long shot, see \cite{convex body}. However, the approach in \cite{convex body} relies on a rather general convex body sparse domination principle and naturally covers all \cz operators.

In this paper, we initiate the investigation of biparameter matrix-weighted estimates. In other words, we develop a biparameter $A_p$ class and prove various square function estimates, estimates for martingale multipliers and their generalizations, the dyadic shifts, and deduce via a representation theorem estimates for Journ\'e operators. In the scalar weighted case, this difficult line of investigation dates back to Fefferman and Stein \cite{fefferman-stein}, \cite{fefferman/2}. No sharp weighted theory has been established, even in the scalar case, so our focus here is on boundedness. Various different approaches can be made to work in the matrix-weighted setting, but they have a common feature in their use of various matrix-weighted biparameter square function expressions. One important feature of square functions is that of the possibility to use Khintchine's inequality to approach higher parameter cases elegantly. Another one is the ideological place of square functions in the biparameter theory that is quite natural. The concept of stopping time is more delicate in higher parameters and the use of a bottom up approach via the use of square functions is the natural one to get, for example, an atomic decomposition of $H^1$ spaces in several parameters. While square functions also played a role in the one-parameter estimate of \cite{convex body}, they could have been completely avoided via a different approach using maximal functions. Indeed, typically the use of square functions does not result in sharp norm estimates, but in the matrix case, the use of maximal functions does not seem to help either. As explained before, in this two parameter setting here, there are numerous good reasons why square functions should arise.

Let us give a more detailed overview of the central objects in this paper. The unweighted $L^p$ theory of biparameter singular integrals was developed in the 80s, with a beautiful unified definition found by J.-L.~Journ\'e \cite{journe}; we refer to the introduction therein for an intuitive historic perspective. The questions on weighted estimates arose and some turned out to be difficult. The biparameter $A_p$ condition takes a supremum over rectangles of arbitrary eccentricity instead of intervals (or cubes). A version of a weighted Theorem \ref{t: main result} in the scalar setting first appeared in R.~Fefferman and E.~M.~Stein \cite{fefferman-stein}, with restrictive assumptions on the kernel. Subsequently, the kernel assumptions were weakened significantly by R.~Fefferman in \cite{fefferman/2}, at the cost of assuming the weight belongs to the more restrictive class $A_{p/2}$. This was due to the use of his sharp function $f\mapsto M\ci{S}(f^2)^{1/2}$, where $M_S$ is the strong maximal function. Finally, R.~Fefferman improved his own result in \cite{fefferman}, where he showed that the $A_p$ class sufficed and obtained the full statement of the weighted theorem. This was achieved by a difficult bootstrapping argument based on his previous result \cite{fefferman/2}. 
In 2017, alternative arguments were given with a modern perspective, independently \cite{barron-pipher}, \cite{holmes-petermichl-wick}  for the weighted estimates of Journ\'e operators, both benefiting strongly from an extremely useful generalization of the use of Haar shifts \cite{hytonen}, \cite{petermichl} to Journ\'e operators \cite{martikainen}, \cite{ou}. Indeed, H.~ Martikainen's work \cite{martikainen} (and Y.~ Ou's work \cite{ou} in higher parameters) has given us a wonderful new tool to tackle Journ\'e-type singular integrals. Aside from a fused coefficient in the Haar representation, the approximating biparameter Haar Shifts still have a resemblance to tensor products. Both \cite{barron-pipher} and \cite{holmes-petermichl-wick} use this representation formula for their proofs and furthermore, both arguments use ``shifted'' adapted square functions. The argument by \cite{barron-pipher} features a very beautiful sparse domination by means of adapted square functions. Thanks to these fairly recent developments in product theory, we can now tackle a problem that would have seemed completely out of reach 20 years ago. It remains remarkable, that one can obtain a $L^p$ weighted theory in the matrix case {\it and} in product spaces. 

Our main result is the following.

\begin{thm}
\label{t: main result}
Let $1<p<\infty$, and let $W,U$ be $d\times d$-matrix valued biparameter $A_p$ weights on $\R^2$ with $[W,U']\ci{A_p(\R\times\R)}<\infty$. Let $T$ be any Journ\'e operator on $\R^2$.
\begin{entry}
\item[(1)] If $T$ is paraproduct-free, then
\begin{equation*}
\Vert T\Vert\ci{L^{p}(W)\rightarrow L^{p}(W)}\lesssim_{d,p,T}[W]\ci{A_p(\R\times\R)}^{\alpha_1(p)+\alpha_2(p)}
\end{equation*}
and
\begin{equation*}
\Vert T\Vert\ci{L^{p}(U)\rightarrow L^{p}(W)}\lesssim_{d,p,T}[W,U']\ci{A_p(\R\times\R)}^{1/p}[W]\ci{A_p(\R\times\R)}^{\alpha_1(p)}[U]\ci{A_p(\R\times\R)}^{\alpha_2(p)},
\end{equation*}
where $\alpha_1(p), \alpha_2(p)$ are given by \eqref{alpha_1(p)},\eqref{alpha_2(p)} respectively.

\item[(2)] If $p=2$, then
\begin{equation*}
\Vert T\Vert\ci{L^{2}(W)\rightarrow L^{2}(W)}\lesssim_{d,T}[W]\ci{A_2(\R\times\R)}^{5}
\end{equation*}
and
\begin{equation*}
\Vert T\Vert\ci{L^{2}(U)\rightarrow L^{2}(W)}\lesssim_{d,T}[W,U']\ci{A_2(\R\times\R)}^{1/2}[W]\ci{A_2(\R\times\R)}^{5/2}[U]\ci{A_2(\R\times\R)}^{5/2}.
\end{equation*}
\end{entry}
\end{thm}

See Subsection \ref{s: definition Journe} for a brief overview of Journ\'e operators and the explanation of the terminology ``paraproduct-free'', as well as Section \ref{s: biparameter Ap} for the definition and properties of matrix-valued biparameter $A_p$ weights. Let us note that in the special case that $T$ is paraproduct-free, we obtain slightly better estimates in part (2) of Theorem \ref{t: main result} (see Corollary \ref{c: paraproduct free Journe} and Remark \ref{r: paraproduct free Journe dual} below). It should be noted that we do not expect any of the exponents of $[W]\ci{A_p(\R\times\R)},[U]\ci{A_p(\R\times\R)}$ appearing in Theorem \ref{t: main result} to be optimal (even for $p=2$).

We also obtain $L^{p}$ matrix-weighted bounds for biparameter Haar multipliers (see Lemma \ref{l: bounds Haar}), the biparameter analog of the modified dyadic Christ--Goldberg maximal function (see Proposition \ref{p: weighted bound modified strong C-G}), as well as a class of biparameter paraproducts (see \ref{s: pure biparameter paraproduct}). Moreover, we obtain $L^{p}$ upper and lower matrix-weighted bounds for biparameter square functions (see Section \ref{s: square functions}). Finally, we make substantial progress towards extending part (2) of Theorem \ref{t: main result} to any $1<p<\infty$, see Section \ref{s: non-cancellative}.

Although for simplicity we mostly restrict our attention to the product space $\R\times\R$, our arguments work in any biparameter product space $\R^n\times\R^m$, with the usual modifications.

\subsection{Difficulties of proofs}

We give here a brief discussion of the main difficulties of our proofs.

First of all, we need to develop a theory of matrix-valued biparameter $A_p$ weights almost from scratch. The powerful machinery of reducing operators introduced and developed by Goldberg in \cite{goldberg} works in great generality and allows one to define biparameter $A_p$ characteristics for matrix weights in a workable way. However, since in many estimates for biparameter operators we rely on iterating known estimates for one parameter counterparts, we also have to study the one parameter characteristics that one gets when one fixes one variable of a matrix biparameter $A_p$ weight, or when one ``averages'' over one variable. While in the scalar setting this can be readily done by just using the Lebesgue differentiation theorem (see \cite{holmes-petermichl-wick}), in the matrix setting several subtleties arise, partly due to lack of commutativity, and partly due to several measurability difficulties, absent in the case $p=2$ but crucial for $p\neq2$ (see \ref{s: iterate reducing operators} for a more detailed discussion).

Next, one of the foundational features of the theory of matrix-weighted estimates in the one-parameter case is that weights are directly incorporated in the definitions of square functions and maximal functions. Unsurprisingly, we follow this approach in our investigation of biparameter matrix-weighted estimates as well. Moreover, as in the one-parameter case, we have to distinguish between several different variants of biparameter matrix-weighted maximal and square functions, namely ones that incorporate the weight in a pointwise fashion, and ones that incorporate it in terms of reducing operators. The natural strategy for bounding these biparameter operators would be some kind of iteration, relying on the known one-parameter results. However, estimates involving crude pointwise composition of one-parameter operators acting in each direction (such as those yielding a direct proof of the weighted boundedness of the strong maximal function in the scalar setting) are not in general available; any iteration scheme must involve the weight in some way. Operators where the weight is incorporated in terms of reducing operators seem not to be amenable to iteration, while operators where the weight is incorporated in a pointwise manner seem to be only sometimes amenable to iteration. As a result, we generally try to prove bounds for the second kind of operators by using some type of properly formulated iteration, and then based on that deduce bounds for the first kind of operators using some duality technique. This strategy, while to a large extent successful for square functions (and their shifted variants), seems to fail completely for the maximal function. Instead, we give a direct argument for the biparameter analog of the modified Christ--Goldberg maximal function that cannot handle the biparameter analog of the Christ--Goldberg maximal function itself. Furthermore, to handle general Journ\'e operators, which include paraproducts in their dyadic expansion, one also needs to consider ``mixed" operators that, roughly speaking, are given as iterations in different directions of an one-parameter square function, shifted square function, or maximal function with a one-parameter operator of another of these three types. While in the scalar setting it is fairly straightforward to define and estimate such operators (see \cite{holmes-petermichl-wick}), in the present matrix-weighted setting even defining them in a reasonable way is challenging. While one can consider several different variants of them, it seems that only particular formulations of them are appropriate for estimating Journ\'e operators. See Section \ref{s: non-cancellative} for more details.

Moreover, at the time of preparation of this article there were no extrapolation results
for matrix-weighted estimates, even in the one-parameter setting.
That meant that proving $L^2$ estimates did not immediately imply $L^{p}$ estimates for $p$ different from 2.
Thus, significant additional effort and new ideas were sought to tackle the case $p\neq 2$.
Although we made considerable progress towards extending part (2) of Theorem \ref{t: main result} to any $1<p<\infty$,
there was a missing piece related to vector-valued estimates for the dyadic (one-parameter)
Christ--Goldberg maximal function that prevented us from obtaining the full result
using our methods (see Remark \ref{r: incomplete general Journe} below).
However, two more recent works on extrapolation in the setting of matrix weighted estimates,
one due to Bownik and Cruz-Uribe \cite{bownik-cruzuribe} and the second due to Vuorinen \cite{vuorinen},
were pointed out by the referees.
Two comments are required on the possible applications of these articles to the present work.
First, it is unclear that the results in \cite{bownik-cruzuribe}, and hence those in \cite{vuorinen},
are applicable to matrix weighted sublinear operators.
Second, even though the extrapolation in \cite{vuorinen} would imply $L^p$ bounds for $p \neq 2,$
that article considers real valued matrix weights, while we consider complex valued ones.
In practice, this amounts to the following change.
While \cite{bownik-cruzuribe} and \cite{vuorinen} use centrally symmetric convex bodies,
that is sets $K$ such that $rK \subseteq K$ for every $-1 \leq r \leq 1,$
one would need to check that their results still hold if one uses
centrally homothetic convex bodies,
that is sets $K$ such that $zK \subseteq K$ for every complex $|z| \leq 1.$
Regarding the vector-valued estimates mentioned before,
two questions concerning Fefferman--Stein type inequalities involving a matrix-weight remain open for some $p$ and are of independent interest.
These are mentioned in Question \ref{quest: vector valued maximal function} and Question \ref{quest: vector valued maximal function-not modified} below
and they will be addressed in subsequent works.

Finally, our methods distinguish very strongly between one-weight and two-weight situations. This is a direct consequence of the fact that the matrix weights have to be incorporated in the operators themselves, and not just appear in the weighted norms.

\subsection*{Acknowledgements}
The authors would like to thank the referees of the paper for valuable comments and suggestions that improved the exposition and its clarity.

\section{Background}
Here we recall the definition and some basic facts about dyadic cubes and rectangles and the Haar system.
For definiteness, in what follows intervals in $\R$ will always be assumed to be left-closed, right-open and bounded. A cube in $\R^n$ will be a set of the form $Q=I_1\times\ldots\times I_n$, where $I_k,~k=1,\ldots,n$ are intervals in $\R$ of the same length. A rectangle in $\R^n\times\R^m$ (with sides parallel to the coordinate axes) will be a set of the form $R=R_1\times R_2$, where $R_1$ is a cube in $\R^n$ and $R_2$ is a cube in $\R^m$.

\subsection{Dyadic grids}
A collection $\cD$ of intervals in $\R$ will be said to be a \emph{dyadic grid} if one can write $\cD=\bigcup_{k\in\Z}\cD_{k}$, such that the following hold:
\begin{enumerate}[(i)]
\item
for all $k\in\Z$, $\cD_k$ forms a partition of $\R$, and all intervals in $\cD_k$ have length $2^{-k},$
\item
for all $k\in\Z$, every $J\in\cD_{k}$ is the union of exactly $2$ intervals in $\cD_{k+1}$.
\end{enumerate}
We say that $\cD$ is the \emph{standard dyadic grid} in $\R$ if
\begin{equation*}
\cD:=\lbrace [m 2^{k},(m+1) 2^{k}):~k,m\in\Z\rbrace.
\end{equation*}

A collection $\cD$ of cubes in $\R^n$ will be said to be a \emph{dyadic grid} if for some dyadic grids $\cD^1,\ldots,\cD^n$ in $\R$ one can write $\cD=\bigcup_{k\in\Z}\cD_{k}$, where
\begin{equation*}
\cD_{k}=\lbrace I_1\times\ldots\times I_{n}:~ I_{i}\in\cD^i,~|I_{i}|=2^{-k},~i=1,\ldots,n\rbrace,\qquad\forall k\in\Z.
\end{equation*}
We say that $\cD$ is the \emph{standard dyadic grid} in $\R^n$ if
\begin{equation*}
\cD:=\lbrace [m_1 2^{k},(m_1+1) 2^{k}) \times \dots \times [m_n 2^{k},(m_n+1) 2^{k}):~k,m_1,\ldots,m_n\in\Z\rbrace.
\end{equation*}
If $\cD$ is a dyadic grid in $\R^n$, then we denote
\begin{equation*}
\ch_{i}(Q):=\lbrace K\in\cD:~K\subseteq Q,~|K|=2^{-in}|Q|\rbrace,\qquad Q\in\cD,~i=0,1,2,\ldots.
\end{equation*}
We emphasize that in this paper we follow the so called ``probabilistic indexing" of dyadic cubes, so that cubes of the $k$-th generation $\cD_k$ have sidelength $2^{-k}$ (and not $2^k$). 

A collection $\bfD$ of rectangles in $\R^n\times\R^m$ is said to be a \emph{product dyadic grid} if for a dyadic grid $\cD^1$ in $\R^n$ and a dyadic grid $\cD^2$ in $\R^m$ we have
\begin{equation*}
\bfD:=\lbrace R_1\times R_2:~R_i\in\cD^i,~i=1,2\rbrace,
\end{equation*}
and in this case we write (slightly abusing the notation) $\bfD=\cD^1\times\cD^2$.

If $\bfD$ is a product dyadic grid in $\R^n\times\R^m$, then for all $R\in\bfD$ and for all pairs of nonnegative integers $\bi=(i_1,i_2)$ we denote
\begin{equation*}
\ch_{\bi}(R):=\lbrace Q_1\times Q_2:~Q_1\in\ch_{i_{1}}(R_{1}),~Q_2\in\ch_{i_{2}}(R_{2})\rbrace,
\end{equation*}
and we say that $R$ is the $\bi$-th ancestor of a rectangle $P\in\bfD$ if $P\in\ch_{\bi}(R)$.

\subsection{Sparse families}\label{s: sparse}
Consider a dyadic grid $\cD$ in $\R^n.$
Following the terminology of \cite[Subsection~2]{convex body}, we say that a family $\cS$ of dyadic cubes in $\R^n$ is dyadically $\delta$-sparse (where $0<\delta<1$) if
\begin{equation*}
\sum_{K\in\ch\ci{\cS}(Q)}|K|\leq(1-\delta)|Q|,\qquad\forall Q\in\cS,
\end{equation*}
where $\ch\ci{\cS}(Q)$ is the family of all maximal cubes in $\cS$ that are contained in $Q$.

Consider now a product dyadic grid $\bfD$ in $\R^n \times \R^m$ and a collection $\cS$ of rectangles in $\bfD.$ Let also $0 < \delta < 1.$
We say that $\cS$ is a \emph{dyadic Carleson family} with constant $\delta$ if
\begin{equation*}
\sum_{\substack{R\in\cS\\ R\subseteq\Omega}} |R| \leq \frac{1}{\delta} |\Omega|
\end{equation*}
for all open sets $\Omega \subseteq \R^n \times \R^m.$
On the other hand, we say that $\cS$ is a \emph{weakly $\delta$-sparse} family with constant $\delta$ if for each $R \in \cS$ there is a measurable set $E_R \subseteq R$ with
\begin{equation*}
|E_R| \geq \delta |R|
\end{equation*}
and such that the collection $\lbrace E_R \rbrace\ci{R\in\cS}$ is pairwise disjoint.

It had been proved by A.~K.~Lerner and F.~Nazarov \cite{lerner-nazarov} that a collection of dyadic \emph{cubes} $\cS$ is a dyadic Carleson family if and only if it is a weakly sparse family of dyadic cubes.
Later on, T.~S.~H\"anninen \cite{hanninen} proved that a collection of dyadic rectangles $\cS$ is a dyadic Carleson family with constant $\delta$ if and only if it is a weakly $\delta$-sparse family.
In fact, he shows his result for collections of general Borel sets.
Observe as well that, if one defines one-parameter analogues of dyadic Carleson families and weakly sparse families of dyadic cubes, they are not equivalent to dyadically sparse families.
In fact, a dyadically sparse family is also weakly sparse, but the converse only holds if the weakly sparse constant satisfies $\delta > 1/2$ (see \cite[Subsection~2]{convex body}).

The concept of dyadic Carleson families appears naturally when one studies dyadic product BMO spaces, although in light of this equivalence one can just use weakly sparse families.
In this work, we mainly consider weakly sparse families of dyadic rectangles.

\subsection{Haar systems}

\subsubsection{Haar system on \texorpdfstring{$\R$}{R} and \texorpdfstring{$\R^n$}{Rn}}

Consider first a dyadic grid $\cD$ in $\R$. For any interval $I \in \cD$, $h^{0}\ci{I}, h^{1}\ci{I}$ will denote, respectively, the $L^2$-normalized \emph{cancellative} and \emph{noncancellative} Haar functions over the interval $I\in\cD$, that is
\begin{equation*}
h^{0}\ci{I}:=\frac{\1\ci{I_{+}}-\1\ci{I_{-}}}{\sqrt{|I|}},\qquad h^{1}\ci{I}:=\frac{\1\ci{I}}{\sqrt{|I|}}
\end{equation*}
(so $h^{0}\ci{I}$ has mean 0).
For simplicity we denote $h\ci{I}:=h\ci{I}^{0}$.
For any function $f \in L^1\ti{loc}(\R)$, we denote $f\ci{I}:=(f,h\ci{I})$, $I\in\cD$.
We will also denote by $Q\ci{I}$ the projection on the one-dimensional subspace spanned by $h\ci{I}$, so that
\begin{equation*}
Q\ci{I}f := f\ci{I}h\ci{I}, \qquad f \in L^1\ti{loc}(\R).
\end{equation*}
It is well-known that one has the expansion
\begin{equation*}
f = \sum_{I\in\cD} f\ci{I}h\ci{I}, \qquad \forall f\in L^2(\R)
\end{equation*}
in the $L^2(\R)$-sense, and that the system $\lbrace h\ci{I}\rbrace\ci{I\in\cD}$ forms an orthonormal basis for $L^2(\R)$.
Of course, all these notations and facts extend to $\C^d$-valued functions in the obvious coordinate-wise way.

Now, consider instead that $\cD$ is dyadic grid in $\R^n$.
Denote $\cE := \lbrace 0, 1\rbrace^n \setminus \lbrace (1,\ldots,1) \rbrace,$ which we call the set of one-parameter signatures,
and consider a cube $I = I_1 \times \dots \times I_n \in \cD.$
For $\e = (\e_1,\ldots,\e_n) \in \cE,$ we denote by $h^{\e}\ci{I}$ the $L^2$-normalized cancellative Haar function over the cube $I$ defined by
\begin{equation*}
  h^{\e}\ci{I}(x) := h^{\e_1}\ci{I_1}(x_1) \dots h^{\e_n}\ci{I_n}(x_n), \qquad x = (x_1,\ldots,x_n) \in \R^n.
\end{equation*}
In some particular cases, we will need to consider the set $\overline{\cE} = \cE \cup \lbrace (1,\ldots,1) \rbrace$ of \emph{extended} one-parameter signatures.
If this is the case and $\e = (1,\ldots,1),$ we say that $h^{\e}\ci{I} = h^{1}\ci{I_1}(x_1) \dots h^{1}\ci{I_n}(x_n)$ is the $L^2$-normalized noncancellative Haar function over the cube $I.$
For any function $f \in L^1\ti{loc}(\R^n)$ and $\e \in \overline{\cE}$, we denote $f^{\e}\ci{I} := (f,h^{\e}\ci{I})$, $I\in\cD$.
Analogously to what we did before, we will also denote by $Q^{\e}\ci{I}$ the projection on the one-dimensional subspace spanned by $h^{\e}\ci{I}$, that is
\begin{equation*}
Q^{\e}\ci{I}f := f^{\e}\ci{I}h^{\e}\ci{I}, \qquad f \in L^1\ti{loc}(\R^n).
\end{equation*}
It is well-known that one has the expansion
\begin{equation*}
f = \sum_{I\in\cD} \sum_{\e\in\cE} f^{\e}\ci{I}h^{\e}\ci{I}, \qquad \forall f\in L^2(\R^n)
\end{equation*}
in the $L^2(\R^n)$-sense, and that the system $\lbrace h^{\e}\ci{I}\colon ~I \in \cD,~\e \in \cE\rbrace$ forms an orthonormal basis for $L^2(\R^n)$.
All these notations and facts extend to $\C^d$-valued functions in the obvious coordinate-wise way.

\subsubsection{Haar system on the product space \texorpdfstring{$\R^n\times\R^m$}{RxR}}

Let $\bfD = \cD^1 \times \cD^2$ be any product grid in $\R^n \times \R^m.$
Denote $\cE := \cE^{1}\times\cE^2$, where
\begin{equation*}
\cE^1:=\lbrace 0,1 \rbrace^n \setminus \lbrace (1,\ldots,1) \rbrace,~\cE^2:=\lbrace 0,1 \rbrace^m \setminus \lbrace (1,\ldots,1) \rbrace.
\end{equation*}
We shall call $\cE$ the set of biparameter signatures.
It can also be the case that we need to consider the set $\overline{\cE} = \overline{\cE^1} \times \overline{\cE^2}$ of extended biparameter signatures, where $\overline{\cE^1} := \lbrace 0,1 \rbrace^n$ and $\overline{\cE^2} := \lbrace 0,1 \rbrace^m$.
Let us remark that the dimensions of the rectangles in the product grid $\bfD$ will be clear from the context, as well as whether $\cE$ and $\overline{\cE}$ refer to one-parameter or biparameter sets of signatures and their dimensions.

If $R = I \times J \in \bfD,$ we denote by $h\ci{R}^{\boldsymbol{\e}},$ with $\boldsymbol{\e} = (\e_1,\e_2) \in \overline{\cE},$ the $L^2$-normalized Haar function over $R$ defined by
\begin{equation*}
  h\ci{R}^{\boldsymbol{\e}} = h\ci{I}^{\e_1} \otimes h\ci{J}^{\e_2},
\end{equation*}
that is
\begin{equation*}
  h\ci{R}^{\boldsymbol{\e}}(t_1,t_2) = h\ci{I}^{\e_1}(t_1) h\ci{J}^{\e_2}(t_2),
  \qquad (t_1,t_2) \in \R^n \times \R^m.
\end{equation*}
When $\boldsymbol{\e} \in \cE,$ we say that the Haar function $h\ci{R}^{\boldsymbol{\e}},$ with $R \in \bfD,$ is a cancellative Haar function; otherwise, we say that the Haar function is non-cancellative.
For any function $f \in L\ti{loc}^1(\R^{n+m}),$ we denote $f\ci{R}^{\boldsymbol{\e}} := (f,h\ci{R}^{\boldsymbol{\e}}),$ for $R \in \bfD$ and $\boldsymbol{\e} \in    \overline{\cE}.$
As before, we will also denote by $Q\ci{R}^{\boldsymbol{\e}}$ the projection on the one-dimensional subspace spanned by $h\ci{R}^{\boldsymbol{\e}},$ that is
\begin{equation*}
  Q\ci{R}^{\boldsymbol{\e}}f := f\ci{R}^{\boldsymbol{\e}} h\ci{R}^{\boldsymbol{\e}},
  \qquad f \in L\ti{loc}^1(\R^{n+m}).
\end{equation*}
Again from the corresponding one-parameter facts, we immediately deduce the expansion
\begin{equation*}
  f = \sum_{R\in\bfD} \sum_{\boldsymbol{\e}\in\cE} f\ci{R}^{\boldsymbol{\e}} h\ci{R}^{\boldsymbol{\e}},
  \qquad \forall f \in L^2(\R^{n+m})
\end{equation*}
in the $L^2$-sense, and that the system $\lbrace h\ci{R}^{\boldsymbol{\e}}\colon R \in \bfD, \boldsymbol{\e} \in \cE \rbrace$ forms an orthonormal basis for $L^2(\R^{n+m}).$

For $I \in \cD^1,$ $J \in \cD^2$, $\e_1\in\cE^1$ and $\e_2\in\cE^2$, we denote
\begin{equation*}
  f^{\e_1,1}\ci{I}(t_2) := (f(\fdot,t_2),h\ci{I}^{\e_1}),\qquad 
  f^{\e_2,2}\ci{J}(t_1):=(f(t_1,\fdot),h\ci{J}^{\e_2}).
\end{equation*}
Moreover, we denote by $Q\ci{I}^{\e_1,1},$ $Q\ci{J}^{\e_2,2}$ the operators acting on functions $f \in L\ti{loc}^1 (\R^{n+m})$ by
\begin{equation*}
  Q\ci{I}^{\e_1,1}f(t_1,t_2) = f^{\e_1,1}\ci{I}(t_2) h^{\e_1}\ci{I}(t_1),\qquad
  Q\ci{J}^{\e_2,2}f(t_1,t_2) = f^{\e_2,2}\ci{J}(t_1)h^{\e_2}\ci{J}(t_2).
\end{equation*}
Thus, if $R = I \times J$ and $\boldsymbol{\e} = (\e_1,\e_2) \in \cE,$ then $Q\ci{R}^{\boldsymbol{\e}} = Q\ci{I}^{\e_1,1}Q\ci{J}^{\e_2,2}=Q\ci{J}^{\e_2,2}Q\ci{I}^{\e_1,1}$.
All these notations and facts extend to $\C^d$-valued functions in the obvious coordinate-wise way.

\subsection{Product BMO and product \texorpdfstring{$H^1$}{H1}}
\label{s:ProductSpaces}

Let $\bfD$ be a product grid in $\R^n \times \R^m.$
Consider a wavelet system $\{\varphi_R^{\bee}\}\ci{R\in\bfD,\bee\in\cE}$ of enough regularity adapted to the dyadic rectangles in $\bfD$ (for our purposes, we can assume that these wavelets are smooth).
The product BMO space on the product space $\R^n \times \R^m,$ denoted by $\BMOprod(\R^n\times\R^m),$ is the space of functions $f \in L^1\ti{loc}(\R^{n+m})$ for which
\begin{equation}
\label{eq:BMOProd}
    \Vert f \Vert\ci{\BMOprod(\R^n\times\R^m)}
    \coloneq \sup_{\Omega} \bigg(\frac{1}{|\Omega|} \sum_{\substack{R\in\bfD\\R\subseteq\Omega}}|( f,\varphi^{\bee}\ci{R})|^2 \bigg)^{1/2}
    < \infty,
\end{equation}
where summation over all signatures $\bee\in\cE$ is implicit, and the supremum ranges over all non-empty open sets $\Omega\subseteq\R^n\times\R^m$ of finite measure.
It can be seen that this definition does not depend on the grid $\bfD$ or the wavelet system, as long as the latter has enough regularity (see \cite{chang-fefferman}).
Observe that one can easily extend this definition to functions defined in any general product space $\R^{n_1} \times \dots \times \R^{n_t}.$
Moreover, one can restrict the open sets $\Omega$ to finite unions of dyadic rectangles by standard limiting arguments.
Furthermore, we also define the dyadic product BMO space on $\R^n \times \R^m,$ denoted by $\BMOprodD(\R^n\times\R^m),$ as the space of functions $f \in L^1\ti{loc}(\R^{n+m})$ for which
\begin{equation}
\label{eq:BMOProdD}
    \Vert f \Vert\ci{\BMOprodD(\R^n\times\R^m)}
    \coloneq \sup_{\Omega} \bigg(\frac{1}{|\Omega|} \sum_{\substack{R\in\bfD\\R\subseteq\Omega}}|f\ci{R}^{\bee}|^2 \bigg)^{1/2}
    < \infty,
\end{equation}
where summation over all signatures $\bee\in\cE$ is implicit, the supremum ranges over all non-empty open sets $\Omega\subseteq\R^n\times\R^m$ of finite measure as before and the same observations apply.
It must be noted that in \eqref{eq:BMOProd} and \eqref{eq:BMOProdD} one cannot restrict the supremum to rectangles.
This follows from a counterexample due to L.~Carleson \cite{carleson} (see also \cite{blasco-pott-cc}, \cite{fefferman-bmo} or \cite{tao}).

These spaces generalize the classical John--Nirenberg BMO and dyadic BMO spaces to the multiparameter setting.
In particular, product BMO is the dual to product $H^1$ (see \cite{chang} and \cite{fefferman-bmo}) and dyadic product BMO is the dual to dyadic product $H^1_{\bfD}$ (see \cite{bernard}).
We only state the duality between the dyadic spaces as this will be enough for our purposes.
Recall that for a function $f \in L^1\ti{loc}(\R^{n+m})$ we define its dyadic square function $S_{\bfD}f$ by
\begin{equation*}
    S_{\bfD}f \coloneq \bigg(\sum_{R\in\bfD} |f^{\bee}\ci{R}|^2 \frac{\1\ci{R}}{|R|}\bigg)^{1/2},
\end{equation*}
where summation over all signatures $\bee\in\cE$ is implicit. We say that a function $f \in L^1\ti{loc}(\R^{n+m})$ is in the dyadic space $H^1_{\bfD}(\R^n\times\R^m)$ if $S_{\bfD}f \in L^1(\R^{n+m}).$
Furthermore, if $f \in H^1_{\bfD},$ we take $\Vert f \Vert\ci{H^1_{\bfD}} \coloneq \Vert S_{\bfD}f \Vert\ci{L^1}.$
Then, the space $\BMOprodD$ is the dual of $H^1_{\bfD},$ so that in particular, if $f \in \BMOprodD,$ we have that
\begin{equation*}
    |(f,g)| \lesssim_{n,m} \Vert f \Vert\ci{\BMOprodD} \Vert g \Vert\ci{H^1_{\bfD}},
    \qquad \forall g \in H^1_{\bfD}.
\end{equation*}

\subsection{\jr operators}
\label{s: definition Journe}
\jr operators are the generalization of tensor products of \cz operators,
in the sense that they include the latter class of singular integral operators and they admit similar kernel representations.
In order to keep this exposition as self-contained as possible, we recall the definitions of both.
Let $\Delta_n = (\R^n \times \R^n) \setminus \{(x,y)\colon x=y\}.$

\begin{df}
Consider a continuous function $K\colon \Delta_n \rightarrow B,$ where $B$ is some Banach space with norm $\Vert \cdot \Vert\ci{B},$ and let $\delta \in (0,1).$
We say that $K$ is a \emph{$B$-valued $\delta$-standard kernel} on $\R^n$ if there exists a constant $C > 0$ satisfying the following properties.
\begin{enumerate}
    \item
    Size condition: if $(x,y) \in \Delta_n,$ then
    \begin{equation*}
        \Vert K(x,y) \Vert\ci{B} \leq \frac{C}{|x-y|^n}.
    \end{equation*}
    \item
    H\"{o}lder condition: if $(x,y) \in \Delta_n$ and $x' \in \R^n$ is such that $|x-x'| < |x-y|/2,$ then
    \begin{equation*}
        \begin{split}
            \Vert K(x,y)-K(x',y) \Vert\ci{B} &\leq C \frac{|x-x'|^\delta}{|x-y|^{n+\delta}},\\
            \Vert K(y,x)-K(y,x') \Vert\ci{B} &\leq C \frac{|x-x'|^\delta}{|x-y|^{n+\delta}}.
        \end{split}
    \end{equation*}
\end{enumerate}
Moreover, we denote by $C(\delta,K)$ the minimum constant satisfying these properties.
\end{df}

Here we denote the space of all smooth compactly supported functions on $\R^n$ by $\cC\ti{c}^{\infty}(\R^n),$
or simply $\cC\ti{c}^\infty$ whenever the domain is clear.
For a topological vector space $X,$ we denote its topological dual by $X'.$
Given $g \in X$ and $f \in X',$ we denote the duality pairing by $(f,g)$ in analogy to the usual unweighted $L^2$ pairing.

\begin{df}
A continuous linear mapping $T\colon \cC\ti{c}^{\infty}(\R^n) \rightarrow [\cC\ti{c}^{\infty}(\R^n)]'$ is a \emph{$\delta$-singular integral operator} if there exists a $\C$-valued $\delta$-standard kernel $K$ on $\R^n$ such that the representation
\begin{equation*}
    ( Tf,g ) = \int f(y) K(x,y) g(x) \mathd x \mathd y
\end{equation*}
holds whenever $f,g \in \cC\ti{c}^\infty$ and $\supp(f) \cap \supp(g) = \emptyset.$
\end{df}

\begin{df}
A $\delta$-singular integral operator $T$ is a \emph{$\delta$-\cz operator} on $\R^n$ if it extends boundedly from $L^2$ to itself.
If $T$ is such an operator with kernel $K,$ we define its norm by
\begin{equation*}
    \Vert T \Vert\ci{\delta CZ} = \Vert T \Vert\ci{L^2 \rightarrow L^2} + C(\delta,K).
\end{equation*}
One can see that the set of $\delta$-\cz operators on $\R^n$ equipped with this norm is a Banach space.
\end{df}

The original definition of biparameter operators due to \jr \cite{journe} generalizes the previous by looking at kernel representations of tensor products of \cz operators.

\begin{df}
A continuous linear mapping $T\colon \cC\ti{c}^{\infty}(\R^n) \otimes \cC\ti{c}^\infty(\R^m) \rightarrow [\cC\ti{c}^{\infty}(\R^n) \otimes \cC\ti{c}^\infty(\R^m)]'$ is a \emph{biparameter $\delta$-singular integral operator} on $\R^n \times \R^m$ if there exists a pair $(K_1,K_2)$ of $\delta$-\cz operator-valued $\delta$-standard kernels (on $\R^n$ and on $\R^m$ respectively) such that, for $f_1,g_1  \in \cC\ti{c}^\infty(\R^n)$ and $f_2,g_2 \in \cC\ti{c}^\infty(\R ^m),$ the representation
\begin{equation*}
    ( T(f_1 \otimes f_2),g_1 \otimes g_2 ) = \int_{\Delta_n} f_1(y) (K_1(x,y)f_2,g_2) g_1(x) \mathd x \mathd y
\end{equation*}
holds whenever $\supp(f_1) \cap \supp(g_1) = \emptyset$ and the representation
\begin{equation*}
    ( T(f_1 \otimes f_2),g_1 \otimes g_2 ) = \int_{\Delta_m} f_2(y) (K_2(x,y)f_1,g_1) g_2(x) \mathd x \mathd y
\end{equation*}
holds whenever $\supp(f_2) \cap \supp(g_2) = \emptyset.$
\end{df}

Let us define for a continuous linear mapping $T\colon \cC\ti{c}^{\infty}(\R^n) \otimes \cC\ti{c}^\infty(\R^m) \rightarrow [\cC\ti{c}^{\infty}(\R^n) \otimes \cC\ti{c}^\infty(\R^m)]'$ the partial adjoint $T_1$ by
\begin{equation}
\label{eq:PartialAdjoint}
   ( T_1(f_1 \otimes f_2), g_1 \otimes g_2 ) = ( T(g_1 \otimes f_2), f_1 \otimes g_2 ).
\end{equation}
Observe that if $T$ is a biparameter $\delta$-singular integral operator, so is $T_1.$

\begin{df}
A biparameter $\delta$-singular integral operator $T$ is a \emph{biparameter $\delta$-\jr operator} if both $T$ and $T_1$ extend boundedly from $L^2$ to itself.
\end{df}

In our applications it will be especially relevant to distinguish the case of \emph{paraproduct-free \jr operators}.
We recall their definition for the reader's convenience.

\begin{df}
We say that a \jr operator $T$ is a \emph{paraproduct-free \jr operator} if
\begin{equation*}
T(\cdot \otimes 1) = T(1 \otimes \cdot) = T^\ast(\cdot \otimes 1) = T^\ast(1 \otimes \cdot) = 0.    
\end{equation*}
\end{df}

\begin{rem}
Note that the previous definitions can be easily extended to the general multiparameter setting.
\end{rem}

There is a more recent definition of $\delta$-\jr operators introduced by S.~Pott and P.~Villarroya \cite{pott-villarroya}.
This one assumes a purely biparameter kernel representation for all suitable pairs of functions.
Moreover, for each pair of (suitable) functions, it assumes the existence of a pair of one-parameter $\C$-valued kernels yielding partial representations on each variable, with some control in their size and regularity.
It also replaces the boundedness assumption by some cancellation and weak boundedness properties.
We refer to \cite{pott-villarroya} for the details.
Note that it was shown by A.~Grau~de~la~Herr\'{a}n \cite{GrauDeLaHerran} that both definitions are equivalent.

\subsubsection{Random dyadic lattices and Martikainen's representation theorem}

One of the most fundamental tools used in the present paper is H.~Martikainen's \cite{martikainen} representation theorem for Journ\'e operators in terms of fast decaying averages of Haar shifts. These averages are taken over ``random dyadic grids". The powerful technique of averaging over random dyadic grids was introduced by Nazarov--Treil--Volberg in \cite{tb-nazarov-treil-volberg} for the purpose of proving $Tb$ theorems in nonhomogeneous spaces. This technique was subsequently used by T.~Hyt\ddoto nen \cite{hytonen} for the representation theorem for \cz operators, leading to a proof of the $A_2$ conjecture (see also \cite{hytonen-perez-treil-volberg} and \cite{volberg1}), as well as by Martikainen \cite{martikainen} for the aforementioned representation theorem for Journ\'e operators.

Let $\cD_{n}^{0}$ be the standard dyadic grid in $\R^n.$
Following the probabilistic indexing of cubes we denote by $\cD_{n,k}^{0}$ the set of all cubes in $\cD_{n}^{0}$ of sidelength $2^{-k}$, for all $k\in\Z$. Let $\Omega_{n}$ be the set of all maps $\omega:\Z\rightarrow\lbrace0,1\rbrace^{n}$. We equip $\Omega_n$ with the natural product probability measure $\bP_{n}$ that to each coordinate assigns each of the values 0,1 with probability equal to $1/2$, and independently of the other coordinates. For each $k\in\Z$, for each $P\in\cD_{n,k}^{0}$ and for each $\omega\in\Omega_n$ we denote
\begin{equation*}
\omega\dotplus P:=\omega^{k}\plus P:=\lbrace x+\omega^{k}:~x\in P\rbrace,
\end{equation*}
where
\begin{equation*}
\omega^{k}:=\sum_{j=k+1}^{\infty}\omega(j)2^{-j}.
\end{equation*}
Then, for each $\omega\in\Omega$, we set
\begin{equation*}
\cD_{n,k}(\omega):=\lbrace \omega\dotplus P:~P\in\cD_{n,k}^{0}\rbrace
\end{equation*}
and
\begin{equation*}
\cD_{n}(\omega):=\bigcup_{k\in\Z}\cD_{n,k}(\omega).
\end{equation*}
It is easy to see that $\cD_{n}(\omega)$ is a dyadic grid on $\R^n$, for all $\omega\in\Omega_n$. Note that $\cD_{n}^{0}=\cD_{n}(0)$. In what follows, we will say that $\cD_n(\omega)$ is a system of \emph{random dyadic grids} on $\R^n$ parametrized by $\omega$.

To state Martikainen's representation theorem from \cite{martikainen} precisely, we need one more definition. Let $\bi=(i_1,i_2),\bj=(j_1,j_2)$ be pairs of nonnegative integers, and let $\bfD$ be any product dyadic grid in $\R^n\times\R^m$. A \emph{biparameter Haar shift} $T^{\bi,\bj}$ of complexity $(\bi,\bj)$ with respect to the product dyadic grid $\bfD$ is an operator of the form
\begin{equation*}
T^{\bi,\bj}f = \sum_{R\in\bfD} \sum_{P\in\ch_{\bi}(R)} \sum_{Q\in\ch_{\bj}(R)} a\ci{RPQ}^{\boldsymbol{\e}} f\ci{P}^{\boldsymbol{\e}} h\ci{Q}^{\boldsymbol{\e}}, \qquad f \in L^2(\R^{n+m}),
\end{equation*}
where summation over \emph{all} possible signatures $\boldsymbol{\e} \in \overline{\cE}$ is implicit and
\begin{equation*}
|a_{RPQ}^{\boldsymbol{\e}}| \leq \frac{\sqrt{|P| |Q|}}{|R|} = 2^{-\frac{1}{2}(i_1+i_2+j_1+j_2)}.
\end{equation*}
Observe that it can either happen that all non-zero coefficients $a_{RPQ}^{\boldsymbol{\e}}$ correspond to cancellative Haar functions, in which case the shift is called \emph{cancellative}, or that there appear non-cancellative terms in the sum, in which case it is called a \emph{non-cancellative} shift.

We can now state Martikainen's representation theorem from \cite{martikainen} precisely.

\begin{thm}[Martikainen \cite{martikainen}]
\label{t: martikainen representation}
Let $T$ be a $\delta$-Journ\'e operator on $\R^n\times\R^m$ (as defined above). Let $\cD_{n}(\omega_n)$, $\cD_{m}(\omega_m)$ be systems of random dyadic grids on $\R^n,\R^m$ respectively parametrized by $\omega_n,\omega_m$ respectively. Then, there exist a constant $C$ (depending only on $T,n,m$) as well as a family of biparameter Haar shifts
\begin{equation*}
\lbrace T^{\bi,\bj}\ci{\omega_n,\omega_m}:~\bi,\bj\in\Z^{2}_{+},~\omega_n\in\Omega_n,~\omega_m\in\Omega_m\rbrace
\end{equation*}
(depending only on $T,n,m$), where for all $\bi,\bj\in\Z^{2}_{+}$ and for all $\omega_n\in\Omega_n,\omega_m\in\Omega_m$, $T^{\bi,\bj}\ci{\omega_n,\omega_m}$ is a biparameter Haar shift of complexity $(\bi,\bj)$ with respect to the product dyadic grid $\cD_{n}(\omega_n)\times\cD_{m}(\omega_m)$ in $\R^n\times\R^m$, such that
\begin{equation}
\label{martikainen representation}
(Tf,g)= C\E_{\omega_n} \E_{\omega_m} \sum_{\substack{\bi=(i_1,i_2)\in\Z^2_{+}\\\bj=(j_1,j_2)\in\Z^2_{+}}} 2^{-\max(i_1,i_2)\delta/2} 2^{-\max(j_1,j_2)\delta/2} (T^{\bi,\bj}_{\omega_n,\omega_m}f,g),
\end{equation}
where $\E_{\omega_n},\E_{\omega_m}$ denote expectation with respect to $\mathd\bP_{n}(\omega_n),\mathd\bP_{m}(\omega_m)$ respectively. Moreover, non-cancellative Haar shifts in \eqref{martikainen representation} can only appear for $\bi=(0,0)$ or $\bj=(0,0)$.
\end{thm}

In the particular case of paraproduct-free Journ\'e operators, it follows from Martikainen's proof \cite{martikainen} that all Haar shifts in \eqref{martikainen representation} will be of cancellative type.

\section{One-parameter and biparameter matrix \texorpdfstring{$A_p$}{Ap} weights}
\label{s: biparameter Ap}

In what follows, we denote by $\lbrace e_1,\ldots,e_d\rbrace$ the standard basis of $\C^d$. We will be often using the fact that
\begin{equation}
\label{equivalence_matrix_norm_columns}
|A|\sim_{d}\sum_{k=1}^{d}|Ae_k|,\qquad\forall A\in M_{d}(\C).
\end{equation}
In particular, if $A,B,C,D\in M_{d}(\C)$ are self-adjoint matrices such that $|Ae|\lesssim_{d}|Ce|$ and $|Be|\lesssim_{d}|De|$, for all $e\in\C^d$, then we may apply~\eqref{equivalence_matrix_norm_columns} twice to get
\begin{equation}
\label{comparability_from_vector_to_matrix}
|AB| \lesssim_{d} |CB| = |(CB)^{\ast}| = |BC| \lesssim_{d} |DC| = |CD|.
\end{equation}
Moreover, we will be often using the elementary fact that for any $s\in(0,\infty)$, for any positive integer $k$ and for any $x_1,\ldots,x_k\in[0,\infty)$ there holds
\begin{equation}
\label{break power of sum}
\min(1,k^{s-1})(x_1^{s}+\ldots+x_k^{s})\leq(x_1+\ldots+x_k)^{s}\leq\max(1,k^{s-1})(x_1^{s}+\ldots+x_k^{s}).
\end{equation}
This is just a consequence of convexity or concavity (depending on $s$) and Jensen's inequality (or H\ddoto lder's inequality). 

\subsection{Matrix-weighted Lebesgue spaces} A function $W$ on $\R^n$ is said to be a \emph{$d\times d$-matrix valued weight} if it is a locally integrable $M_{d}(\C)$-valued function such that $W(x)$ is a positive-definite matrix for a.e.~$x\in\R^n$.
Given a $d\times d$-matrix valued weight $W$ on $\R^n$ and $1<p<\infty$, we define the weighted space $L^{p}(W)$ norm as
\begin{equation*}
\Vert f\Vert\ci{L^{p}(W)}:=\left(\int_{\R^n}|W(x)^{1/p}f(x)|^{p}\mathd x\right)^{1/p},
\end{equation*}
for all measurable $\C^d$-valued functions $f$ on $\R^n$ for which this quantity is finite. Assume that $W':=W^{-p'/p}=W^{-1/(1-p)}$, where $p':=p/(p-1)$, is also a $d\times d$-matrix valued weight on $\R^n$. Then, for all (suitable) $\C^d$-valued functions $f$ on $\R^n$, we have
\begin{align*}
\Vert f\Vert\ci{L^{p}(W)}&=\sup\left\lbrace\left|\int_{\R^n}\La W(x)^{1/p}f(x),g(x)\Ra \mathd x\right|:~\Vert g\Vert\ci{L^{p'}}=1\right\rbrace\\
&=\sup\left\lbrace\left|\int_{\R^n}\La f(x),h(x)\Ra \mathd x\right|:~\Vert W^{-1/p}h\Vert\ci{L^{p'}}=1\right\rbrace\\
&=\sup\left\lbrace\left|\int_{\R^n}\La f(x),h(x)\Ra \mathd x\right|:~\Vert h\Vert\ci{L^{p'}(W')}=1\right\rbrace.
\end{align*}
Therefore, under the standard (unweighted) $L^2$-pairing, we have $(L^{p}(W))^{\ast}=L^{p'}(W')$.

\subsection{Reducing matrices}

In this subsection we recall a few facts about reducing matrices that are important for our purposes.
Reducing matrices, were first been introduced by Nazarov--Treil \cite{nazarov-treil}
and by Volberg \cite{volberg} and have also been used by M.~Goldberg \cite{goldberg}
to study $L^{p}$ bounds for matrix-weighted maximal functions for $p\neq2$.
These can be understood as substitutes of averages of matrix-valued weights whenever one is interested in the case $p\neq 2$.

%

Let $E$ be a bounded measurable subset of $\R^n$ of non-zero measure. In what follows, $E$ will always be considered to be equipped with normalized Lebesgue measure; in particular, we follow this convention in the notation $L^{p}(E;C^{d})$, $1<p<\infty$, and we emphasize it when integrating by using the notation $\strokedint_E \mathd x.$ Let $W$ be a $M_{d}(\C)$-valued function on $E$ that is integrable (meaning that $\strokedint_{E}|W(x)|\mathd x<\infty$) and that takes a.e.~values in the set of positive-definite $d\times d$-matrices. In the rest of this subsection, we fix $1<p<\infty$. We denote
\begin{equation*}
\Vert f\Vert\ci{L^{p}(E,W)}:=\left(\strokedint_{E}|W(x)^{1/p}f(x)|^{p}\mathd x\right)^{1/p},
\end{equation*}
for all $\C^d$-valued measurable functions $f$ on $E$. Notice that by the spectral theorem we have $|W(x)^{1/p}|^{p}=|W(x)|$ for a.e. $x\in E$, therefore $W^{1/p}$ is $p$-integrable over $E$.
Using standard properties of John ellipsoids, Goldberg \cite[Proposition 1.2]{goldberg} proved that there exists a (not necessarily unique) positive-definite matrix $\cW\ci{E}\in M_{d}(\C)$, such that
\begin{equation}
\label{reducing_matrices_definition}
\left(\strokedint_{E}|W(x)^{1/p}e|^{p}\mathd x\right)^{1/p}\leq |\cW\ci{E} e|\leq\sqrt{d}\left(\strokedint_{E}|W(x)^{1/p}e|^{p}\mathd x\right)^{1/p},\qquad\forall e\in\C^d.
\end{equation}
We emphasize that the constants in the inequalities in \eqref{reducing_matrices_definition} do not depend on $W$, the set $E$ or $p$: all this information has been absorbed in $\cW\ci{E}$. The matrix $\cW\ci{E}$ is called \emph{reducing matrix} (or \emph{reducing operator}) of $W$ over $E$ corresponding to the exponent $p$. We emphasize that in this paper, reducing matrices are always chosen to be self-adjoint. If $d=1$, i.e. $W$ is scalar-valued, then we simply take $\cW\ci{E}:=\La W\Ra\ci{E}^{1/p}$. On the other hand, if $p=2$ and for any $d,$ we have
\begin{equation*}
\strokedint_{E}|W(x)^{1/2}e|^{2}\mathd x=\strokedint_{E}\La W(x)e,e\Ra \mathd x=\La \La W\Ra\ci{E}e,e\Ra=|\La W\Ra\ci{E}^{1/2}e|^2,\qquad\forall e\in\C^d,
\end{equation*}
and thus in this special case we will take $\cW\ci{E}:=\La W\Ra\ci{E}^{1/2}$. We note that for \emph{any} $A\in M_{d}(\C)$, using \eqref{equivalence_matrix_norm_columns} twice we get
\begin{equation}
\label{reducing_matrix_action_on_matrix}
\left(\strokedint_{E}|W(x)^{1/p}A|^{p}\mathd x\right)^{1/p} \sim_{d,p} |\cW\ci{E}A|.
\end{equation}
If the function
\begin{equation}
\label{conjugate weight}
W':=W^{-1/(p-1)}
\end{equation}
(sometimes called \emph{Muckenhoupt conjugate weight} to $W$) also happens to be integrable over $E$, then we let $\cW'\ci{E}$ be the reducing matrix of $W'$ over $E$ corresponding to the exponent $p':=p/(p-1)$, so that
\begin{equation*}
|\cW'\ci{E} e|\sim_{d}\left(\strokedint_{E}|W'(x)^{1/p'}e|^{p'}\mathd x\right)^{1/p'}=\left(\strokedint_{E}|W(x)^{-1/p}e|^{p'}\mathd x\right)^{1/p'},\qquad\forall e\in\C^d.
\end{equation*}
Note that it is clear by the definitions that in this case one can take $\cW\ci{E}''=\cW\ci{E}$. Notice also that for all measurable $\C^d$-valued functions $f$ on $E$ we have by H\ddoto lder's inequality that
\begin{align*}
\strokedint_{E}|f(x)|\mathd x &\leq \left(\strokedint_{E}|W(x)^{-1/p}|^{p'}\mathd x\right)^{1/p'} \left(\strokedint_{E}|W(x)^{1/p}f(x)|^{p}\mathd x\right)^{1/p}\\
&=\left(\strokedint_{E}|W'(x)|\mathd x\right)^{1/p'}\left(\strokedint_{E}|W(x)^{1/p}f(x)|^{p}\mathd x\right)^{1/p},
\end{align*}
therefore if $W'\in L^1(E;M_{d}(\C))$, then $L^{p}(E,W)\subseteq L^1(E;\C^d)$.

The following lemma is stated and proved for $p=2$ in \cite[Lemma 2.1]{angle-past-future}, and in a slightly different form for any $1<p<\infty$ in \cite[Proposition 2.1]{goldberg} (adapting the ideas from \cite[Lemma 2.1]{angle-past-future}). For the sake of clarity we give the full statement here, which was explicitly proved in \cite{isralowitz-pott-treil}. Let $A\ci{E}$ be the averaging operator over $E$, i.e. $A\ci{E}$ acts on functions $f\in L^1(E;\C^d)$ by
\begin{equation*}
A\ci{E}f:=\La f\Ra\ci{E}\1\ci{E}.
\end{equation*}

\begin{lm}
\label{l: averaging operators}
Set
\begin{equation*}
C\ci{E}:=\strokedint_{E}\left(\strokedint_{E}|W(x)^{1/p}W'(y)^{1/p'}|^{p'}\mathd y \right)^{p/p'}\mathd x.
\end{equation*}
Then, $A\ci{E}$ is bounded as an operator acting from $L^p(E,W)$ into $L^p(E,W)$ if any only if $C\ci{E}<\infty$, and in this case $W'$ is integrable over $E$ and
\begin{equation*}
\Vert A\ci{E}\Vert\ci{L^{p}(E,W)\rightarrow L^{p}(E,W)}\sim_{p,d}|\cW'\ci{E}\cW\ci{E}|\sim_{p,d}C\ci{E}^{1/p}.
\end{equation*}
\end{lm}

Let us observe that more generally, if $V$ is another integrable $M_{d}(\C)$-valued function on $E$ that takes a.e.~values in the set of positive definite $d\times d$-matrices, then denoting by $\cV\ci{E}$ the reducing matrix of $V$ over $E$ corresponding to the exponent $p'$ and using \eqref{reducing_matrix_action_on_matrix} again and self-adjointness, we obtain
\begin{align*}
&\left(\strokedint_{E}\left(\strokedint_{E}|W(x)^{1/p}V(y)^{1/p'}|^{p'}\mathd y \right)^{p/p'}\mathd x\right)^{1/p}\\
&\sim_{p,d}\left(\strokedint_{E}|\cV\ci{E} W(x)^{1/p}|^{p}\mathd x\right)^{1/p} \sim |\cW\ci{E}\cV\ci{E}|\\
&=|\cV\ci{E}\cW\ci{E}| \sim_{d,p} \left(\strokedint_{E}\left(\strokedint_{E}|V(x)^{1/p'}W(y)^{1/p}|^{p}\mathd y \right)^{p'/p}\mathd x\right)^{1/p'},
\end{align*}

In the rest of this subsection we will denote
\begin{equation*}
C\ci{E,W,p} \coloneq \strokedint_{E}\left(\strokedint_{E}|W(x)^{1/p}W'(y)^{1/p'}|^{p'}\mathd y \right)^{p/p'}\mathd x
\end{equation*}
and we will assume that $C\ci{E,W,p} < \infty,$ so that $W'\in L^1(E;M_{d}(\C))$ and
\begin{equation*}
    \Vert A\ci{E}\Vert\ci{L^{p}(E,W)\rightarrow L^{p}(E,W)} \sim_{p,d} C\ci{E,W,p}^{1/p}
\end{equation*}
by Lemma \ref{l: averaging operators} with this notation.
We remark that it follows from the proof of Lemma \ref{l: averaging operators} that in the special case $d=1$ we just have
\begin{equation}
    \label{eq:ScalarWeightAveragingOp}
    \Vert A\ci{E}\Vert^p\ci{L^{p}(E,W)\rightarrow L^{p}(E,W)} = C\ci{E,W,p} = \La W\Ra\ci{E}\La W^{-1/(p-1)}\Ra\ci{E}^{p-1}.
\end{equation}

The following lemma is stated in a slightly different form in \cite[Lemma~2.2]{isralowitz-kwon-pott}.
In this statement we also include two explicit observations that follow from elementary arguments for scalar weights (see also \cite{goldberg}).

\begin{lm}
\label{l: second scalar}
(1) Assume that $d=1$. Denote $w:=W$. Then
\begin{equation*}
\La w\Ra\ci{E}^{1/p}\leq C\ci{E,w,p}^{1/p}\La w^{1/p}\Ra\ci{E}.
\end{equation*}

(2) For any $d,$ let $e\in\C^d\setminus\lbrace0\rbrace$ be arbitrary. Set $w_{e}:=|W^{1/p}e|^{p}$. Then
\begin{equation*}
\left(\strokedint_{E}w_{e}(x)\mathd x\right)\left(\strokedint_{E}w_{e}^{-1/(p-1)}(x)\mathd x\right)^{p-1}\lesssim_{p,d} C\ci{E,W,p}.
\end{equation*}

(3) For any $d$ and for all $e\in\C^d$, there holds
\begin{equation*}
|\cW\ci{E} e|\lesssim_{p,d}C\ci{E,W,p}^{1/p}\La |W^{1/p}e|\Ra\ci{E}.
\end{equation*}
\end{lm}

Note that by~\eqref{reducing_matrices_definition} and Jensen's inequality, for all $e\in\C^d$ we have that 
\begin{equation}
\label{trivial bound}
|\cW\ci{E}e|\geq\La|W^{1/p}e|^{p}\Ra\ci{E}^{1/p}\geq \La|W^{1/p}e|\Ra\ci{E}\geq |\La W^{1/p}\Ra\ci{E}e|.
\end{equation}
We also note that a stronger form of part (3) of Lemma \ref{l: second scalar} is proved in \cite[Lemma 2.2]{isralowitz-kwon-pott}, in fact implicitly under the weaker assumption
\begin{equation*}
\exp\left(\strokedint_{E}\log|W(x)^{1/p}e|\mathd x\right)\leq C|(\cW\ci{E}')^{-1}e|,\quad\forall e\in\C^d
\end{equation*}
for some $0<C<\infty$ introduced in \cite{volberg} as the defining property of the so-called $A_{p,\infty}$ class of matrix weights defined there (see also \cite{cruz-uribe-isralowitz-moen} and \cite{convex body} for further variants of this class, as well as \cite{duoandikoetxea--martin-reyes--ombrosi} for a detailed discussion of the situation in the scalar case that motivates the various definitions in the matrix case).

The next two results are useful facts that will be necessary later on.
We include the details for the reader' convenience, although they are just minor modifications of known results.
In particular, Lemma~\ref{l: first scalar} is an adaptation of \cite[Corollary~2.3]{goldberg}.

\begin{lm}
\label{l: sum of scalar A_p weights}
Let $k$ be any positive integer. Let $w_1,\ldots,w_k$ be a.e.~positive integrable functions on $E$ such that $w_{i}^{-1/(p-1)}\in L^1(E)$, for all $i=1,\ldots,k$. Set
\begin{equation*}
c:=\max\lbrace \La w_i\Ra\ci{E}\La w_i^{-1/(p-1)}\Ra^{p-1}\ci{E}:~i=1,\ldots,k\rbrace    
\end{equation*}
and
\begin{equation*}
w:=\sum_{i=1}^{k}w_i.
\end{equation*}
Then $\La w\Ra\ci{E}\La w^{-1/(p-1)}\Ra\ci{E}^{p-1}\leq c$.
\end{lm}

\begin{proof}
Note that $w':=w^{-1/(p-1)}\in L^1(E)$ since, for example, $w'\leq w_1':=w_1^{-1/(p-1)}\in L^1(E).$ Observe now that for all $f\in L^1(E)$, by Lemma \ref{l: averaging operators} we have
\begin{align*}
\Vert A\ci{E}f\Vert\ci{L^{p}(E,w)}^{p}=\sum_{i=1}^{k}\Vert A\ci{E}f\Vert\ci{L^{p}(E,w_i)}^{p}\leq c\sum_{i=1}^{k}\Vert f\Vert\ci{L^{p}(E,w_i)}^{p}=c\Vert f\Vert\ci{L^{p}(E,w)}^{p}.
\end{align*}
Appealing again to Lemma \ref{l: averaging operators} we deduce the required result.
\end{proof}

\begin{lm}
\label{l: first scalar}
Let $A$ be any invertible matrix in $M_{d}(\C)$. Consider the scalar-valued function
\begin{equation*}
u:=|W^{1/p}A|^{p}.
\end{equation*}
Then
\begin{equation*}
\left(\strokedint_{E}u(x)\mathd x\right)\left(\strokedint_{E}u(x)^{-1/(p-1)}\mathd x\right)^{p-1}\lesssim_{p,d} C\ci{E,W,p}.
\end{equation*}
\end{lm}

\begin{proof}
Denote by $\lbrace e_1,\ldots,e_d\rbrace$ the standard basis in $\C^d.$
By~\eqref{equivalence_matrix_norm_columns} and~\eqref{break power of sum} we have
\begin{equation*}
u\sim_{p,d}\sum_{k=1}^{d}|W^{1/p}v_k|^{p},
\end{equation*}
where $v_k:=Ae_k$, $k=1,\ldots,d$. Set $w_{k}:=|W^{1/p}v_k|^{p}$, $k=1,\ldots,d$. By part (2) of Lemma \ref{l: second scalar} we have
\begin{equation*}
\La w_k\Ra\ci{E}\La w_k^{-1/(p-1)}\Ra\ci{E}^{p-1}\lesssim_{p,d}C\ci{E,W,p},\qquad\forall k=1,\ldots,d.
\end{equation*}
Therefore, by Lemma \ref{l: sum of scalar A_p weights} we obtain $\La u\Ra\ci{E}\La u^{-1/(p-1)}\Ra\ci{E}^{p-1}\lesssim_{p,d} C\ci{E,W,p}$, concluding the proof.
\end{proof}

Next lemma states a well known property of reducing operators
(see for instance \cite[Proposition~1.1]{goldberg} and the remark after Proposition~1.2 therein).

\begin{lm}
\label{l: replace inverse by prime}
There holds
\begin{equation*}
|\cW\ci{E}^{-1} e|\leq |\cW\ci{E}'e|\lesssim_{p,d}C\ci{E,W,p}^{1/p}|\cW\ci{E}^{-1}e|,\qquad\forall e\in\C^d.
\end{equation*}
Moreover
\begin{equation*}
|(\cW\ci{E}')^{-1} e|\leq |\cW\ci{E}e|,\qquad\forall e\in\C^d.
\end{equation*}
\end{lm}

%

Notice that Lemma \ref{l: replace inverse by prime} implies in particular that
\begin{equation}
\label{lower bound matrix A_p characteristic}
|\cW'\ci{E}\cW\ci{E}|\geq 1.
\end{equation}

We refer to \cite{goldberg}, \cite{isralowitz-commutators}, \cite{isralowitz-kwon-pott}, \cite{isralowitz-pott-treil} and the references therein for additional properties of reducing matrices.

\subsubsection{``Iterating'' reducing operators}
\label{s: iterate reducing operators}

Let $E,F$ be measurable subsets of $\R^n,\R^m$ respectively with $0<|E|,|F|<\infty$. Let $1<p<\infty$. Let $W$ be a $M_{d}(\C)$-valued integrable function on $E\times F$ taking a.e.~values in the set of positive-definite $d\times d$-matrices. For all $x_1\in E$, set $W_{x_1}(x_2):=W(x_1,x_2)$, $x_2\in F$. For a.e.~$x_1\in E$, denote by $\cW\ci{x_1,F}$ the reducing operator of $W_{x_1}$ over $F$ with respect to the exponent $p$. If $p=2$, then
\begin{equation*}
\cW\ci{x_1,F}:=\La W(x_1,\fdot)\Ra\ci{F}^{1/2}.
\end{equation*}
and therefore the function $\cW\ci{x_1,F}$ is clearly measurable in $x_1$. In fact, it is not hard to see that for any $1<p<\infty$, one can choose the reducing operator $\cW\ci{x_1,F}$ in a way that is measurable in $x_1$. We supply the details in Section \ref{s:MeasurableReducingOp} in the appendix.

Set $W\ci{F}(x_1):=\cW\ci{x_1,F}^{p}$, for a.e.~$x_1\in E$. Then $W\ci{F}\in L^1(E;M_{d}(\C))$, since
\begin{equation*}
\strokedint_{E}|W\ci{F}(x_1)^{1/p}e|^{p}\mathd x_1 = \strokedint_{E}|\cW\ci{x_1,F}e|^{p}\mathd x_1\sim_{p,d}
\strokedint_{E\times F}|W(x_1,x_2)^{1/p}e|^{p}\mathd x<\infty,
\end{equation*}
for all $e\in\C^d$. In particular, this computation shows that
\begin{equation}
\label{iterate reducing operators}
|\cW\ci{F,E}e|\sim_{p,d}|\cW\ci{E\times F}e|,\qquad\forall e\in\C^d,
\end{equation}
where $\cW\ci{F,E}$ is the reducing operator of $W\ci{F}$ over $E$ with respect to the exponent $p$, and $\cW\ci{E\times F}$ is the reducing operator of $W$ over $E\times F$ with respect to the exponent $p$.

Notice now that $W'\ci{F}(x_1):=W\ci{F}(x_1)^{-1/(p-1)}=\cW\ci{x_1,F}^{-p'}$, for a.e.~$x_1\in E$. Thus, for all $e\in\C^d$, applying Lemma \ref{l: replace inverse by prime} in the second step below we have
\begin{align*}
\strokedint_{E}|W'\ci{F}(x_1)^{1/p'}e|^{p'}\mathd x_1
&=\strokedint_{E}|\cW\ci{x_1,F}^{-1}e|^{p'}\mathd x_1
\leq \strokedint_{E}|\cW\ci{x_1,F}'e|^{p'}\mathd x_1\\
&\sim_{p,d} \strokedint_{E\times F}|W'(x_1,x_2)^{1/p'}e|^{p'}\mathd x<\infty,
\end{align*}
where $W_{x_1}':=W_{x_1}^{-1/(p-1)}$, and $\cW\ci{x_1,F}'$ is the reducing operator of $W_{x_1}'$ over $F$ with respect to the exponent $p'$, for a.e.~$x_1\in E$. Thus, $W'\ci{F}\in L^1(E;M_{d}(\C))$, and in particular the last computation shows that
\begin{equation}
\label{iterate reducing operators dual}
|\cW'\ci{F,E}e|\lesssim_{p,d}|\cW'\ci{E\times F}e|,\qquad\forall e\in\C^d,
\end{equation}
where $\cW\ci{F,E}'$ is the reducing operator of $W\ci{F}'$ over $E$ with respect to the exponent $p'$, and $\cW\ci{E\times F}'$ is the reducing operator of $W':=W^{-1/(p-1)}$ over $E\times F$ with respect to the exponent $p'$.

Combining \eqref{iterate reducing operators} and \eqref{iterate reducing operators dual} with \eqref{comparability_from_vector_to_matrix}, we deduce
\begin{equation}
\label{uniform domination of characteristics of averages-local}
|\cW'\ci{F,E}\cW\ci{F,E}|\lesssim_{p,d}|\cW'\ci{E\times F}\cW\ci{E\times F}|.
\end{equation}

Of course, analogous facts hold if one ``iterates" the operation of taking reducing operators in the reverse order.
These observations will be crucial for handling the ``mixed" operators of Section \ref{s: non-cancellative}.

\subsection{Matrix \texorpdfstring{$A_p$}{Ap} weights}

Let $W$ be a $d\times d$-matrix valued weight on $\R^n$. We say that $W$ is a (one-parameter) \emph{$d\times d$-matrix valued $A_p$ weight} if
\begin{equation*}
[W]\ci{A_p(\R^n)}:=\sup_{Q}\strokedint_{Q}\left(\strokedint_{Q}|W(x)^{1/p}W(y)^{-1/p}|^{p'}\mathd y \right)^{p/p'}\mathd x<\infty,
\end{equation*}
where the supremum is taken over all cubes $Q$ in $\R^{n}$, and we call $[W]\ci{A_p(\R^n)}$ the $A_p$ characteristic of $W$ (we omit the dependence on $n$ as it will be clear from the context).
Lemma \ref{l: averaging operators} yields that if $[W]\ci{A_p(\R^n)}$ is finite, then $W':=W^{-1/(p-1)}$ is a $d\times d$-matrix valued $A_{p'}$ weight on $\R^n$ with $[W']\ci{A_{p'}(\R^n)}^{1/p'}\sim_{p,d}[W]\ci{A_p(\R^n)}^{1/p}$, and
\begin{equation*}
[W]\ci{A_p(\R^n)}\sim_{p,d}\sup_{Q}|\cW'\ci{Q}\cW\ci{Q}|^{p},
\end{equation*}
where the reducing matrices for $W$ correspond to exponent $p$, and those for $W'$ correspond to exponent $p'$.

Analogously, consider a $d\times d$-matrix valued weight $W$ on $\R^n\times\R^m$.
We say that $W$ is a \emph{biparameter $d\times d$-matrix valued $A_p$ weight} if
\begin{equation*}
[W]\ci{A_p(\R^n\times\R^m)}:=\sup_{R}\strokedint_{R}\left(\strokedint_{R}|W(x)^{1/p}W(y)^{-1/p}|^{p'}\mathd y \right)^{p/p'}\mathd x<\infty,
\end{equation*}
where the supremum is taken over all rectangles $R$ in $\R^{n}\times\R^{m}$ (with sides parallel to the coordinate axes).
We call the quantity $[W]\ci{A_p(\R^n\times\R^m)}$ the biparameter $A_p$ characteristic of $W$ and, if it is finite, then $W':=W^{-1/(p-1)}$ is a $d\times d$-matrix valued biparameter $A_{p'}$ weight on $\R^n\times\R^m$ with $[W']\ci{A_{p'}(\R^n\times\R^m)}^{1/p'}\sim_{p,d}[W]\ci{A_p(\R^n\times\R^m)}^{1/p}$, and
\begin{equation*}
[W]\ci{A_p(\R^n\times\R^m)}\sim_{p,d}\sup_{R}|\cW'\ci{R}\cW\ci{R}|^{p},
\end{equation*}
where the reducing matrices for $W$ correspond to exponent $p$, and those for $W'$ correspond to exponent $p'$.

More generally, we define two-weight characteristics as follows.
If $V,W$ are $d\times d$-matrix valued weights on $\R^n,$ we define
\begin{equation*}
[W,V]\ci{A_p(\R^n)}:=\sup_{Q}\strokedint_{Q}\left(\strokedint_{Q}|W(x)^{1/p}V(y)^{1/p'}|^{p'}\mathd y \right)^{p/p'}\mathd x,
\end{equation*}
and if they are $d\times d$-matrix valued weights on $\R^n \times \R^m,$ then we define
\begin{equation}
\label{definition two weight A_p}
[W,V]\ci{A_p(\R^n\times\R^m)}:=\sup_{R}\strokedint_{R}\left(\strokedint_{R}|W(x)^{1/p}V(y)^{1/p'}|^{p'}\mathd y \right)^{p/p'}\mathd x.
\end{equation}
Observe that by the observation immediately after Lemma~\ref{l: averaging operators} we have
\begin{equation}
\label{one-parameter two-weight A_p reducing}
[W,V]\ci{A_p(\R^n)}\sim_{p,d}\sup_{Q}|\cV\ci{Q}\cW\ci{Q}|^{p}
\end{equation}
for one-parameter $d \times d$-matrix valued weights
where the reducing matrices for $W$ correspond to exponent $p$ and those for $V$ correspond to exponent $p'$,
and similarly
\begin{equation}
\label{biparameter two-weight A_p reducing}
[W,V]\ci{A_p(\R^n\times\R^m)}\sim_{p,d}\sup_{R}|\cV\ci{R}\cW\ci{R}|^{p}
\end{equation}
for biparameter $d \times d$-matrix valued weights using the same convention for reducing operators.
Notice that \eqref{one-parameter two-weight A_p reducing} implies that
\begin{equation*}
[W,V]\ci{A_p(\R^n)}^{1/p}\sim_{p,d}[V,W]\ci{A_{p'}(\R^n)}^{1/p'}
\end{equation*}
in the one-parameter case and
\begin{equation*}
[W,V]\ci{A_p(\R^n\times\R^m)}^{1/p}\sim_{p,d}[V,W]\ci{A_{p'}(\R^n\times\R^m)}^{1/p'}
\end{equation*}
for biparameter weights.

If $\cD$ is any dyadic grid in $\R^n$, we define the dyadic one-parameter $A_p$ characteristics
\begin{equation*}
[W]\ci{A_p,\cD}:=\sup_{Q\in\cD}\strokedint_{Q}\left(\strokedint_{Q}|W(x)^{1/p}W(y)^{-1/p}|^{p'}\mathd y \right)^{p/p'}\mathd x
\end{equation*}
and
\begin{equation*}
[W,V]\ci{A_p,\cD}:=\sup_{Q\in\cD}\strokedint_{Q}\left(\strokedint_{Q}|W(x)^{1/p}V(y)^{1/p'}|^{p'}\mathd y \right)^{p/p'}\mathd x.
\end{equation*}
We say that $W$ is a one-parameter \emph{$d\times d$-matrix valued $\cD$-dyadic $A_p$ weight} if $[W]\ci{A_p,\cD}<\infty$.
For a product dyadic grid in $\R^n \times \R^m,$ we define the dyadic biparameter $A_p$ characteristics
\begin{equation*}
[W]\ci{A_p,\bfD}:=\sup_{R\in\bfD}\strokedint_{R}\left(\strokedint_{R}|W(x)^{1/p}W(y)^{-1/p}|^{p'}\mathd y \right)^{p/p'}\mathd x
\end{equation*}
and
\begin{equation}
\label{definition two weight dyadic A_p}
[W,V]\ci{A_p,\bfD}:=\sup_{R\in\bfD}\strokedint_{R}\left(\strokedint_{R}|W(x)^{1/p}V(y)^{1/p'}|^{p'}\mathd y \right)^{p/p'}\mathd x.
\end{equation}
In this case, we say that $W$ is a biparameter \emph{$d\times d$-matrix valued $\bfD$-dyadic biparameter $A_p$ weight} if $[W]\ci{A_p,\bfD}<\infty$.

\subsubsection{Matrix biparameter \texorpdfstring{$A_p$}{Ap} implies uniform matrix \texorpdfstring{$A_p$}{Ap} in each coordinate}

\begin{lm}
\label{l: two-weight-biparameter A_p implies uniform A_p in each coordinate}
Let $W,V$ be $d\times d$-matrix valued weights on $\R^{n+m}$.
\begin{entry}

\item[(1)] Let $\bfD=\cD^1\times\cD^2$ be any product dyadic grid in $\R^n\times\R^m$. Then, there holds
\begin{equation*}
[W(\fdot,x_2),V(\fdot,x_2)]\ci{A_p,\cD^1}\lesssim_{d,p}[W,V]\ci{A_p,\bfD},\qquad\text{for a.e.~} x_2\in \R^m,
\end{equation*}
\begin{equation*}
[W(x_1,\fdot),V(x_1,\fdot)]\ci{A_p,\cD^2}\lesssim_{d,p}[W,V]\ci{A_p,\bfD},\qquad\text{for a.e.~} x_1\in \R^n.
\end{equation*}

\item[(2)] There holds
\begin{equation*}
[W(\fdot,x_2),V(\fdot,x_2)]\ci{A_p(\R^n)}\lesssim_{d,p,n}[W,V]\ci{A_p(\R^n\times\R^m)},\qquad\text{for a.e.~} x_2\in \R^m,
\end{equation*}
\begin{equation*}
[W(x_1,\fdot),V(x_1,\fdot)]\ci{A_p(\R^m)}\lesssim_{d,p,m}[W,V]\ci{A_p(\R^n\times\R^m)},\qquad\text{for a.e.~} x_1\in \R^n.
\end{equation*}

\end{entry}
\end{lm}

\begin{proof}
(1) We identify $\Q^{2d}$ with the subset $\Q^{d}+i\Q^{d}$ of $\C^{d}$ throughout the proof. Denote by $\lbrace e_1,\ldots,e_d\rbrace$ the standard basis for $\C^d$.

Assume that $[W,V]\ci{A_p,\bfD}<\infty$. We only show that
\begin{equation*}
[W(\fdot,x_2),V(\fdot,x_2)]\ci{A_p,\cD^1}\lesssim_{p,d}[W,V]\ci{A_p,\bfD},\qquad\text{for a.e.~} x_2\in \R^m,
\end{equation*}
the proof of the other estimate being symmetric.

Since $W,V$ are locally integrable over $\R^{n}\times\R^m$ and $\R^m,\R^n$ can be covered by countably many balls, we deduce that there exists a Lebesgue measurable subset $A$ of $\R^m$ with $|\R^m\setminus A|=0$ such that
\begin{equation*}
W_{x_2}:=W(\fdot, x_2),~V_{x_2}:=V(\fdot,x_2)
\end{equation*}
are locally integrable over $\R^n$, for all $x_2\in A$.

For each $Q\in\cD^1$ and for each $e\in\Q^{2d}$, we have that the function
\begin{equation*}
F\ci{Q,e}(x_2):=\int_{Q}|W(x_1,x_2)^{1/p}e|^{p}\mathd x_1,\qquad x_2\in A
\end{equation*}
is locally integrable over $\R^m$; let $A\ci{Q,e} \subseteq \R^m$ be the set of all its Lebesgue points. Moreover, for each $Q\in\cD^1$ and for each $e\in\Q^{2d}$, we have that the function
\begin{equation*}
G\ci{Q,e}(x_2):=\int_{Q}|V(x_1,x_2)^{1/p'}e|^{p'}\mathd x_1,\qquad x_2\in A
\end{equation*}
is locally integrable over $\R^m$; let $B\ci{Q,e} \subseteq \R^m$ be the set of all its Lebesgue points. Set
\begin{equation*}
C:=A\cap\bigg(\bigcap_{\substack{Q\in\cD^1\\e\in\Q^{2d}}}A\ci{Q,k}\bigg)\cap\bigg(\bigcap_{\substack{Q\in\cD^1\\e\in\Q^{2d}}}B\ci{Q,k}\bigg).
\end{equation*}
Then $C$ is a Lebesgue measurable subset of $\R^m$ with $|\R^m\setminus C|=0$. We will now show that
\begin{equation*}
[W(\fdot,x_2),V(\fdot,x_2)]\ci{A_p,\cD^1}\lesssim_{p,d}[W,V]\ci{A_p,\bfD},\qquad\forall x_2\in C.
\end{equation*}
Let $I\in\cD^1$ and $x_2\in C$ be arbitrary. Pick a sequence $(J_\ell)^{\infty}_{\ell=1}$ of cubes in $\cD^2$ shrinking to $x_2$. Then, for all $e\in\Q^{2d}$ we have
\begin{equation*}
\lim_{\ell\rightarrow\infty}\strokedint_{J_\ell}\left(\strokedint_{I}|W(x_1,y_2)^{1/p}e|^{p}\mathd x_1\right)\mathd y_2=\strokedint_{I}|W(x_1,x_2)^{1/p}e|^{p}\mathd x_1\sim_{p,d}|\cW\ci{x_2,I}e|^{p},
\end{equation*}
where $\cW\ci{x_2,I}$ is the reducing matrix over $I$ of $W_{x_2}$ with respect to exponent $p$, and
\begin{equation*}
\strokedint_{J_\ell}\left(\strokedint_{I}|W(x_1,y_2)^{1/p}e|^{p}\mathd x_1\right)\mathd y_2\sim_{p,d}|\cW\ci{I\times J_\ell}e|^{p},\qquad\forall \ell=1,2,\ldots,
\end{equation*}
where $\cW\ci{I\times J_{\ell}}$ is the reducing matrix of $W$ over $I\times J_{\ell}$ with respect to the exponent $p$, for all $\ell=1,2,\ldots$. In particular, for each $k=1,\ldots,d$ there exists $0<c_k<\infty$ such that
\begin{equation*}
|\cW\ci{I\times J_{\ell}}e_k|\leq c_k,\qquad\forall \ell=1,2,\ldots.
\end{equation*}
Then, by \eqref{equivalence_matrix_norm_columns} we have
\begin{align*}
|\cW\ci{I\times J_{\ell}}|\sim_{d}\sum_{k=1}^{d}|\cW\ci{I\times J_{\ell}}e_k|\leq\sum_{k=1}^{d}c_k<\infty,\qquad\forall \ell=1,2,\ldots.
\end{align*}
Thus, the sequence of matrices $(\cW\ci{I\times J_\ell})_{\ell=1}^{\infty}$ is bounded in matrix-norm, therefore by compactness we can extract a subsequence $(\cW\ci{I\times J_{\ell_k}})^{\infty}_{k=1}$ converging in matrix-norm to some self-adjoint matrix $\cW\ci{I,2}\in M_{d}(\C)$. Then, by density of $\Q^{2d}$ in $\C^{d}$ we deduce
\begin{equation*}
|\cW\ci{I,2}e|\sim_{p,d}|\cW\ci{x_2,I}e|,\qquad\forall e\in\C^d.
\end{equation*}
Working similarly we can extract a further subsequence $(\cV\ci{I\times J_{\ell_{t_{k}}}})^{\infty}_{k=1}$ converging in matrix-norm to some self-adjoint matrix $\cV\ci{I,2}\in M_{d}(\C)$ for which
\begin{equation*}
|\cV\ci{I,2}e|\sim_{p,d}|\cV\ci{x_2,I}e|,\qquad\forall e\in\C^d,
\end{equation*}
where $\cV\ci{I\times J_{\ell}}$ is the reducing matrix of $V$ over $I\times J_{\ell}$ with respect to the exponent $p'$, for all $\ell=1,2,\ldots$, and $\cV\ci{x_2,I}$ is the reducing matrix over $I$ of the weight $V_{x_2}$ with respect to exponent $p'$.

It follows that
\begin{align*}
|\cV\ci{x_2,I}\cW\ci{x_2,I}|\sim_{p,d}|\cV\ci{I,2}\cW\ci{I,2}|=\lim_{k\rightarrow\infty}|\cV\ci{I\times J_{\ell_{t_{k}}}}\cW\ci{I\times J_{\ell_{t_{k}}}}|\lesssim_{p,d}[W,V]\ci{A_p,\bfD}^{1/p}, 
\end{align*}
where in the first step we used \eqref{comparability_from_vector_to_matrix}, yielding the desired result.

(2) This follows by using the well-known $1/3$-trick (see for instance \cite{lerner-nazarov}) coupled with part (1) and the definitions \eqref{definition two weight A_p}, \eqref{definition two weight dyadic A_p} of two-matrix weight $A_p$ characteristics. 
\end{proof}

In the special case $V:=W'$, a converse to Lemma \ref{l: two-weight-biparameter A_p implies uniform A_p in each coordinate} is true as well.

\begin{lm}
\label{l: uniform A_p in each coordinate implies biparameter A_p}
Let $W$ be any $d\times d$-matrix valued weight on $\R^{n+m}$. Then, there holds
\begin{equation*}
[W]\ci{A_p(\R^n\times\R^m)}\lesssim_{d,p}(\esssup_{x_2\in\R^m}[W(\fdot,x_2)]\ci{A_p(\R^n)})\cdot(\esssup_{x_1\in\R^n}[W(x_1,\fdot)]\ci{A_p(\R^m)}).
\end{equation*}
The dyadic version of this is true as well.
\end{lm}

\begin{proof}
It follows immediately by iterating Lemma \ref{l: averaging operators}, after noticing that if $Q=Q_1\times Q_2$ is any rectangle in $\R^n\times\R^m$, then
\begin{equation*}
A\ci{Q}f(x_1,x_2)=A\ci{Q_1}(E\ci{Q_2}(\fdot,x_2))(x_1),\qquad E\ci{Q_2}(y_1,x_2):=A\ci{Q_2}(f(y_1,\fdot))(x_2),
\end{equation*}
where $A\ci{Q},A\ci{Q_1},A\ci{Q_2}$ denote the averaging operators over $Q,Q_1,Q_2$ respectively.
\end{proof}

\subsubsection{\texorpdfstring{$A_p$}{Ap} characteristics of ``sliced" reducing operators}

Let $1<p<\infty$. Let $W$ be a biparameter matrix $A_p$ weight on $\R^n\times\R^m$. For a.e. $x_1\in\R^n$, set $W\ci{x_1}(x_2):=W(x_1,x_2)$, $x_2\in\R^m$. Fix any cube $Q$ in $\R^m$. For a.e.~$x_1\in\R^n$, let $\cW\ci{x_1,Q}$ be the reducing operator of $W\ci{x_1}$ over $Q$ with respect to the exponent $p$. Set $W\ci{Q}(x_1):=\cW\ci{x_1,Q}^{p}$, for a.e.~$x_1\in\R^n$. Then, by \eqref{uniform domination of characteristics of averages-local} we deduce
\begin{equation}
\label{uniform domination of characteristics of averages}
[W\ci{Q}]\ci{A_p(\R^n)}\lesssim_{p,d}[W]\ci{A_p(\R^n\times\R^m)}.
\end{equation}
Of course, the dyadic version of this is also true. Moreover, both versions remain true if one ``slices" with respect to the first coordinate instead of the second one.

\section{Dyadic strong Christ--Goldberg maximal function}
\label{s: maximal function}

In this section we briefly study matrix-weighted bounds for biparameter maximal functions.
These and their vector-valued counterparts (see Subsection \ref{s: vector-valued maximal function} below) are crucial for estimating the non-cancellative terms in Martikainen's representation \eqref{martikainen representation} (see Section \ref{s: non-cancellative}).
What follows is based on previous works \cite{Michael Christ and Goldberg} and \cite{goldberg}.

\subsection{Dyadic Christ--Goldberg maximal function}

Let $\cD$ be any dyadic grid in $\R^n$. Let $1<p<\infty$, and let $W$ be a $d\times d$-matrix valued weight on $\R^n$. Consider the \emph{(dyadic) Christ--Goldberg maximal function} corresponding to the weight $W$ (and the exponent $p$)
\begin{equation*}
M\ci{\cD,W}f(x):=\sup_{I\in\cD}\La|W(x)^{1/p}f|\Ra\ci{I}\1\ci{I}(x),\qquad x\in\R^n,~f\in L^1\ti{loc}(\R^n;\C^d).
\end{equation*}
Consider also the \emph{(dyadic) modified Christ--Goldberg maximal function}
\begin{equation*}
\wt{M}\ci{\cD,W}f(x):=\sup_{I\in\cD}\La|\cW\ci{I}f|\Ra\ci{I}\1\ci{I}(x),\qquad x\in\R^n,~f\in L^1\ti{loc}(\R^n;\C^d),
\end{equation*}
where the reducing operators for $W$ correspond to the exponent $p$. Both of these operators were defined by Christ and Goldberg for $p=2$ in \cite{Michael Christ and Goldberg}, and by Goldberg for general $1<p<\infty$ in \cite{goldberg}.
In these two works it is proved that if $[W]\ci{A_p,\cD}<\infty$, then
\begin{equation*}
\Vert M\ci{\cD,W}\Vert\ci{L^{p}(W)\rightarrow L^{p}},\Vert \wt{M}\ci{\cD,W}\Vert\ci{L^{p}(W)\rightarrow L^{p}}<\infty.
\end{equation*}
In fact, J.~Isralowitz and K.~Moen \cite{isralowitz-moen} proved the bound
\begin{equation}
\label{weighted bound C-G}
\Vert M\ci{\cD,W}\Vert\ci{L^{p}(W)\rightarrow L^{p}}\lesssim_{p,d,n}[W]\ci{A_p,\cD}^{1/(p-1)},
\end{equation}
and Isralowitz, H.~K.~Kwon and Pott \cite{isralowitz-kwon-pott} showed the bound
\begin{equation}
\label{weighted bound modified C-G}
\Vert\wt{M}\ci{\cD,W}f\Vert\ci{L^{p}(W)\rightarrow L^{p}}\lesssim_{n,p,d}[W]\ci{A_p,\cD}^{1/(p-1)}.
\end{equation}

\subsection{Modified dyadic strong Christ--Goldberg maximal function}

Let $\bfD$ be any product dyadic grid in $\R^n\times\R^m$. Let $1<p<\infty$, and let $W$ be a $d\times d$-matrix valued weight on $\R^n\times\R^m$. Consider the \emph{modified dyadic strong Christ--Goldberg maximal function}
\begin{equation*}
\wt{M}\ci{\bfD,W}f(x):=\sup_{R\in\bfD}\La|\cW\ci{R}f|\Ra\ci{R}\1\ci{R}(x),\qquad x\in\R^{n+m},~f\in L^1\ti{loc}(\R^{n+m};\C^d),
\end{equation*}
where reducing operators for $W$ correspond to the exponent $p$.

\begin{prop}
\label{p: weighted bound modified strong C-G}
There holds
\begin{equation*}
\Vert\wt{M}\ci{\bfD,W}f\Vert\ci{L^{p}(W)\rightarrow L^{p}}\lesssim_{p,d,n,m}[W]\ci{A_p,\bfD}^{\frac{p+1}{p(p-1)}}.
\end{equation*}
\end{prop}

\begin{proof}
The proof is an adaptation of the one-parameter case due to Isralowitz--Kwon--Pott \cite{isralowitz-kwon-pott}.
We only outline the argument.

Lemma \ref{l: first scalar} applied to $W'$ with exponent $p'$ and matrix $A:=\cW\ci{R}$ gives
\begin{equation*}
[|W^{-1/p}\cW\ci{R}|^{p'}]\ci{A_{p'},\bfD}\leq C(p,d) [W]\ci{A_p,\bfD}^{p'/p},\qquad\forall R\in\bfD.
\end{equation*}
Hence, the function $|W^{-1/p}\cW\ci{R}|^{p'}$ is a scalar $\bfD$-dyadic biparameter $A_{p'}$ weight and it satisfies the reverse H\ddoto lder inequality \cite[Proposition 2.2]{holmes-petermichl-wick} (for the reader's convenience a full proof is included in the appendix), for all $R\in\bfD$.
Taking $\e \lesssim_{n,m,p,d} [W]\ci{A_p,\bfD}^{-p'/p}$
we have
\begin{equation*}
\left(\strokedint_{R}|\cW\ci{R}W^{-1/p}(x)|^{p'(1+\e)}\mathd x\right)^{1/(p'(1+\e))}
\lesssim_{p,n,m}\left(\strokedint_{R}|W^{-1/p}(x)\cW\ci{R}|^{p'}\mathd x\right)^{1/p'}
\lesssim_{p,d}[W]\ci{A_p,\bfD}^{1/p},
\end{equation*}
for all $R\in\bfD$.
It follows by H\ddoto lder's inequality that
\begin{align*}
&\La|\cW\ci{R}f|\Ra\ci{R}\\
&\leq\left(\strokedint_{R}|\cW\ci{R}W^{-1/p}(x)|^{p'(1+\e)}\mathd x\right)^{1/(p'(1+\e))}\left(\strokedint_{R}|W^{1/p}(x)f(x)|^{p(1+\e)/(1+p\e)}\mathd x\right)^{(1+p\e)/(p(1+\e))}\\
&\lesssim_{p,d,n,m}[W]\ci{A_p,\bfD}^{1/p}M\ci{a,\bfD}(|W^{1/p}f|)(x),
\end{align*}
for all $x\in R$ and for all $R\in\bfD$, where $a:=\frac{p(1+\e)}{1+p\e}\in(1,p)$ and
\begin{equation*}
M\ci{a,\bfD}g:=\sup_{R\in\bfD}\La|g|^{a}\Ra\ci{R}^{1/a}\1\ci{R}.
\end{equation*}
Thus
\begin{equation}
\label{dominate weighted maximal function by unweighted}
\wt{M}\ci{\bfD,W}f\lesssim_{p,d,n,m}[W]^{1/p}\ci{A_p,\bfD}M\ci{a,\bfD}(|W^{1/p}f|).
\end{equation}
Notice that~\eqref{lower bound matrix A_p characteristic} implies $\e\leq c(m,n,p,d)$ for some finite positive constant $c(m,n,p,d)$ depending only on $n,m,p,d$, but independent of the weight and its $A_p$ characteristic. By iterating the classical bounds for the dyadic Hardy--Littlewood maximal function, we get that
\begin{equation*}
\Vert M\ci{\bfD}\Vert\ci{L^{q}\rightarrow L^{q}}\leq 4(q')^{2/q},\qquad\forall 1<q<\infty,
\end{equation*}
where
\begin{equation*}
M\ci{\bfD}g:=\sup_{R\in\bfD}\La|g|\Ra\ci{R}\1\ci{R}.
\end{equation*}
Therefore, we have
\begin{align*}
\Vert M\ci{a,\bfD}g\Vert\ci{L^{p}}^{p}
&=\Vert M\ci{\bfD}(|g|^{a})\Vert\ci{L^{p/a}}^{p/a}\leq 4^{p/a} ((p/a)')^{2}\Vert |g|^{a}\Vert\ci{L^{p/a}}^{p/a}\\
&\lesssim_{p,d} \e^{-2}\Vert |g|^{a}\Vert\ci{L^{p/a}}^{p/a}\sim_{p,d,n,m}[W]\ci{A_p,\bfD}^{2/(p-1)}\Vert g\Vert\ci{L^{p}}^{p},
\end{align*}
where we used that $\frac{p}{a}<p$ and that
\begin{equation*}
(p/a)'\leq\frac{1+p\cdot c(m,n,p,d)}{p-1}\e^{-1}.
\end{equation*}
Finally, since $\Vert |W^{1/p}f|\Vert\ci{L^{p}}=\Vert f\Vert\ci{L^{p}(W)}$, the conclusion follows.
\end{proof}

\section{Biparameter matrix-weighted square functions}
\label{s: square functions}

In this section we obtain matrix-weighted bounds for square functions. These and their counterparts for shifted square functions (see Subsection \ref{s: shifted square functions}) are crucial for estimating the cancellative terms in Martikainen's representation \eqref{martikainen representation} (see Section \ref{s: cancellative}).

\subsection{Matrix-weighted vector-valued extensions for linear operators}

\begin{lm}
\label{l: matrix weight vector valued extensions}
Let $W$ be a $d\times d$-matrix weight on $\R^n$. Let $T:L^{p}(W)\rightarrow L^{p}(\R^n;H)$ be a bounded linear operator, where $H$ is any Hilbert space. Then
\begin{equation*}
\bigg\Vert\bigg(\sum_{k=1}^{\infty}\Vert T(f_k)\Vert^2\bigg)^{1/2}\bigg\Vert\ci{L^{p}}\lesssim_{p}\Vert T\Vert\cdot\bigg\Vert \bigg(\sum_{k=1}^{\infty}|W^{1/p}f_k|^2\bigg)^{1/2}\bigg\Vert\ci{L^{p}},
\end{equation*}
for any sequence $(f_k)^{\infty}_{k=1}$ in $L^{p}(W)$, where the implied constant does not depend on the Hilbert space $H$ or $d$.
\end{lm}

\begin{proof}
First of all, by the Monotone Convergence Theorem it suffices to prove only that
\begin{equation*}
\bigg\Vert\bigg(\sum_{k=1}^{N}\Vert T(f_k)\Vert^2\bigg)^{1/2}\bigg\Vert\ci{L^{p}}\lesssim_{p}\Vert T\Vert\cdot\bigg\Vert \bigg(\sum_{k=1}^{N}|W^{1/p}f_k|^2\bigg)^{1/2}\bigg\Vert\ci{L^{p}},\qquad\forall N=1,2,\ldots,
\end{equation*}
provided the implied constant does not depend on $N$.

We fix now a positive integer $N$. Then, using Khintchine's inequalities for Hilbert spaces (see e.g. \cite{hytonen-analysis-banach-spaces}) we have
\begin{align*}
&\bigg\Vert\bigg(\sum_{k=1}^{N}\Vert T(f_k)\Vert^2\bigg)^{1/2}\bigg\Vert\ci{L^{p}}^{p}
\sim_{p}\bigg\Vert\bigg(\int_{\Omega}\bigg\Vert\sum_{k=1}^{N}\e_k(\omega)T(f_k)\bigg\Vert^{p}\mathd\bP(\omega)\bigg)^{1/p}\bigg\Vert\ci{L^{p}}^{p}\\
&=\int_{\Omega}\bigg\Vert T\bigg(\sum_{k=1}^{N}\e_k(\omega)f_k\bigg)\bigg\Vert\ci{L^{p}(\R^n;H)}^{p}\mathd\bP(\omega)\leq\Vert T\Vert^p\cdot\int_{\Omega}\bigg\Vert \sum_{k=1}^{N}\e_k(\omega)f_k\bigg\Vert\ci{L^{p}(W)}^{p}\mathd\bP(\omega)\\
&=\Vert T\Vert^{p}\cdot\int_{\R^n}\int_{\Omega}\bigg|\sum_{k=1}^{N}\e_k(\omega)W(x)^{1/p}f_k(x)\bigg|^{p}\mathd\bP(\omega)\mathd x\\
&\sim_{p}\Vert T\Vert^{p}\cdot\int_{\R^n}\bigg(\sum_{k=1}^{N}|W(x)^{1/p}f_k(x)|^{2}\bigg)^{p/2}\mathd x,
\end{align*}
concluding the proof.
\end{proof}

\subsection{One-parameter matrix-weighted square functions}

Fix any dyadic grid $\cD$ in $\R^n$. We define the standard (unweighted) square function
\begin{equation*}
S\ci{\cD}f(x):=\bigg(\sum_{\substack{Q\in\cD\\\e\in\cE}}|f\ci{Q}^{\e}|^2\frac{\1\ci{Q}(x)}{|Q|}\bigg)^{1/2},\qquad f\in L^1\ti{loc}(\R^{n};\C^d),
\end{equation*}
and we recall that standard (unweighted) dyadic Littlewood--Paley theory gives
\begin{equation*}
\Vert S\ci{\cD}f\Vert\ci{L^{p}}\sim_{n,p,d}\Vert f\Vert\ci{L^{p}}.
\end{equation*}

Let $1<p<\infty$, and let $W$ be a $d\times d$-matrix valued weight on $\R^n$. The direct generalization of the square function from the scalar context is
\begin{equation*}
S\ci{\cD,W}f(x):=\bigg(\sum_{\substack{Q\in\cD\\\e\in\cE}}|W(x)^{1/p}f\ci{Q}^{\e}|^2\frac{\1\ci{Q}(x)}{|Q|}\bigg)^{1/2},\qquad x\in\R^n,~f\in L^1\ti{loc}(\R^n;\C^d).
\end{equation*}
However, in the theory of matrix-valued weights it is sometimes more convenient to use the following Triebel--Lizorkin type matrix-weighted square function corresponding to the weight $W$ (and the exponent $p$):
\begin{equation*}
\wt{S}\ci{\cD,W}f(x):=\bigg(\sum_{\substack{Q\in\cD\\\e\in\cE}}|\cW\ci{Q}f\ci{Q}^{\e}|^2\frac{\1\ci{Q}(x)}{|Q|}\bigg)^{1/2},\qquad x\in\R^n,~f\in L^1\ti{loc}(\R^n;\C^d),
\end{equation*}
where $\cW\ci{Q}$ is the reducing operator over $Q$ of the weight $W$ corresponding to the exponent $p$ (so if $p=2$, then just $\cW\ci{Q}=(W)\ci{Q}^{1/2}$). Notice that both definitions take into account the exponent $p$. 

Notice that in the special case $p=2$ we simply have
\begin{equation*}
\Vert\wt{S}\ci{\cD,W}f\Vert\ci{L^2}^2 = \sum_{\substack{Q\in\cD\\\e\in\cE}}|\La W\Ra\ci{Q}^{1/2}f\ci{Q}^{\e}|^2
= \sum_{\substack{Q\in\cD\\\e\in\cE}}\La|W^{1/2}f\ci{Q}^{\e}|^2\Ra\ci{Q}
= \Vert S\ci{\cD,W}f\Vert\ci{L^2}^2.
\end{equation*}

It will be useful in the following to view $S\ci{\cD,W}$ as the vector-valued linear operator $\vec{S}\ci{\cD,W}$ given by
\begin{equation*}
\vec{S}\ci{\cD,W}(f)(x):=\lbrace W^{1/p}(x)f^{\e}\ci{Q}h^{\e}\ci{Q}(x)\rbrace\ci{(Q,\e)\in\cD\times\cE},\qquad x\in\R^n,
\end{equation*}
in the sense that
\begin{equation*}
S\ci{\cD,W}(f)(x)=\Vert \vec{S}\ci{\cD,W}(f)(x)\Vert\ci{\ell^2(\cD\times\cE;\C^d)}.
\end{equation*}
In particular
\begin{equation*}
\Vert S\ci{\cD,W}(f)\Vert\ci{L^p}=\Vert \vec{S}\ci{\cD,W}(f)\Vert\ci{L^{p}(\R^n;\ell^2(\cD\times\cE;\C^d))}.
\end{equation*}

It has been proved in \cite{hytonen-petermichl-volberg}, \cite{treil-nonhomogeneous_square_function} that if $p=2$, then
\begin{equation}
\label{upper bound square functions p=2}
\Vert S\ci{\cD,W}f\Vert\ci{L^2}\lesssim_{n,d}[W]\ci{A_2,\cD}\Vert f\Vert\ci{L^2(W)},\qquad\forall f\in L^1\ti{loc}(\R^n;\C^d),
\end{equation}
and this estimate is sharp (even in the scalar case). For general $1<p<\infty$, it has been proved in \cite{isralowitz-square function} that
\begin{equation}
\label{isralowitz one parameter square function}
\Vert S\ci{\cD,W}f\Vert\ci{L^{p}}\lesssim_{p,d,n}[W]\ci{A_p,\cD}^{\gamma(p)}\Vert f\Vert\ci{L^{p}(W)},\qquad\forall f\in L^1\ti{loc}(\R^n;\C^d),
\end{equation}
where
\begin{equation}
\label{gamma(p)}
\gamma(p):=
\begin{cases}
\frac{1}{p-1},\text{ if }1<p\leq 2\\\\
\frac{1}{2}+\frac{1}{p(p-1)},\text{ if }p>2
\end{cases}
.
\end{equation}
It is noted in \cite{isralowitz-square function} that while this estimate is sharp in the regime $p\leq 2$ even in the scalar case, it is probably not optimal in the regime $p>2$.

\subsection{Biparameter matrix-weighted square functions}

Let $\bfD=\cD^1\times\cD^2$ be any dyadic product grid of rectangles in $\R^n\times\R^m$. We define the standard (unweighted) biparameter square function
\begin{equation*}
S\ci{\bfD}f(x):=\bigg(\sum_{\substack{R\in\bfD\\\bee\in\cE}}|f\ci{R}^{\bee}|^2\frac{\1\ci{R}(x)}{|R|}\bigg)^{1/2},\qquad f\in L^1\ti{loc}(\R^{n+m};\C^d),
\end{equation*}
and we recall that standard (unweighted) dyadic Littlewood--Paley theory gives
\begin{equation*}
\Vert S\ci{\bfD}f\Vert\ci{L^{p}}\sim_{p,d,n,m}\Vert f\Vert\ci{L^{p}},
\end{equation*}
for all (suitable) functions $f,$ which follows by iterating the classical one parameter result
(also check \cite{chang-fefferman} and \cite{tao}).

Let $1<p<\infty$, and let $W$ be a $d\times d$-matrix valued weight on $\R^n\times\R^m$. The direct generalization of the biparameter square function from the scalar context is
\begin{equation*}
S\ci{\bfD,W}f(x):=\bigg(\sum_{\substack{R\in\bfD\\\bee\in\cE}}|W(x)^{1/p}f\ci{R}^{\bee}|^2\frac{\1\ci{R}(x)}{|R|}\bigg)^{1/2},\qquad x\in\R^{n+m},~f\in L^1\ti{loc}(\R^{n+m};\C^d).
\end{equation*}
We also consider the following Triebel--Lizorkin type biparameter matrix-weighted square function corresponding to the weight $W$ (and the exponent $p$):
\begin{equation*}
\wt{S}\ci{\bfD,W}f(x):=\bigg(\sum_{\substack{R\in\bfD\\\bee\in\cE}}|\cW\ci{R}f\ci{R}^{\bee}|^2\frac{\1\ci{R}(x)}{|R|}\bigg)^{1/2},\qquad f\in L^1\ti{loc}(\R^{n+m};\C^d),
\end{equation*}
where reducing operators for $W$ are taken with respect to the exponent $p$. As before, in the special case $p=2$ we simply have $\Vert\wt{S}\ci{\bfD,W}f\Vert\ci{L^2}=\Vert S\ci{\bfD,W}f\Vert\ci{L^2}$.

\begin{lm}
\label{l: upper bound square function}
Let $1<p<\infty$, and let $W$ be a $d\times d$-matrix valued $\bfD$-dyadic biparameter $A_p$ weight on $\R^n\times\R^m$. Then, there holds
\begin{equation*}
\Vert S\ci{\bfD,W}f\Vert\ci{L^p}\lesssim_{d,p,n,m}[W]\ci{A_p,\bfD}^{2\gamma(p)}\Vert f\Vert\ci{L^p(W)},\qquad\forall f\in L^{\infty}\ti{c}(\R^{n+m};\C^d),
\end{equation*}
where $L^{\infty}\ti{c}$ indicates the space of compactly supported $L^{\infty}$ functions.
\end{lm}

\begin{proof}
The proof will rely on Lemma \ref{l: matrix weight vector valued extensions}. We have
\begin{align*}
&\Vert S\ci{\bfD,W}f\Vert\ci{L^p}^{p}=\int_{\R^n}\bigg(\int_{\R^m}\bigg(\sum_{\substack{R_1\in\cD^1\\\e_1\in\cE^1}}\sum_{\substack{R_2\in\cD^2\\\e_2\in\cE^2}}|W(x_1,x_2)^{1/p}f\ci{R_1\times R_2}^{\e_1\e_2}|^2\frac{\1\ci{R_1}(x_1)\1\ci{R_2}(x_2)}{|R_1|\cdot|R_2|}\bigg)^{p/2}\mathd x_2\bigg)\mathd x_1\\
&=\int_{\R^n}\bigg(\int_{\R^m}\bigg(\sum_{\substack{R_1\in\cD^1\\\e_1\in\cE^1}}(S\ci{\cD^2,W(x_1,\fdot)}((Q^{\e_1,1}
\ci{R_1}f)(x_1,\fdot))(x_2))^2\bigg)^{p/2}\mathd x_2\bigg)\mathd x_1\\
&=\int_{\R^n}\bigg(\int_{\R^m}\bigg(\sum_{\substack{R_1\in\cD^1\\\e_1\in\cE^1}}\Vert\vec{S}\ci{\cD^2,W(x_1,\fdot)}((Q^{\e_1,1}
\ci{R_1}f)(x_1,\fdot))(x_2)\Vert\ci{l^2(\cD^2\times\cE^2)}^2\bigg)^{p/2}\mathd x_2\bigg)\mathd x_1,
\end{align*}
where we recall that
\begin{equation*}
Q^{\e_1,1}\ci{R_1}f(y_1,y_2):=Q\ci{R_1}^{\e_1}(f(\fdot,y_2))(y_1)=(f(\fdot,y_2),h^{\e_1}\ci{R_1})h^{\e_1}\ci{R_1}(y_1).
\end{equation*}
We have that $[W(x_1,\fdot)]\ci{A_p,\cD^2}\lesssim_{p,d}[W]\ci{A_p,\bfD}$, for a.e.~$x_1\in\R^n$ (see Lemma \ref{l: uniform A_p in each coordinate implies biparameter A_p}), as well as \eqref{isralowitz one parameter square function}, so in view of Lemma \ref{l: matrix weight vector valued extensions} we obtain that $\Vert S\ci{\bfD,W}f\Vert\ci{L^p}^{p}$ is bounded by
\begin{align*}
&[W]\ci{A_p,\bfD}^{p\gamma(p)}\int_{\R^n}\bigg(\int_{\R^m}\bigg(\sum_{\substack{R_1\in\cD^1\\\e_1\in\cE^1}}|W(x_1,x_2)^{1/p}(Q^{\e_1,1}
\ci{R_1}f)(x_1,x_2)|^2\bigg)^{p/2}\mathd x_2\bigg)\mathd x_1\\
&=[W]\ci{A_p,\bfD}^{p\gamma(p)}\int_{\R^m}\Vert S\ci{\cD^1,W(\fdot,x_2)}(f(\fdot,x_2))\Vert\ci{L^{p}(\R^n)}^p\mathd x_2\\
&\lesssim_{m,d,p}[W]\ci{A_p,\bfD}^{2p\gamma(p)}\int_{\R^m}\Vert f(\fdot,x_2)\Vert\ci{L^{p}(W(\fdot,x_2))}^p\mathd x_2=[W]\ci{A_p,\bfD}^{2p\gamma(p)}\Vert f\Vert\ci{L^{p}(W)}^p,
\end{align*}
where in the last $\lesssim_{p,m,d}$ we used again the one-parameter result from \cite{isralowitz-square function}, concluding the proof.
\end{proof}

\begin{lm}
\label{l: dominate TL square function by square function}
Let $1<p<\infty$, and let $W$ be a $d\times d$-matrix valued $\bfD$-dyadic biparameter $A_p$ weight on $\R^n\times\R^m$. Then, there holds
\begin{equation*}
\Vert\wt{S}\ci{\bfD,W}f\Vert\ci{L^p}\lesssim_{d,p,n,m}[W]\ci{A_p,\bfD}^{1/p} \Vert S\ci{\bfD,W}f\Vert\ci{L^{p}},\qquad\forall f\in L^1\ti{loc}(\R^{n+m};\C^d).
\end{equation*}
\end{lm}

\begin{proof}
By the Monotone Convergence Theorem we can assume that $f$ has only finitely many nonzero Haar coefficients. Now, using Lemma \ref{l: second scalar} (3) we get
\begin{align*}
\Vert\wt{S}\ci{\bfD,W}f\Vert\ci{L^p}^{p}&=\int_{\R^n\times\R^m}\bigg(\sum_{\substack{R\in\bfD\\\bee\in\cE}}|\cW\ci{R}f\ci{R}^{\bee}|^2\frac{\1\ci{R}(x)}{|R|}\bigg)^{p/2}\mathd x\\
&\lesssim_{p,d}[W]\ci{A_p,\bfD}\int_{\R^n\times\R^m}\bigg(\sum_{\substack{R\in\bfD\\\bee\in\cE}}\La|W^{1/p} f\ci{R}^{\bee}|\Ra\ci{R}^2\frac{\1\ci{R}(x)}{|R|}\bigg)^{p/2}\mathd x.
\end{align*}
Thus, by dyadic Littlewood--Paley theory we only need to see that
\begin{align*}
\Vert F\Vert\ci{L^{p}}:=\bigg\Vert \sum_{\substack{R\in\bfD\\\bee\in\cE}}\La|W^{1/p}f\ci{R}^{\bee}|\Ra\ci{R}h\ci{R}^{\bee}\bigg\Vert\ci{L^p}\lesssim_{d,p,n,m} \Vert S\ci{\bfD,W}f\Vert\ci{L^{p}}.
\end{align*}

We use duality for that. Let $g\in L^{p'}(\R^{n+m})$. Then, we have
\begin{align*}
|( F,g)|&\leq \int_{\R^n\times\R^m}\sum_{\substack{R\in\bfD\\\bee\in\cE}}|W^{1/p}(x)f^{\bee}\ci{R}|\cdot h\ci{R}^{\bee}(x)\cdot|g^{\bee}\ci{R}|\cdot h\ci{R}^{\bee}(x)\mathd x\\
&\leq\int_{\R^n\times\R^m}(S\ci{\bfD,W}f)(x)(S\ci{\bfD}g)(x)\mathd x\\
&\leq \Vert S\ci{\bfD,W}f\Vert\ci{L^{p}}\Vert S\ci{\bfD}g\Vert\ci{L^{p'}}
\sim_{p,d,n,m}\Vert S\ci{\bfD,W}f\Vert\ci{L^{p}}\Vert g\Vert\ci{L^{p'}},
\end{align*}
concluding the proof.
\end{proof}

\begin{cor}
\label{c: upper bound TL square function}
Let $1<p<\infty$, and let $W$ be a $d\times d$-matrix valued $\bfD$-dyadic biparameter $A_p$ weight on $\R^n\times\R^m$. Then, there holds
\begin{equation*}
\Vert\wt{S}\ci{\bfD,W}f\Vert\ci{L^p}\lesssim_{d,p,n,m}[W]\ci{A_p,\bfD}^{\frac{1}{p}+2\gamma(p)}\Vert f\Vert\ci{L^p(W)},\qquad\forall f\in L^{\infty}\ti{c}(\R^{n+m};\C^d).
\end{equation*}
\end{cor}

Using the above upper bounds and duality, we can deduce lower bounds.

\begin{cor}
\label{c: lower bounds square functions}
Let $1<p<\infty$, and let $W$ be a $d\times d$-matrix valued $\bfD$-dyadic biparameter $A_p$ weight on $\R^n\times\R^m$. Then, there holds
\begin{equation*}
\Vert f\Vert\ci{L^p(W)}\lesssim_{d,p,n,m}[W]\ci{A_{p},\bfD}^{\frac{2\gamma(p')}{p-1}}\Vert S\ci{\bfD,W}f\Vert\ci{L^{p}},\qquad\forall f\in L^{\infty}\ti{c}(\R^{n+m};\C^d),
\end{equation*}
\begin{equation*}
\Vert f\Vert\ci{L^p(W)}\lesssim_{d,p,n,m}[W]\ci{A_p,\bfD}^{\frac{1}{p}+\frac{2\gamma(p')}{p-1}}\Vert \wt{S}\ci{\bfD,W}f\Vert\ci{L^p},\qquad\forall f\in L^{\infty}\ti{c}(\R^{n+m};\C^d).
\end{equation*}
\end{cor}

\begin{proof}
We use duality. Let $f,g\in L^{\infty}\ti{c}(\R^{n+m};\C^d)$. Then, we have
\begin{align*}
|(f,g)|&=\bigg|\sum_{\substack{R\in\bfD\\\bee\in\cE}}\La f\ci{R}^{\bee},g\ci{R}^{\bee}\Ra\bigg|\leq
\sum_{\substack{R\in\bfD\\\bee\in\cE}}\strokedint_{R}|W(x)^{1/p}f\ci{R}^{\bee}|\cdot|W(x)^{-1/p}g\ci{R}^{\bee}|\mathd x\\
&\leq\int_{\R^{n+m}}(S\ci{\bfD,W}f(x))(S\ci{\bfD,W'}g(x))\mathd x\leq\Vert S\ci{\bfD,W}f\Vert\ci{L^{p}}\Vert S\ci{\bfD,W'}g\Vert\ci{L^{p'}}\\
&\lesssim_{p,d,n,m}\Vert S\ci{\bfD,W}f\Vert\ci{L^{p}}[W]\ci{A_{p'},\bfD}^{2\gamma(p')}\Vert g\Vert\ci{L^{p'}(W')}
\sim_{p,d}\Vert S\ci{\bfD,W}f\Vert\ci{L^{p}}[W]\ci{A_{p},\bfD}^{\frac{2\gamma(p')}{p-1}}\Vert g\Vert\ci{L^{p'}(W')}.
\end{align*}

The proof for $\wt{S}\ci{\bfD,W}$ is similar. It suffices to use that Lemma~\ref{l: replace inverse by prime} implies
\begin{align*}
|\La f\ci{R}^{\bee},g\ci{R}^{\bee}\Ra|\leq|\cW\ci{R}f\ci{R}^{\bee}|\cdot|(\cW\ci{R})^{-1}g\ci{R}^{\bee}|\leq|\cW\ci{R}f\ci{R}^{\bee}|\cdot|\cW'\ci{R}g\ci{R}^{\bee}|
\end{align*}
for all $R\in\bfD$ and for all $\e\in\cE.$
\end{proof}

In Section \ref{s: non-cancellative} we will need the following one-parameter counterpart of Corollary \ref{c: lower bounds square functions} (the proof is identical, using \eqref{isralowitz one parameter square function} instead of Lemma \ref{l: upper bound square function}).

\begin{cor}
\label{c: lower bounds square functions one parameter}
Let $1<p<\infty$. Let $\cD$ be a dyadic grid in $\R^n$, and let $W$ be a $d\times d$-matrix valued $\cD$-dyadic $A_p$ weight on $\R^n$. Then, there holds
\begin{equation*}
\Vert f\Vert\ci{L^p(W)}\lesssim_{d,p,n}[W]\ci{A_{p},\cD}^{\frac{\gamma(p')}{p-1}}\Vert S\ci{\cD,W}f\Vert\ci{L^{p}},\qquad\forall f\in L^{\infty}\ti{c}(\R^{n};\C^d).
\end{equation*}
\end{cor}

\begin{lm}
\label{l: replace weight tl square function}
Let $W,U$ be $d\times d$-matrix valued weights on $\R^{n+m}$. Then, there holds
\begin{equation*}
\wt{S}\ci{\bfD,W}f\lesssim_{d,p}[W,U']\ci{A_p,\bfD}^{1/p}\wt{S}\ci{\bfD,U}f.
\end{equation*}
\end{lm}

\begin{proof}
We have
\begin{align*}
(\wt{S}\ci{\bfD,W}f)^2&=\sum_{R\in\bfD}|\cW\ci{R}f\ci{R}|^2\frac{\1\ci{R}}{|R|}\leq
\sum_{R\in\bfD}|\cW\ci{R}\cU^{-1}\ci{R}|\cdot|\cU\ci{R}f\ci{R}|^{2}\frac{\1\ci{R}}{|R|}\\
&\lesssim_{p,d}\sum_{R\in\bfD}|\cW\ci{R}\cU'\ci{R}|\cdot|\cU\ci{R}f\ci{R}|^{2}\frac{\1\ci{R}}{|R|}
\lesssim_{p,d}[W,U']\ci{A_p,\bfD}^{2/p}(\wt{S}\ci{\bfD,U}f)^2,
\end{align*}
where in the first $\lesssim_{p,d}$ we used Lemma \ref{l: replace inverse by prime} coupled with \eqref{comparability_from_vector_to_matrix}, and in the second $\lesssim_{p,d}$ we used \eqref{biparameter two-weight A_p reducing}. 
\end{proof}

\subsection{\texorpdfstring{$L^{p}$}{Lp} matrix-weighted bounds for biparameter Haar multipliers}
\label{s: matrix-weight Haar}
Let $\bfD$ be any product dyadic grid in $\R\times\R$. Let $\s=(\s\ci{R})\ci{R\in\bfD}$ be a sequence of complex numbers such that $C:=\sup_{R\in\bfD}|\s\ci{R}|<\infty$. Set
\begin{equation*}
T_{\s}f:=\sum_{R\in\bfD}\s\ci{R}f\ci{R}h\ci{R},\qquad f\in L^2(\R^{2};\C^d).
\end{equation*}

\begin{lm}
\label{l: bounds Haar}
Let $1<p<\infty$. Let $W,U$ be $d\times d$-matrix valued $\bfD$-dyadic biparameter $A_p$ weights on $\R\times\R$ with $[W,U']\ci{A_p,\bfD}<\infty$. Then, we have
\begin{equation*}
\Vert T_{\sigma}\Vert\ci{L^{p}(W)\rightarrow L^{p}(W)}\lesssim_{d,p}C[W]\ci{A_p,\bfD}^{\alpha(p)},
\end{equation*}
where
\begin{equation}
\label{alpha(p)}
\alpha(p):=2\gamma(p)+\frac{2\gamma(p')}{p-1},
\end{equation}
and
\begin{equation*}
\Vert T_{\sigma}\Vert\ci{L^{p}(U)\rightarrow L^{p}(W)}\lesssim_{d,p}C[W,U']\ci{A_p,\bfD}^{1/p}[U]\ci{A_p,\bfD}^{\frac{1}{p}+2\gamma(p)}[W]\ci{A_p,\bfD}^{\frac{1}{p}+\frac{2\gamma(p')}{p-1}}.
\end{equation*}
If $p=2$, then
\begin{equation*}
\Vert T_{\sigma}\Vert\ci{L^{2}(U)\rightarrow L^{2}(W)}\lesssim_{d}C[W,U']\ci{A_2,\bfD}^{1/2}[U]\ci{A_2,\bfD}^{2}[W]\ci{A_2,\bfD}^{2}.
\end{equation*}
\end{lm}

\begin{proof}
Set $g:=T_{\s}f$. Using Corollaries~\ref{c: lower bounds square functions} and~\ref{c: upper bound TL square function} we obtain
\begin{equation*}
\Vert g\Vert\ci{L^{p}(W)} \lesssim_{p,d} [W]\ci{A_p,\bfD}^{\frac{2\gamma(p')}{p-1}}\Vert S\ci{\bfD,W}g\Vert\ci{L^{p}} \leq
C[W]\ci{A_p,\bfD}^{\frac{2\gamma(p')}{p-1}}\Vert S\ci{\bfD,W}f\Vert\ci{L^{p}}
\lesssim_{p,d}C[W]\ci{A_p,\bfD}^{\alpha(p)}\Vert f\Vert\ci{L^{p}(W)}.
\end{equation*}
Moreover, by part (2) of Corollary~\ref{c: lower bounds square functions}, Lemma~\ref{l: replace weight tl square function} and Corollary~\ref{c: upper bound TL square function}, we obtain
\begin{align*}
\Vert g\Vert\ci{L^{p}(W)}&\lesssim_{p,d}[W]\ci{A_p,\bfD}^{\frac{1}{p}+\frac{2\gamma(p')}{p-1}}\Vert \wt{S}\ci{\bfD,W}g\Vert\ci{L^{p}}\leq
C[W]\ci{A_p,\bfD}^{\frac{1}{p}+\frac{2\gamma(p')}{p-1}}\Vert \wt{S}\ci{\bfD,W}f\Vert\ci{L^{p}}\\
&\lesssim_{p,d}C[W]\ci{A_p,\bfD}^{\frac{1}{p}+\frac{2\gamma(p')}{p-1}}[W,U']\ci{A_p,\bfD}^{1/p}\Vert\wt{S}\ci{\bfD,U}f\Vert\ci{L^{p}}\\
&\lesssim_{p,d}C[W]\ci{A_p,\bfD}^{\frac{1}{p}+\frac{2\gamma(p')}{p-1}}[W,U']\ci{A_p,\bfD}^{1/p}[U]\ci{A_p,\bfD}^{\frac{1}{p}+2\gamma(p)}\Vert f\Vert\ci{L^{p}(U)}.
\end{align*}
The estimate for $p=2$ is proved similarly.
One has to use part (1) of Corollary~\ref{c: lower bounds square functions}, Lemma~\ref{l: upper bound square function} and the fact that if $p=2,$ then $\Vert\wt{S}\ci{\bfD,W}f\Vert\ci{L^2}=\Vert S\ci{\bfD,W}f\Vert\ci{L^2}$ and $\Vert\wt{S}\ci{\bfD,U}f\Vert\ci{L^2}=\Vert S\ci{\bfD,U}f\Vert\ci{L^2}$.
\end{proof}

\section{Handling the cancellative terms}
\label{s: cancellative}

In the sequel we restrict ourselves to the product space $\R\times\R$; this is done for the sake of simplicity and to ease the notation. However, all results and proofs extend readily to any product space $\R^n\times\R^m$.

\subsection{One-parameter matrix-weighted shifted square functions}

Let $\cD$ be any dyadic grid in $\R$. Let $i$ be a non-negative integer. The (unweighted) shifted square function of complexity $i$ is defined by
\begin{equation*}
S^{i}\ci{\cD}f:=\bigg(\sum_{I\in\cD}\bigg(\sum_{K\in\text{ch}_{i}(I)}|f\ci{K}|\bigg)^2\frac{\1\ci{I}}{|I|}\bigg)^{1/2},\qquad f\in L^1\ti{loc}(\R;\C^d).
\end{equation*}
Such shifted square functions arise naturally as (essentially) the square functions of (one-parameter) Haar shifts in the context of the representation theorem for Calder\'on--Zygmund operators due to Hyt\ddoto nen \cite{hytonen}, and also play from the same point of view a very important role in \cite{holmes-lacey-wick}.

Let $1<p<\infty$, and let $W$ by a $d\times d$-matrix valued weight on $\R$.  Define the matrix-weighted shifted square function corresponding to the weight $W$ (and the exponent $p$) of complexity $i$ by
\begin{equation*}
S^{i}\ci{\cD,W}f(x):=\bigg(\sum_{I\in\cD}\bigg(\sum_{K\in\text{ch}_{i}(I)}|W(x)^{1/p}f\ci{K}|\bigg)^2\frac{\1\ci{I}(x)}{|I|}\bigg)^{1/2},\qquad x\in\R, f\in L^1\ti{loc}(\R;\C^d),
\end{equation*}
as well as the analougous operator where the weight is introduced in terms of its reducing operators (corresponding to exponent p):
\begin{equation*}
\wt{S}^{i}\ci{\cD,W}f:=\bigg(\sum_{I\in\cD}\bigg(\sum_{K\in\text{ch}_{i}(I)}|\cW\ci{I}f\ci{K}|\bigg)^2\frac{\1\ci{I}}{|I|}\bigg)^{1/2},\qquad f\in L^1\ti{loc}(\R;\C^d).
\end{equation*}
Both of these definitions follow \cite{holmes-lacey-wick}, up to incorporating the weight in the operator. Consider also the following variant of $S^{i}\ci{\cD,W}$:
\begin{equation*}
S^{i}\ci{\cD,W,\ast}f(x):=\bigg(\sum_{I\in\cD}\bigg|W(x)^{1/p}\sum_{K\in\text{ch}_{i}(I)}f\ci{K}\bigg|^2\frac{\1\ci{I}(x)}{|I|}\bigg)^{1/2},\qquad x\in\R, f\in L^1\ti{loc}(\R;\C^d).
\end{equation*}
Obviously $S^i\ci{\cD,W,\ast}f\leq S^i\ci{\cD,W}f$.

\subsubsection{Sparse domination and \texorpdfstring{$L^p$}{Lp} matrix-weighted bounds}

Let $1<p<\infty$. Let $\cD$ be any dyadic grid in $\R$. Let $W$ be a $d\times d$-matrix valued weight on $\R$. Let $i$ be a non-negative integer. For all $J\in\cD$, define the localized operator
\begin{align*}
S^{i}\ci{\cD,W,J}f(x)&:=
\bigg(\sum_{R\in\cD(J)}\bigg(\sum_{P\in\text{ch}_{i}(R)}|W(x)^{1/p}f\ci{P}|\bigg)^2\frac{\1\ci{R}(x)}{|R|}\bigg)^{1/2},\qquad x\in\R,f\in L^1\ti{loc}(\R;\C^d).
\end{align*}
Given any dyadically sparse collection $\cS$ of dyadic intervals in $\R$, following \cite{isralowitz-square function} we define the sparse positive operator
\begin{equation*}
\cA\ci{\cS,W}f(x):=\bigg(\sum_{L\in\cS}\La|\cW\ci{L}f|\Ra\ci{L}^2|W(x)^{1/p}\cW\ci{L}^{-1}|^2\1\ci{L}(x)\bigg)^{1/2},\qquad x\in\R,f\in L^1\ti{loc}(\R;\C^d).
\end{equation*}
See Subsection \ref{s: sparse} for a brief overview of the concept of sparse families.

\begin{prop}
\label{p: sparse_one_parameter_shifted_square_function}
Let $J\in\cD$. Then, for all $\e\in(0,1)$ and for all $f\in L^1\ti{loc}(\R;\C^d)$, there exists a dyadically $\e$-sparse collection $\cS$ of dyadic subintervals of $J$ (depending on $\e,J,p,W,f$ and $i$), such that
\begin{equation*}
S^{i}\ci{\cD,W,J}f\lesssim_{p,d,\e} i2^{i+/2}\cA\ci{\cS,W}f.
\end{equation*}
\end{prop}

\begin{proof}
The proof will be an adaptation of the proof of pointwise sparse domination of the usual matrix-weighted square function in \cite[Section 2]{isralowitz-square function}, in the same way that the proof in \cite[Proposition 5.3]{barron-pipher} is an adaptation of the proof in \cite{hytonen-petermichl-volberg}.

Set $\cS_{0}:=\lbrace J\rbrace$. Let $\cS_1$ be the family of all maximal dyadic subintervals $L$ of $J$ such that
\begin{equation*}
\sum_{\substack{I\in\cD(J)\\L\subseteq I}}\bigg(\sum_{P\in\text{ch}_{i}(I)}|\cW\ci{J}f\ci{P}|\bigg)^2\frac{1}{|I|}>c^2i^22^{i}\La|\cW\ci{J}f|\Ra\ci{J}^2,
\end{equation*}
for some large enough $c>0$ to be chosen in a moment (depending only on $\e$ and $d$). It is clear that
\begin{equation*}
L\subseteq\lbrace S^{i}\ci{\cD}(\1\ci{J}\cW\ci{J}f)>c\cdot i2^{i/2}\La|\cW\ci{J}f|\Ra\ci{J}\rbrace,\qquad\forall L\in\cS_1
\end{equation*}
(we emphasize that $\cW\ci{J}$ is a constant matrix). It was proved in \cite[Proposition 3.6]{barron_phd} that
\begin{equation*}
\Vert S^{i}\ci{\cD}\Vert\ci{L^1\rightarrow L^{1,\infty}}\lesssim i2^{i/2},
\end{equation*}
which trivially implies
\begin{equation*}
\Vert S^{i}\ci{\cD}\Vert\ci{L^1(\R;\C^d)\rightarrow L^{1,\infty}}\lesssim_{d} i2^{i/2}.
\end{equation*}
It follows that
\begin{equation*}
\sum_{L\in\cS_1}|L|\lesssim_{d}\frac{|J|}{c},
\end{equation*}
therefore
\begin{equation*}
\sum_{L\in\cS_1}|L|\leq(1-\e)|J|
\end{equation*}
provided $c$ is chosen to be large enough (depending only on $d$ and $\e$).

Set
\begin{equation*}
\cF:=\cD(J)\setminus\bigg(\bigcup_{L\in\cS_1}\cD(L)\bigg).
\end{equation*}
Set also $E:=\bigcup_{L\in\cS_1}L$. Then, we can write
\begin{align}
\label{decomposition}
S^{i}\ci{\cD,W,J}f(x)^{2}&=
\sum_{L\in\cF}\bigg(\sum_{I\in\text{ch}_{i}(L)}|W(x)^{1/p}f\ci{I}|\bigg)^2\frac{\1\ci{L\cap E}(x)}{|L|}\\
\nonumber&+\sum_{L\in\cF}\bigg(\sum_{I\in\text{ch}_{i}(L)}|W(x)^{1/p}f\ci{I}|\bigg)^2\frac{\1\ci{L\setminus E}(x)}{|L|}\\
\nonumber&+\sum_{L\in\cS_1}\sum_{I\in\cD(L)}\bigg(\sum_{K\in\text{ch}_{i}(I)}|W(x)^{1/p}f\ci{K}|\bigg)^2\frac{\1\ci{I}(x)}{|I|}.
\end{align}
In each of the terms
\begin{equation*}
\sum_{I\in\cD(L)}\bigg(\sum_{K\in\text{ch}_{i}(I)}|W(x)^{1/p}f\ci{K}|\bigg)^2\frac{\1\ci{I}(x)}{|I|},\qquad L\in\cS_1
\end{equation*}
we just iterate the construction. So it remains to estimate the first two terms in the right-hand side of \eqref{decomposition}. Call them $A(x)$ and $B(x)$ respectively.

To estimate term $A(x)$, we notice that $A(x)$ vanishes outside $E$, and that for all $I\in\cS_1$ and for all $x\in I$, if $L\in\cF$ is such that $x\in L$, then $L\cap I\neq\emptyset$, therefore necessarily $I\subsetneq L$, therefore
\begin{align*}
A(x)&=
\sum_{\substack{L\in\cF\\ I\subsetneq L}}\bigg(\sum_{K\in\text{ch}_{i}(L)}|W(x)^{1/p}f\ci{K}|\bigg)^2\frac{1}{|L|}\\
&\leq\sum_{\substack{L\in\cF\\ I\subsetneq L}}\bigg(\sum_{K\in\text{ch}_{i}(L)}|\cW\ci{J}f\ci{K}|\bigg)^2|W(x)^{1/p}\cW\ci{J}^{-1}|^2\frac{1}{|L|}\\
&=\sum_{\substack{L\in\cF\\ \hat{I}\subseteq L}}\bigg(\sum_{K\in\text{ch}_{i}(L)}|\cW\ci{J}f\ci{K}|\bigg)^2|W(x)^{1/p}\cW\ci{J}^{-1}|^2\frac{1}{|L|}\\
&\leq c\cdot i^22^{i}|W(x)^{1/p}\cW\ci{J}^{-1}|^2\La|\cW\ci{J}f|\Ra\ci{J}^21\ci{J}(x),
\end{align*}
by the maximality of $I$.

To estimate $B(x)$, we notice that for all $x\in J\setminus E$, we have
\begin{equation*}
\sum_{\substack{I\in\cD(J)\\L\subseteq I}}\bigg(\sum_{P\in\text{ch}_{i}(R)}|\cW\ci{J}f\ci{P}|\bigg)^2\frac{1}{|I|}\leq c\cdot i^22^{i}\La|\cW\ci{J}f|\Ra\ci{J}^2,
\end{equation*}
for all $L\in\cD(J)$ with $x\in L$, therefore if $(L_n)^{\infty}_{n=1}$ is the strictly decreasing (maybe finite) sequence of all intervals in $\cF$ containing $x$, then
\begin{equation*}
\sum_{\substack{I\in\cD(J)\\L_k\subseteq I}}\bigg(\sum_{P\in\text{ch}_{i}(I)}|\cW\ci{J}f\ci{P}|\bigg)^2\frac{1}{|I|}\leq c\cdot i^22^{i}\La|\cW\ci{J}f|\Ra\ci{J}^2,\qquad\forall k=1,2,\ldots,
\end{equation*}
therefore
\begin{align*}
B(x)&=\sum_{k=1}^{\infty}\bigg(\sum_{I\in\text{ch}_{i}(L_k)}|W(x)^{1/p}f\ci{I}|\bigg)^2\frac{1}{|L_k|}\\
&\leq|W(x)^{1/p}\cW\ci{J}^{-1}|^2\lim_{k\rightarrow\infty}\sum_{l=1}^{k}\bigg(\sum_{I\in\text{ch}_{i}(L_l)}|\cW\ci{J}f\ci{I}|\bigg)^2\frac{1}{|L_l|}\\
&\leq|W(x)^{1/p}\cW\ci{J}^{-1}|^2\lim_{k\rightarrow\infty}\sum_{\substack{I\in\cD(J)\\L_{k}\subseteq I}}\bigg(\sum_{K\in\text{ch}_{i}(I)}|\cW\ci{J}f\ci{K}|\bigg)^2\frac{1}{|I|}\\
&\leq ci^22^{i}|W(x)^{1/p}\cW\ci{J}^{-1}|^2\La |\cW\ci{J}f|\Ra\ci{J}^21\ci{J}(x),
\end{align*}
concluding the proof.
\end{proof}

\begin{cor}
\label{c: bound one parameter shifted square function}
Let $i,p,\cD,W$ be as before. Then, there holds
\begin{align*}
\Vert S^{i}\ci{\cD,W,\ast}\Vert\ci{L^{p}(W)\rightarrow L^{p}}\leq\Vert S^{i}\ci{\cD,W}\Vert\ci{L^{p}(W)\rightarrow L^{p}}\lesssim_{p,d}i2^{i/2}[W]\ci{A_p,\cD}^{\gamma(p)}.
\end{align*}
\end{cor}

\begin{proof}
It was proved in \cite[Section 3]{isralowitz-square function} that
\begin{align*}
\Vert\cA\ci{W,\cS}\Vert\ci{L^{p}(W)\rightarrow L^{p}}\lesssim_{p,d,\e}[W]\ci{A_p,\cD}^{\gamma(p)},
\end{align*}
for any dyadically $\e$-sparse collection $\cS$ of dyadic subintervals of $J$, for any $J\in\cD$. It follows by Proposition \ref{p: sparse_one_parameter_shifted_square_function} that
\begin{align*}
\Vert S^{i}\ci{\cD,W,J}f\Vert\ci{L^{p}}\lesssim_{p,d}i2^{i/2}[W]\ci{A_p,\cD}^{\gamma(p)}\Vert f\1\ci{J}\Vert\ci{L^{p}(W)},\qquad\forall f\in L^1\ti{loc}(\R;\C^d),
\end{align*}
for all $J\in\cD$. A simple application of the Monotone Convergences Theorem yields then the desired result.
\end{proof}

\subsection{Biparameter matrix-weighted shifted square functions}
\label{s: shifted square functions}

Let $\bfD=\cD^1\times\cD^2$ be any dyadic product grid in $\R\times\R$. Let $\bi=(i_1,i_2)$ be a pair of non-negative integers. The (unweighted) biparameter shifted square function of complexity $\bi$ is defined by
\begin{equation*}
S^{\bi}\ci{\bfD}f:=\bigg(\sum_{R\in\bfD}\bigg(\sum_{P\in\ch_{\bi}(R)}|f\ci{P}|\bigg)^2\frac{\1\ci{R}}{|R|}\bigg)^{1/2},\qquad f\in L^1\ti{loc}(\R^2;\C^d).
\end{equation*}
Such shifted square functions arise naturally as (essentially) the square functions of biparameter Haar shifts in the context of the representation theorem for biparameter Journ\'e operators due to Martikainen \cite{martikainen}, and also play from the same point of view a very important role in \cite{holmes-petermichl-wick}.

Let $1<p<\infty$, and let $W$ be a $d\times d$-matrix valued weight on $\R^2$. Define the biparameter matrix-weighted shifted square function corresponding to the weight $W$ (and the exponent $p$) of complexity $\bi$ by
\begin{equation*}
S^{\bi}\ci{\bfD,W}f(x):=\bigg(\sum_{R\in\bfD}\bigg(\sum_{P\in\ch_{\bi}(R)}|W(x)^{1/p}f\ci{P}|\bigg)^2\frac{\1\ci{R}(x)}{|R|}\bigg)^{1/2},
\end{equation*}
for $x\in\R^2$ and $f\in L^1\ti{loc}(\R^2;\C^d)$.
We have of course also the alternative definition
\begin{equation*}
\wt{S}^{\bi}\ci{\bfD,W}f:=\bigg(\sum_{R\in\bfD}\bigg(\sum_{P\in\ch_{\bi}(R)}|\cW\ci{R}f\ci{P}|\bigg)^2\frac{\1\ci{R}}{|R|}\bigg)^{1/2},\qquad f\in L^1\ti{loc}(\R^2;\C^d),
\end{equation*}
where reducing operators for $W$ are taken with respect to the exponent $p$. Consider also the variant of $S^{\bi}\ci{\bfD,W}$
\begin{equation*}
S^{\bi}\ci{\bfD,W,\ast}f(x):=\bigg(\sum_{R\in\bfD}\bigg|W(x)^{1/p}\sum_{P\in\ch_{\bi}(R)}f\ci{P}\bigg|^2\frac{\1\ci{R}(x)}{|R|}\bigg)^{1/2}
\end{equation*}
for $x\in\R^2$ and $f\in L^1\ti{loc}(\R^2;\C^d)$. We will now investigate $L^{p}$ matrix-weighted bounds for these operators.

Lemma \ref{l: compare2} below is proved by iterating Corollary \ref{c: bound one parameter shifted square function} with the help of Lemma \ref{l: matrix weight vector valued extensions}, in the exact same way that an iteration of estimate \eqref{isralowitz one parameter square function} yields Lemma \ref{l: upper bound square function}, so we omit its proof.

\begin{lm}
\label{l: compare2}
Let $1<p<\infty$. Let $W$ be a $d\times d$-matrix valued $\bfD$-dyadic biparameter $A_p$ weight on $\R\times\R$. Then, there holds
\begin{equation*}
\Vert S^{\bi}\ci{\bfD,W,\ast}\Vert\ci{L^{p}(W)\rightarrow L^{p}(\R^2)}\lesssim_{p,d}i_1i_22^{(i_1+i_2)/2}[W]\ci{A_p,\bfD}^{2\gamma(p)}.
\end{equation*}
\end{lm}

\begin{cor}
\label{c: upper bound TL shifted square function}
Let $1<p<\infty$, and let $W$ be a $d\times d$-matrix valued $\bfD$-dyadic biparameter $A_p$ weight on $\R\times\R$. Then, there holds
\begin{equation*}
\Vert\wt{S}^{\bi}\ci{\bfD,W}\Vert\ci{L^{p}(W)\rightarrow L^{p}}\lesssim_{p,d}i_1i_22^{(i_1+i_2)/2}[W]\ci{A_p,\bfD}^{\frac{1}{p}+2\gamma(p)+\alpha(p)}.
\end{equation*}
\end{cor}

\begin{proof}
We first show that there exists a sequence $\s=(\s\ci{R})\ci{R\in\bfD}$ in $\lbrace-1,1\rbrace$ depending on $d,i,W$ and $f$, such that
\begin{equation}
\label{compare1}
\Vert\wt{S}^{\bi}\ci{\bfD,W}f\Vert\ci{L^{p}}\lesssim_{d}[W]\ci{A_p,\bfD}^{1/p}\Vert S^{\bi}\ci{\bfD,W,\ast}(T_{\s}f)\Vert\ci{L^{p}},\qquad\forall f\in L^{2}(\R^{2};\C^d).
\end{equation}
By Lemma \ref{l: vectors} below, we have that for all $R\in\bfD$, there exists a sequence $(\s\ci{P}^{R})\ci{P\in\ch_{\bi}(R)}$ in $\lbrace-1,1\rbrace$, depending on $d,\bi,W,f$ and $R$, such that
\begin{equation*}
\sum_{P\in\ch_{\bi}(R)}|\cW\ci{R}f\ci{P}|\sim_{d}\bigg|\sum_{P\in\ch_{\bi}(R)}\s\ci{P}^{R}\cW\ci{R}f\ci{P}\bigg|.
\end{equation*}
Clearly, every rectangle $P\in\bfD$ has a unique $\bi$-th ancestor $P^{(\bi)}$ in $\bfD$. Therefore, we can consider the sequence $\s=(\s\ci{P})\ci{P\in\bfD}$ in $\lbrace-1,1\rbrace$ given by
\begin{equation*}
\s\ci{P}:=\s\ci{P}^{P^{(\bi)}},\qquad\forall P\in\bfD.
\end{equation*}
Note that the sequence $\sigma$ depends only on $d,i,W$ and $f$. Consider the biparameter martingale transform $\wt{f}:=T_{\s}f$ of $f$. Then clearly
\begin{equation*}
\wt{S}^{\bi}\ci{\bfD,W}f\sim_{d}\bigg(\sum_{R\in\bfD}\bigg|\cW\ci{R}\sum_{P\in\ch_{\bi}(R)}\wt{f}\ci{P}\bigg|^2\frac{\1\ci{R}}{|R|}\bigg)^{1/2},
\end{equation*}
where $\wt{f}\ci{P}:=\s\ci{P}f\ci{P},~P\in\bfD,$ are the Haar coefficients of $\wt{f}$. Then, we have
\begin{align*}
\Vert\wt{S}^{\bi}\ci{\bfD,W}f\Vert\ci{L^p}^{p}&\sim_{d}\int_{\R^2}\bigg(\sum_{R\in\bfD}\bigg|\cW\ci{R}\sum_{P\in\ch_{\bi}(R)}\wt{f}\ci{P}\bigg|^2\frac{\1\ci{R}(x)}{|R|}\bigg)^{p/2}\mathd x\\
&\lesssim_{p,d}[W]\ci{A_p,\bfD}\int_{\R^2}\bigg(\sum_{R\in\bfD}\bigg(\bigg|W^{1/p}\sum_{P\in\ch_{\bi}(R)}\wt{f}\ci{P}\bigg|\bigg)\ci{R}^2\frac{\1\ci{R}(x)}{|R|}\bigg)^{p/2}\mathd x.
\end{align*}
where we applied Lemma \ref{l: second scalar} to get from the first to the second line. In the exact same way as in the proof of Lemma \ref{l: dominate TL square function by square function}, we obtain
\begin{align*}
\int_{\R^2}\bigg(\sum_{R\in\bfD}\bigg(\bigg|W^{1/p}\sum_{P\in\ch_{\bi}(R)}\wt{f}\ci{P}\bigg|\bigg)\ci{R}^2\frac{\1\ci{R}(x)}{|R|}\bigg)^{p/2}\mathd x\lesssim_{p,d}\Vert S^{\bi}\ci{\bfD,W,\ast}\wt{f}\Vert\ci{L^{p}}^{p},
\end{align*}
proving \eqref{compare1}.

Combining \eqref{compare1} with Lemmas \ref{l: compare2} and \ref{l: bounds Haar} we deduce the desired result.
\end{proof}

\begin{lm}
\label{l: vectors}
Let $v_1,\ldots,v_k$ be vectors in $\C^d$. Then, there exist $\s_1,\ldots,\s_k\in\lbrace-1,1\rbrace$, such that
\begin{equation*}
\sum_{i=1}^{k}|v_i|\sim_{d}\bigg|\sum_{i=1}^{k}\s_iv_i\bigg|.
\end{equation*}
\end{lm}

\begin{proof}
For each $i=1,\ldots,k$, write $w_i^1:=\text{Re}(v_i)$ and $w_i^2:=\text{Im}(v_i)$.
Also, for each $i=1,\ldots,k$ and $j=1,2,$ express the vector $w^{j}_{i}$ in $\R^d$ in components as $w^{j}_{i}=(w^{j}_{i,1},\ldots,w^{j}_{i,d})$. Then, it is clear that
\begin{equation*}
\sum_{i=1}^{k}|v_i|\leq C(d)\sum_{i=1}^{k}\sum_{\ell=1}^{d}\sum_{j=1}^{2}|w_{i,\ell}^{j}|.
\end{equation*}
Therefore, there exist $\ell\in\lbrace1,\ldots,d\rbrace$ and $j\in\lbrace1,2\rbrace$, such that
\begin{equation*}
\sum_{i=1}^{k}|w_{i,\ell}^{j}|\geq\frac{1}{2dC(d)}\sum_{i=1}^{k}|v_i|.
\end{equation*}
Taking $\s_i:=\text{sgn}(w_{i,\ell}^{j})$, for all $i=1,\ldots,k$, where $\text{sgn}(t):=t/|t|$ for $t\in\R\setminus\lbrace0\rbrace$, and $\text{sgn}(0):=1$, we obtain the required result.
\end{proof}

For $p=2$, the estimate of Corollary \ref{c: upper bound TL shifted square function} can be considerably improved, by using a different, direct method, adapting the prooof of \cite[Lemma 2.2]{holmes-lacey-wick}.

\begin{lm}
\label{l: tl shifted-square-function-p=2}
Assume $p=2$. Let $W$ be a $d\times d$-matrix valued weight on $\R^2$. Then, there holds
\begin{equation*}
\Vert \wt{S}\ci{\bfD,W}^{\bi}f\Vert\ci{L^{2}}\lesssim_{d}2^{(i_1+i_2)/2}[W]\ci{A_2,\bfD}^{5/2}\Vert f\Vert\ci{L^{2}(W)},
\end{equation*}
for all $f\in L^{\infty}\ti{c}(\R^{2};\C^d)$.
\end{lm}

\begin{proof}
We have
\begin{align*}
\Vert\wt{S}\ci{\bfD,W}^{\bi}f\Vert\ci{L^{2}}^2&=\sum_{R\in\bfD}\bigg(\sum_{P\in\ch_{\bi}(R)}|\cW\ci{R}f\ci{P}|\bigg)^2\leq\sum_{R\in\bfD}\bigg(\sum_{P\in\ch_{\bi}(R)}|\cW\ci{R}\cW\ci{P}^{-1}|\cdot|\cW\ci{P}f\ci{P}|\bigg)^2\\
&\leq\sum_{R\in\bfD}\bigg(\sum_{P\in\ch_{\bi}(R)}|\cW\ci{P}f\ci{P}|^2\bigg)\bigg(\sum_{P\in\ch_{\bi}(R)}|\cW\ci{R}\cW\ci{P}^{-1}|^2\bigg).
\end{align*}
We notice that for all $R\in\bfD$, we have
\begin{align*}
\sum_{P\in\ch_{\bi}(R)}|\cW\ci{R}\cW\ci{P}^{-1}|^2&=\sum_{P\in\ch_{\bi}(R)}|\cW\ci{P}^{-1}\cW\ci{R}|^2\lesssim_{d}\sum_{P\in\ch_{\bi}(R)}|\cW\ci{P}'\cW\ci{R}|^2\\
&\sim_{d}\sum_{P\in\ch_{\bi}(R)}\strokedint_{P}|W(x)^{-1/2}\cW\ci{R}|^2\mathd x=
2^{i_1+i_2}\strokedint_{R}|W(x)^{-1/2}\cW\ci{R}|^2\mathd x\\
&\lesssim_{d}2^{i_1+i_2}|\cW\ci{R}'\cW\ci{R}|^2\lesssim_{d}2^{i_1+i_2}[W]\ci{A_2,\bfD},
\end{align*}
where in the first $\lesssim_{d}$ we used Lemma \ref{l: replace inverse by prime} coupled with \eqref{equivalence_matrix_norm_columns}. Thus
\begin{align*}
\Vert\wt{S}\ci{\bfD,W}^{\bi}f\Vert\ci{L^{2}}^2&\lesssim_{d}2^{i_1+i_2}[W]\ci{A_2,\bfD}\sum_{R\in\bfD}\sum_{P\in\ch_{\bi}(R)}|\cW\ci{P}f\ci{P}|^2=2^{i_1+i_2}[W]\ci{A_2,\bfD}\Vert S\ci{\bfD,W}f\Vert\ci{L^2}^2\\
&=2^{i_1+i_2}[W]\ci{A_2,\bfD}\Vert \wt{S}\ci{\bfD,W}f\Vert\ci{L^2}^2\lesssim_{d}2^{i_1+i_2}[W]\ci{A_2,\bfD}^{5}\Vert f\Vert\ci{L^2(W)}^2,
\end{align*}
where in the last $\lesssim_{d}$ we used Lemma \ref{l: upper bound square function}.
\end{proof}

In Section \ref{s: non-cancellative} (and essentially only for $p\neq2$) we will need to use the following (slightly stronger) one-parameter counterpart of Corollary \ref{c: upper bound TL shifted square function}.

\begin{lm}
\label{l: upper bound TL shifted square function one parameter}
Let $1<p<\infty$. Let $\cD$ be a dyadic grid in $\R$. Let $W$ be a $d\times d$-matrix valued $\cD$-dyadic $A_p$ weight on $\R$. Let $i$ be a nonnnegative integer. Then, there holds
\begin{equation*}
\Vert\wt{S}^{i}\ci{\cD,W}\Vert\ci{L^{p}(W)\rightarrow L^{p}}\lesssim_{p,d}i2^{i/2}[W]\ci{A_p,\cD}^{\frac{1}{p}+\gamma(p)}.
\end{equation*}
\end{lm}

\begin{proof}
For any bounded sequence $\s=(\s\ci{I})\ci{I\in\cD}$ of complex numbers, we denote
\begin{equation*}
T_{\s}f:=\sum_{I\in\cD}\s\ci{I}f\ci{I}h\ci{I}.
\end{equation*}
Then, similarly to the proof of Corollary \ref{c: upper bound TL shifted square function} we have that there exist a sequence $\sigma=(\sigma\ci{I})\ci{I\in\cD}$ in $\lbrace-1,1\rbrace$ (depending only on $d,i,W$ and $f$) such that
\begin{align*}
\Vert \wt{S}^{i}\ci{\cD,W}f\Vert\ci{L^{p}}&\lesssim_{p,d}[W]^{1/p}\ci{A_p,\cD}\Vert S^{i}\ci{\cD,W,\ast}(T_{\sigma}f)\Vert\ci{L^{p}}\leq[W]^{1/p}\ci{A_p,\cD}\Vert S^{i}\ci{\cD,W}(T_{\sigma}f)\Vert\ci{L^{p}}\\
&\leq
[W]^{1/p}\ci{A_p,\cD}\Vert S^{i}\ci{\cD,W}f\Vert\ci{L^{p}}.
\end{align*}
The desired conclusion follows then from Corollary \ref{c: bound one parameter shifted square function}.
\end{proof}

\begin{lm}
\label{l: replace shifted square function with tl shifted square function}
Let $W,U$ be $d\times d$-matrix valued weights on $\R\times\R$. Then, there holds
\begin{equation*}
\Vert S\ci{\bfD,W}^{\bi}f\Vert\ci{L^{p}}\lesssim_{p,d}[W]\ci{A_p,\bfD}^{\frac{1}{p}+2\gamma(p)+\frac{2\gamma(p')}{p-1}}\Vert \wt{S}\ci{\bfD,W}^{\bi}f\Vert\ci{L^{p}}.
\end{equation*}
If $p=2$, then
\begin{equation*}
\Vert S\ci{\bfD,W}^{\bi}f\Vert\ci{L^{2}}\lesssim_{d}\Vert \wt{S}\ci{\bfD,W}^{\bi}f\Vert\ci{L^{2}}.
\end{equation*}
\end{lm}

\begin{proof}
We have
\begin{align*}
\Vert S\ci{\bfD,W}^{\bi}f\Vert\ci{L^{p}}^p&=\int_{\R^2}\left(\sum_{R\in\bfD}\left(\sum_{P\in\ch_{\bi}(R)}|W(x)^{1/p}f\ci{P}|\right)^2\frac{\1\ci{R}(x)}{|R|}\right)^{p/2}\dd x\\
&\leq\int_{\R^2}\left(\sum_{R\in\bfD}\left(\sum_{P\in\ch_{\bi}(R)}|\cW\ci{R}f\ci{P}|\right)^2|W(x)^{1/p}\cW\ci{R}^{-1}|^2\frac{\1\ci{R}(x)}{|R|}\right)^{p/2}\dd x\\
&=\int_{\R^2}\left(\sum_{R\in\bfD}|W(x)^{1/p}a\ci{R}\cW\ci{R}^{-1}|^2\frac{\1\ci{R}(x)}{|R|}\right)^{p/2}\dd x,
\end{align*}
where
\begin{equation*}
a\ci{R}:=\sum_{P\in\ch_{\bi}(R)}|\cW\ci{R}f\ci{P}|,\qquad R\in\bfD.
\end{equation*}
Using first Lemma \ref{l: upper bound square function} and then part (2) of Corollary \ref{c: lower bounds square functions}, both coupled with \eqref{equivalence_matrix_norm_columns} and \eqref{break power of sum}, we get
\begin{align*}
&\int_{\R^2}\left(\sum_{R\in\bfD}|W(x)^{1/p}a\ci{R}\cW\ci{R}^{-1}|^2\frac{\1\ci{R}(x)}{|R|}\right)^{p/2}\dd x\\
&\lesssim_{p,d}[W]\ci{A_p,\bfD}^{2\gamma(p)+\frac{1}{p}+\frac{2\gamma(p')}{p-1}}
\int_{\R^2}\left(\sum_{R\in\bfD}|\cW\ci{R}a\ci{R}\cW\ci{R}^{-1}|^2\frac{\1\ci{R}(x)}{|R|}\right)^{p/2}\dd x\\
&=[W]\ci{A_p,\bfD}^{2\gamma(p)+\frac{1}{p}+\frac{2\gamma(p')}{p-1}}\Vert\wt{S}^{\bi}\ci{\bfD,W}f\Vert\ci{L^{p}}.
\end{align*}
If $p=2$, then
\begin{align*}
\Vert S\ci{\bfD,W}^{\bi}f\Vert\ci{L^{2}}^2&\leq\int_{\R^2}\left(\sum_{R\in\bfD}|W(x)^{1/2}a\ci{R}\cW\ci{R}^{-1}|^2\frac{\1\ci{R}(x)}{|R|}\right)\dd x\\
&=\sum_{R\in\bfD}|a\ci{R}|^2\strokedint_{R}|W(x)^{1/2}\cW\ci{R}^{-1}|^2\dd x\sim_{d}\sum_{R\in\bfD}|a\ci{R}|^2=\Vert\wt{S}^{\bi}\ci{\bfD,W}f\Vert\ci{L^2}^2.
\end{align*}
\end{proof}

The following lemma is proved in the exact same way as Lemma \ref{l: replace weight tl square function}.

\begin{lm}
\label{l: replace weight tl shifted square function}
Let $W,U$ be $d\times d$-matrix valued weights on $\R\times\R$. Then, there holds
\begin{equation*}
\wt{S}\ci{\bfD,W}^{\bi}f\lesssim_{p,d}[W,U']\ci{A_p,\bfD}^{1/p}\wt{S}\ci{\bfD,U}^{\bi}f,
\end{equation*}
for all $f\in L^1\ti{loc}(\R^2;\C^d)$.
\end{lm}

\subsection{\texorpdfstring{$L^{p}$}{Lp} matrix-weighted bounds for cancellative Haar shifts}

Here we show that the approach of \cite[Section 7]{holmes-petermichl-wick} for deducing weighted bounds for biparameter cancellative shifts in the scalar-weighted setting can be adapted to the matrix-weighted setting.

Let $\bi=(i_1,i_2),\bj=(j_1,j_2)$ be pairs of non-negative integers. Let
\begin{equation*}
T^{\bi,\bj}f:=\sum_{R\in\bfD}\sum_{P\in\text{ch}_{\bi}(R)}\sum_{Q\in\text{ch}_{\bj}(R)}a\ci{PQR}f\ci{P}h\ci{Q},\qquad f\in L^2(\R^{2};\C^d),
\end{equation*}
where $a\ci{PQR}$ are complex numbers satisfying the bound
\begin{equation*}
|a\ci{PQR}|\leq\frac{\sqrt{|P|\cdot|Q|}}{\sqrt{|R|}}=2^{-\frac{1}{2}(i_1+i_2+j_1+j_2)}.
\end{equation*}

\begin{lm}
\label{l: bound Haar shifts}
Let $1<p<\infty$. Let $W,U$ be $\bfD$-dyadic biparameter $d\times d$-matrix valued $A_p$-weights on $\R\times\R$ such that $[W,U']\ci{A_p,\bfD}<\infty$. Then, we have the bounds
\begin{equation*}
\Vert T^{\bi,\bj}\Vert\ci{L^p(W)\rightarrow L^p(W)}\lesssim_{p,d}i_1j_1i_2j_2[W]\ci{A_p,\bfD}^{\alpha_1(p)+\alpha_2(p)},
\end{equation*}
and
\begin{equation*}
\Vert T^{\bi,\bj}\Vert\ci{L^p(U)\rightarrow L^p(W)}\lesssim_{p,d}i_1j_1i_2j_2 [W,U']\ci{A_p,\bfD}^{1/p}[W]\ci{A_p,\bfD}^{\alpha_1(p)}[U]\ci{A_p,\bfD}^{\alpha_2(p)},
\end{equation*}
where
\begin{equation}
\label{alpha_1(p)}
\alpha_1(p):=\frac{2\gamma(p')}{p-1},
\end{equation}
and
\begin{equation}
\label{alpha_2(p)}
\alpha_2(p):=
\begin{cases}
\frac{5}{2},\text{ if }p=2\\\\
\frac{2}{p}+6\gamma(p)+\frac{4\gamma(p')}{p-1}
,\text{ otherwise}
\end{cases}
.
\end{equation}
\end{lm}

\begin{proof}
By Corollary \ref{c: lower bounds square functions} we have
\begin{align*}
\Vert T^{\bi,\bj}f\Vert\ci{L^p(W)}\lesssim_{p,d}[W]\ci{A_p,\bfD}^{\frac{2\gamma(p')}{p-1}}\Vert S\ci{\bfD,W}(T^{\bi,\bj}f)\Vert\ci{L^{p}}.
\end{align*}
Recall that every $Q\in\bfD$ has a unique $\bj$-th ancestor in $\bfD$. Therefore, we have
\begin{align*}
(S\ci{\bfD,W}(T^{\bi,\bj}f))^2&=\sum_{R\in\bfD}\sum_{Q\in\text{ch}_{\bj}(R)}\bigg|W(x)^{1/p}\sum_{P\in\ch_{\bi}(R)}a\ci{PQR}f\ci{P}\bigg|^2\frac{\1\ci{Q}}{|Q|}\\
&\leq 2^{-(i_1+j_1+i_2+j_2)}\sum_{R\in\bfD}\sum_{Q\in\text{ch}_{\bj}(R)}\bigg(\sum_{P\in\ch_i(R)}|W(x)^{1/p}f\ci{P}|\bigg)^2\frac{\1\ci{Q}}{|Q|}\\
&=2^{-(i_1+i_2)}(S^{\bi}\ci{\bfD,W}(f))^2,
\end{align*}
where we used the fact that
\begin{equation*}
\sum_{Q\in\ch_{\bj}(R)}\frac{\1\ci{Q}}{|Q|}=2^{j_1+j_2}\frac{\1\ci{R}}{|R|},\qquad\forall R\in\bfD.
\end{equation*}
Combining Corollary \ref{c: upper bound TL shifted square function} and Lemma \ref{l: replace shifted square function with tl shifted square function} we obtain
\begin{equation*}
\Vert S^{\bi}\ci{\bfD,W}f\Vert\ci{L^p}\lesssim_{p,d}i_1i_22^{(i_1+i_2)/2}[W]\ci{A_p,\bfD}^{\frac{2}{p}+6\gamma(p)+\frac{4\gamma(p')}{p-1}}\Vert f\Vert\ci{L^p(W)},
\end{equation*}
and combining Corollary \ref{c: upper bound TL shifted square function} with Lemmas \ref{l: replace shifted square function with tl shifted square function} and \ref{l: replace weight tl shifted square function} we obtain
\begin{equation*}
\Vert S^{\bi}\ci{\bfD,W}f\Vert\ci{L^p}\lesssim_{p,d}i_1i_22^{(i_1+i_2)/2}[W,U']\ci{A_p,\bfD}^{1/p}[U]\ci{A_p,\bfD}^{\frac{2}{p}+6\gamma(p)+\frac{4\gamma(p')}{p-1}}\Vert f\Vert\ci{L^p(U)}.
\end{equation*}
If $p=2$, then using Lemma \ref{l: tl shifted-square-function-p=2} instead of Corollary \ref{c: upper bound TL shifted square function} we get the better bounds
\begin{equation*}
\Vert S^{\bi}\ci{\bfD,W}f\Vert\ci{L^2}\lesssim_{d}2^{(i_1+i_2)/2}[W]\ci{A_2,\bfD}^{9/2}\Vert f\Vert\ci{L^2(W)},
\end{equation*}
and
\begin{equation*}
\Vert S^{\bi}\ci{\bfD,W}f\Vert\ci{L^2}\lesssim_{d}2^{(i_1+i_2)/2}[W,U']\ci{A_2,\bfD}^{1/2}[U]\ci{A_2,\bfD}^{5/2}\Vert f\Vert\ci{L^2(U)}.
\end{equation*}
The desired conclusion follows.
\end{proof}

Combining Lemma \ref{l: bound Haar shifts} with Martikainen's representation theorem (Theorem \ref{t: martikainen representation}), we obtain the following.

\begin{cor}
\label{c: paraproduct free Journe}
Let $T$ be any paraproduct-free \jr operator on $\R\times\R$. Let $1<p<\infty$. Let $W,U$ be biparameter $d\times d$-matrix valued $A_p$-weights on $\R\times\R$ such that $[W,U']\ci{A_p(\R\times\R)}<\infty$. Then, we have the bounds
\begin{equation*}
\Vert T\Vert\ci{L^p(W)\rightarrow L^p(W)}\lesssim_{p,d,T}[W]\ci{A_p(\R\times\R)}^{\alpha_1(p)+\alpha_2(p)},
\end{equation*}
and
\begin{equation*}
\Vert T\Vert\ci{L^p(U)\rightarrow L^p(W)}\lesssim_{p,d,T} [W,U']\ci{A_p(\R\times\R)}^{1/p}[W]\ci{A_p(\R\times\R)}^{\alpha_1(p)}[U]\ci{A_p(\R\times\R)}^{\alpha_2(p)}.
\end{equation*}
\end{cor}

\begin{rem}
\label{r: paraproduct free Journe dual}
Let $T,p,U,W$ be as in Corollary \ref{c: paraproduct free Journe}. Then, $T^{\ast}$ is also a paraproduct-free \jr operator. Thus, we can apply Corollary \ref{c: paraproduct free Journe} for $T^\ast$ in the place of $T$, the exponent $p'$ in the place of $p$, $U'$ in the place of $W$ and $W'$ in the place of $U$. Using then the facts that
\begin{equation*}
\Vert T^{\ast}\Vert\ci{L^{p'}(W')\rightarrow L^{p'}(U')}=\Vert T\Vert\ci{L^{p}(U)\rightarrow L^{p}(W)},
\end{equation*}
\begin{equation*}
[U',(W')']\ci{A_{p'}(\R\times\R)}=[U',W]\ci{A_{p'}(\R\times\R)}\sim_{p,d}[W,U']\ci{A_{p}(\R\times\R)}^{1/(p-1)},
\end{equation*}
\begin{equation*}
[U']\ci{A_{p'}(\R\times\R)}\sim_{p,d}[U]\ci{A_{p}(\R\times\R)}^{1/(p-1)},\qquad [W']\ci{A_{p'}(\R\times\R)}\sim_{p,d}[W]\ci{A_{p}(\R\times\R)}^{1/(p-1)},
\end{equation*}
we obtain additional weighted estimates for $T$ itself.
\end{rem}

\section{Matrix-weighted Fefferman--Stein inequalities}
\label{s: vector-valued maximal function}

For the purpose of obtaining $L^{p}$ matrix-weighted bounds for mixed operators, we need to study vector-valued extensions of the (one-parameter) Christ--Goldberg  maximal function.

Let $\cD$ be any dyadic grid in $\R^n$. The (dyadic version of the) classical Fefferman--Stein vector-valued maximal inequality says
\begin{equation}
\label{Fefferman-Stein}
\bigg\Vert\bigg(\sum_{k=1}^{\infty}(M\ci{\cD}f_k)^{q}\bigg)^{1/q}\bigg\Vert\ci{L^{r}}\leq C(q,r,n)\bigg\Vert\bigg(\sum_{k=1}^{\infty}|f_k|^{q}\bigg)^{1/q}\bigg\Vert\ci{L^{r}},\qquad\forall 1<q,r<\infty,
\end{equation}
where $M\ci{\cD}$ is the usual (unweighted) dyadic Hardy--Littlewood maximal function on $\R^n$,
\begin{equation*}
M\ci{\cD}f:=\sup_{I\in\cD}\La|f|\Ra\ci{I}\1\ci{I},\qquad f\in L^1\ti{loc}(\R^n).
\end{equation*}
We are interested in tracking more precisely the constant $C(q,r,n)$ in this inequality. Examining the proof of \eqref{Fefferman-Stein} in the case $1<r<q$ given e.g. in \cite[Theorem 1.6]{tao-notes5}, we see that if $1<r<q$, then one can take
\begin{equation}
\label{Fefferman-Stein-constant-r<q}
C(q,r,n)=2\bigg(r'+\frac{r}{q-r}\bigg)^{1/r}C_1(q,n)^{\theta}C_{2}(q)^{1-\theta},
\end{equation}
where (for fixed $n$) $C_1(q,n),C_2(q)$ are continuous functions of $q\in(1,\infty)$, and $\theta\in(0,1)$ is given by
\begin{equation*}
\frac{1}{r}=\frac{\theta}{1}+\frac{1-\theta}{q}.
\end{equation*}
We remark that this expression arises through Marcinkiewitz interpolation for sublinear vector-valued operators.

Let now $1<p<2$, $0<c(p,d)<\infty$ and $0<\e<c(p,d)$. Set
\begin{equation*}
a=\frac{p(1+\e)}{1+p\e},\qquad q:=\frac{2}{a},\qquad r:=\frac{p}{a}.
\end{equation*}
Note that $1<r<q$ and $a>1$. We claim that
\begin{equation*}
C(q,r,n)\leq \e^{-1/r}c_1(p,d,n).
\end{equation*}
Indeed, we notice that
\begin{equation*}
1<\frac{2}{p}<q<2(1+p\cdot c(p,d))
\end{equation*}
Since the functions $C_1(\fdot,n),C_2$ are continuous and $0<\theta<1$, we deduce that
\begin{equation*}
C_1(q,n)^{\theta}C_2(q)^{1-\theta}\leq c_0(p,d,n).
\end{equation*}
Moreover, we have $b(p):=\frac{r}{q-r}=\frac{p}{2-p}>0$ and $r,r'>1$, so
\begin{align*}
\bigg(r'+\frac{r}{q-r}\bigg)^{1/r}&=\bigg(r'+b(p)\bigg)^{1/r}<(r'((b(p))^2+3b(p)+2))^{1/r}\\
&\leq(r')^{1/r}[(b(p))^2+3b(p)+2]\lesssim_{p}(r')^{1/r}.
\end{align*}
Finally, we have
\begin{equation*}
r'=\frac{1+p\e}{(p-1)\e}\leq\frac{1+p\cdot c(p,d)}{p-1}\cdot\e^{-1},
\end{equation*}
and also $1<r<p$, so
\begin{equation*}
\bigg(\frac{1+p\cdot c(p,d)}{p-1}\bigg)^{1/r}\leq\max\bigg(\frac{1+p\cdot c(p,d)}{p-1},\bigg(\frac{1+p\cdot c(p,d)}{p-1}\bigg)^{1/p}\bigg),
\end{equation*}
yielding the claim.

\begin{prop}
\label{p: vector valued maximal function}
Let $1<p\leq 2$, and let $W$ be a $d\times d$-matrix valued $\cD$-dyadic $A_p$ weight on $\R^n$. Then, there holds
\begin{equation*}
\bigg\Vert\bigg(\sum_{k=1}^{\infty}(\wt{M}\ci{\cD,W}f_k)^{2}\bigg)^{1/2}\bigg\Vert\ci{L^{p}}\lesssim_{p,n,d}[W]\ci{A_p,\cD}^{1/(p-1)}\bigg\Vert\bigg(\sum_{k=1}^{\infty}|W^{1/p}f_k|^{2}\bigg)^{1/2}\bigg\Vert\ci{L^{p}}.
\end{equation*}
\end{prop}

\begin{proof}
If $p=2$, then this follows immediately from the known weighted bounds for $\wt{M}\ci{\cD,W}$, so we assume $p<2$. We will employ the very useful trick from \cite{isralowitz-kwon-pott} that was also adapted for the proof of Proposition \ref{p: weighted bound modified strong C-G}. Namely, set
\begin{equation*}
\e:=\frac{1}{2^{n+5}C(p,d)[W]\ci{A_p,\cD}^{p'/p}},
\end{equation*}
where $0<C(p,d)<\infty$ satisfies
\begin{equation*}
[|W^{-1/p}\cW\ci{Q}|^{p'}]\ci{A_{p'},\cD}\leq C(p,d)[W]\ci{A_p,\cD}^{p'/p},\qquad\forall Q\in\cD.
\end{equation*}
Observe that this results from applying Lemma~\ref{l: first scalar} with $W'$ in the place of $W$, the exponent $p'$ in the place of the exponent $p$ and the invertible matrix $A:=\cW\ci{R}$.
Isralowitz--Kwon--Pott \cite{isralowitz-kwon-pott} show that
\begin{equation}
\label{dominate weighted maximal function by unweighted, one parameter}
\wt{M}\ci{\cD,W}f\lesssim_{p,d,n}[W]\ci{A_p,\cD}^{1/p}M\ci{a,\cD}(|W^{1/p}f|),
\end{equation}
where $a:=\frac{p(1+\e)}{1+p\e}$ and
\begin{equation*}
M\ci{a,\cD}g:=\sup_{Q\in\cD}\La|g|^{a}\Ra\ci{Q}^{1/a}\1\ci{Q}.
\end{equation*}
We note that \eqref{dominate weighted maximal function by unweighted, one parameter} is just the one-parameter analog of \eqref{dominate weighted maximal function by unweighted}. Thus
\begin{equation*}
\bigg\Vert\bigg(\sum_{k=1}^{\infty}(\wt{M}\ci{\cD,W}f_k)^{2}\bigg)^{1/2}\bigg\Vert\ci{L^{p}}\lesssim_{p,d,n}[W]\ci{A_p,\cD}^{1/p}\bigg\Vert\bigg(\sum_{k=1}^{\infty}(M\ci{a,\cD}(|W^{1/p}f_k|)^{2}\bigg)^{1/2}\bigg\Vert\ci{L^{p}}.
\end{equation*}
Set $q:=2/a$, $r:=p/a$. Notice that $0<\e<c(p,d)$ for some $0<c(p,d)<\infty$. Thus, we have
\begin{align*}
&\bigg\Vert\bigg(\sum_{k=1}^{\infty}(M\ci{a,\cD}(|W^{1/p}f_k|)^{2}\bigg)^{1/2}\bigg\Vert\ci{L^{p}}^{p}=\bigg\Vert\bigg(\sum_{k=1}^{\infty}(M\ci{\cD}(|W^{1/p}f_k|^{a}))^{2/a}\bigg)^{p/2}\bigg\Vert\ci{L^{1}}\\
&=\bigg\Vert\bigg(\sum_{k=1}^{\infty}(M\ci{\cD}(|W^{1/p}f_k|^{a}))^{q}\bigg)^{r/q}\bigg\Vert\ci{L^{1}}=\bigg\Vert\bigg(\sum_{k=1}^{\infty}(M\ci{\cD}(|W^{1/p}f_k|^{a}))^{q}\bigg)^{1/q}\bigg\Vert\ci{L^{r}}^{r}\\
&\leq\e^{-1}c_1(p,d,n)^{r}\bigg\Vert\bigg(\sum_{k=1}^{\infty}(|W^{1/p}f_k|^{a})^{q}\bigg)^{1/q}\bigg\Vert\ci{L^{r}}^{r}\sim_{p,n,d}[W]\ci{A_p,\cD}^{1/(p-1)}\bigg\Vert\bigg(\sum_{k=1}^{\infty}|W^{1/p}f_k|^{2}\bigg)^{1/2}\bigg\Vert\ci{L^{p}}^{p},
\end{align*}
where in the last $\sim_{p,d,n}$ we used the definition of $\e$ and the fact that $1<r<p$.
\end{proof}

Unfortunately, we have not been able to extend the bound of Proposition \ref{p: vector valued maximal function} to the regime $p>2$, because in this case $a:=\frac{p(1+\e)}{1+p\e}>2$ (for small enough $\e$, i.e. large enough $[W]\ci{A_p,\cD}$). Moreover, the argument used in the proof of Proposition \ref{p: vector valued maximal function} seems not to be able to yield vector-valued estimates for the Christ--Goldberg maximal function $M^{W}\ci{\cD}$ itself instead of its modified version $\wt{M}^{W}\ci{\cD}$.

\begin{quest}
\label{quest: vector valued maximal function}
Let $p>2$, and let $W$ be a $d\times d$-matrix valued $\cD$-dyadic $A_p$ weight on $\R^n$. Is it true that there exists a finite constant $C>0$, depending only on depends only on $p,d,n$ and $W$, such that
\begin{equation*}
\bigg\Vert\bigg(\sum_{k=1}^{\infty}(\wt{M}\ci{\cD,W}f_k)^{2}\bigg)^{1/2}\bigg\Vert\ci{L^{p}}\leq C\bigg\Vert\bigg(\sum_{k=1}^{\infty}|W^{1/p}f_k|^{2}\bigg)^{1/2}\bigg\Vert\ci{L^{p}}?
\end{equation*}
\end{quest}

\begin{quest}
\label{quest: vector valued maximal function-not modified}
Let $1<p<\infty$, and let $W$ be a $d\times d$-matrix valued $\cD$-dyadic $A_p$ weight on $\R^n$. Is it true that there exists a finite constant $C>0$, depending only on depends only on $p,d,n$ and $W$, such that
\begin{equation*}
\bigg\Vert\bigg(\sum_{k=1}^{\infty}(M\ci{\cD,W}f_k)^{2}\bigg)^{1/2}\bigg\Vert\ci{L^{p}}\leq C\bigg\Vert\bigg(\sum_{k=1}^{\infty}|W^{1/p}f_k|^{2}\bigg)^{1/2}\bigg\Vert\ci{L^{p}}?
\end{equation*}
\end{quest}

\section{Handling the non-cancellative terms}
\label{s: non-cancellative}

In this section we investigate $L^p$ matrix-weighted bounds for general (not necessarily paraproduct-free) Journ\'e operators. In view of Section \ref{s: cancellative}, it suffices only to estimate the non-cancellative terms in Martikainen's representation \eqref{martikainen representation}, which are all of paraproduct form. To that goal, we will first need to estimate several biparameter operators of mixed type.

\subsection{Auxiliary mixed operators}

Let $\bfD=\cD^1\times\cD^2$ be any product dyadic grid in $\R^2$. Let $1<p<\infty$, and let $W$ be any $\bfD$-dyadic biparameter $A_p$ weight on $\R\times\R$. For a.e.~$x_2\in\R$, we denote by $\cW\ci{x_2,I}$ the reducing operator of the weight $W_{x_2}(x_1):=W(x_1,x_2)$, $x_1\in\R$ over the interval $I\in\cD$ with respect to the exponent $p$, and we also define
\begin{equation*}
W\ci{I}(x_2):=\cW\ci{x_2,I}^{p}.
\end{equation*}
Moreover, for a.e.~$x_1\in\R$, we denote by $\cW\ci{x_1,J}$ the reducing operator of the weight $W_{x_1}(x_2):=W(x_1,x_2)$, $x_2\in\R$ over the interval $J\in\cD$ with respect to the exponent $p$, and we also define
\begin{equation*}
W\ci{J}(x_1):=\cW\ci{x_1,J}^{p}.
\end{equation*}

\subsubsection{Mixed square function-maximal function operators}

For a.e.~$x=(x_1,x_2)\in\R\times\R$, define
\begin{equation*}
[S\wt{M}]\ci{\bfD,W}f(x):=\bigg(\sum_{I\in\cD^1}(\sup_{J\in\cD^2}|\cW\ci{x_1,J}\La f\ci{I}^1{\Ra}\ci{J}|\1\ci{J}(x_2))^2\frac{\1\ci{I}(x_1)}{|I|}\bigg)^{1/2}
\end{equation*}
\begin{equation*}
[\wt{M}S]\ci{\bfD,W}f(x):=\bigg(\sum_{J\in\cD^2}(\sup_{I\in\cD^1}|\cW\ci{x_2,I}\La f\ci{J}^2{\Ra}\ci{I}|\1\ci{I}(x_1))^2\frac{\1\ci{J}(x_2)}{|J|}\bigg)^{1/2},
\end{equation*}
\begin{equation*}
[\wt{S}M]\ci{\bfD,W}f(x)=\bigg(\sum_{I\in\cD^1}(\sup_{J\in\cD^2}|\cW\ci{x_2,I}\La f\ci{I}^1\Ra\ci{J}|\1\ci{J}(x_2))^2\frac{\1\ci{I}(x_1)}{|I|}\bigg)^{1/2},
\end{equation*}
\begin{equation*}
[M\wt{S}]\ci{\bfD,W}f(x):=\bigg(\sum_{J\in\cD^2}(\sup_{I\in\cD^1}|\cW\ci{x_1,J}\La f\ci{J}^2\Ra\ci{I}|\1\ci{I}(x_1))^2\frac{\1\ci{J}(x_2)}{|J|}\bigg)^{1/2}.
\end{equation*}
Intuitively, $[S\wt{M}]\ci{\bfD,W}$ acts like a one-parameter matrix-weighted square function (with the weight introduced in a pointwise fashion) on the first variable, and as a one-parameter matrix-weighted maximal function (with the weight introduced in terms of its reducing operators) on the second variable. A similar symmetric interpretation holds for $[\wt{M}S]\ci{\bfD,W}$. Moreover, similar interpretations hold for $[\wt{S}M]\ci{\bfD,W}$ and its symmetric counterpart $[M\wt{S}]\ci{\bfD,W}$.

Estimating these operators for $p=2$ is very easy.

\begin{lm}
\label{l: SM bound p = 2}
Assume $p=2$. Let $T$ be any of the above four operators. Then, there holds
\begin{equation*}
\Vert T\Vert\ci{L^2(W)\rightarrow L^2}\lesssim_{d}[W]\ci{A_2,\bfD}^2.
\end{equation*}
\end{lm}

\begin{proof}
Let us prove this, for example, for $[\wt{S}M]\ci{\bfD,W}$, the proof for the other three operators being similar. We have
\begin{align*}
\Vert [\wt{S}M]\ci{\bfD,W}f\Vert\ci{L^{2}}^{2}&=\int_{\R^2}\sum_{I\in\cD^1}(\sup_{J\in\cD^2}|\cW\ci{x_2,I}\La f\ci{I}^1{\Ra}\ci{J}|\1\ci{J}(x_2))^2\frac{\1\ci{I}(x_1)}{|I|}\mathd x_1\mathd x_2\\
&=\sum_{I\in\cD^1}\Vert\sup_{J\in\cD^2}|\cW\ci{x_2,I}\La f\ci{I}^1{\Ra}\ci{J}|\1\ci{J}\Vert\ci{L^2(\R)}^2=
\sum_{I\in\cD^1}\Vert \wt{M}\ci{\cD^2,W\ci{I}}(f^{1}\ci{I})\Vert\ci{L^2(\R)}^2\\
&\lesssim_{d}\sum_{I\in\cD^1}[W\ci{I}]\ci{A_2,\cD^2}^{2}\Vert f^{1}\ci{I}\Vert\ci{L^2(W\ci{I})}^2\lesssim_{d}[W]\ci{A_2,\cD}^{2}\sum_{I\in\cD^1}\Vert f^{1}\ci{I}\Vert\ci{L^2(W\ci{I})}^2\\
&=[W]\ci{A_2,\cD}^{2}\sum_{I\in\cD^1}\int_{\R}|\cW\ci{x_2,I}f\ci{I}^{1}(x_2)|^{2}\mathd x_2\\
&=[W]\ci{A_2,\cD}^{2}\int_{\R}\bigg(\int_{\R}\sum_{I\in\cD^1}|\cW\ci{x_2,I}f\ci{I}^{1}(x_2)|^{2}\frac{\1\ci{I}(x_1)}{|I|}\mathd x_1\bigg)\mathd x_2\\
&=[W]\ci{A_2,\cD}^{2}\int_{\R}\Vert \wt{S}\ci{\cD^1,W(\fdot,x_2)}(f(\fdot,x_2))\Vert\ci{L^2(\R)}^2\mathd x_2\\
&\lesssim_{d}[W]\ci{A_2,\bfD}^{4}\int_{\R}\Vert f(\fdot,x_2)\Vert\ci{L^2(W(\fdot,x_2))}^2=[W]\ci{A_2,\bfD}^{4}\Vert f\Vert\ci{L^2(W)}^2,
\end{align*}
where in the second $\lesssim_{d}$ we applied \eqref{weighted bound modified C-G}, in the third $\lesssim_{d}$ we applied \eqref{uniform domination of characteristics of averages}, and in the last $\lesssim_{d}$ we applied \eqref{upper bound square functions p=2} coupled with Lemma \ref{l: two-weight-biparameter A_p implies uniform A_p in each coordinate}, concluding the proof.
\end{proof}

Obtaining bounds for general $1<p<\infty$ is much harder.

\begin{lm}
\label{l: SM bound}
\item[(1)] Assume $1<p\leq 2$. Then, there holds
\begin{equation*}
\Vert [S\wt{M}]\ci{\bfD,W}\Vert\ci{L^{p}(W)\rightarrow L^{p}}\lesssim_{p,d}[W]\ci{A_p,\bfD}^{\gamma(p)+\frac{1}{p-1}}.
\end{equation*}
The same estimate holds for the symmetric counterpart $[\wt{M}S]\ci{\bfD,W}$.

\item[(2)] There holds
\begin{align*}
\Vert[\wt{S}M]\ci{\bfD,W}f\Vert\ci{L^{p}}^{p}\lesssim_{p,d}[W]\ci{A_p,\bfD}\int_{\R}\bigg(\int_{\R}\bigg(\sum_{I\in\cD^1}(M\ci{\cD^2,W(x_1,\fdot)}(Q^{1}\ci{I}f)(x_2))^2\bigg)^{p/2}\mathd x_2\bigg)\mathd x_1.
\end{align*}
A symmetric estimate holds for the symmetric counterpart $[M\wt{S}]\ci{\bfD,W}$.
\end{lm}

\begin{proof}
\item[(1)] For a.e.~$x_1\in\R$, applying first \eqref{trivial bound} and then Proposition \ref{p: vector valued maximal function}, we obtain
\begin{align*}
&\int_{\R}([S\wt{M}]\ci{\bfD,W}f(x_1,x_2))^p\dd x_2=\int_{\R}\sum_{I\in\cD^1}\bigg((\sup_{J\in\cD^2}|\cW\ci{x_1,J}\La f\ci{I}^{1}\Ra\ci{J}|\1\ci{J}(x_2))^2\frac{\1\ci{I}(x_1)}{|I|}\bigg)^{p/2}\dd x_2\\
&=\int_{\R}\bigg(\sum_{I\in\cD^1}(\sup_{J\in\cD^2}|W\ci{J}(x_1)^{1/p}\La f\ci{I}^1\Ra\ci{J}|\1\ci{J}(x_2))^2\frac{\1\ci{I}(x_1)}{|I|}\bigg)^{p/2}\mathd x_2\\
&=\int_{\R}\bigg(\sum_{I\in\cD^1}(\sup_{J\in\cD^2}|\cW\ci{x_1,J}\La f\ci{I}^1\Ra\ci{J}|\1\ci{J}(x_2))^2\frac{\1\ci{I}(x_1)}{|I|}\bigg)^{p/2}\mathd x_2\\
&=\int_{\R}\bigg(\sum_{I\in\cD^1}(\wt{M}\ci{\cD^2,W(x_1,\fdot)}(Q^1\ci{I}f)(x_2))^2\bigg)^{p/2}\mathd x_2\\
&\lesssim_{p,d}[W_{x_1}]\ci{A_p,\cD^2}^{p/(p-1)}\int_{\R}\bigg(\sum_{I\in\cD^1}|W(x_1,x_2)^{1/p}Q^1\ci{I}f(x_2)|^2\bigg)^{p/2}\mathd x_2\\
&\lesssim_{p,d}[W]\ci{A_p,\bfD}^{p/(p-1)}\int_{\R}\bigg(\sum_{I\in\cD^1}|W(x_1,x_2)^{1/p}Q^1\ci{I}f(x_2)|^2\bigg)^{p/2}\mathd x_2.
\end{align*}
Therefore, we have
\begin{align*}
\Vert [S\wt{M}]\ci{\bfD,W}f\Vert\ci{L^{p}}^{p}&\lesssim_{p,d}[W]\ci{A_p,\bfD}^{\frac{p}{p-1}}\int_{\R}\bigg(\int_{\R}\bigg(\sum_{I\in\cD^1}|W(x_1,x_2)^{1/p}Q^1\ci{I}f(x_2)|^2\bigg)^{p/2}\mathd x_2\bigg)\mathd x_1\\
&=[W]\ci{A_p,\bfD}^{\frac{p}{p-1}}\\
&\int_{\R}\bigg(\int_{\R}\bigg(\sum_{I\in\cD^1}(W(x_1,x_2)^{1/p}Q\ci{I}(f(\fdot,x_2))(x_1))^2\bigg)^{p/2}\mathd x_1\bigg)\mathd x_2\\
&=[W]\ci{A_p,\bfD}^{\frac{p}{p-1}}\int_{\R}\Vert S\ci{\cD^1,W_{x_2}}(f(\fdot,x_2))\Vert\ci{L^{p}(\R)}^{p}\mathd x_2\\
&\lesssim_{p,d}[W]\ci{A_p,\bfD}^{\frac{p}{p-1}+p\gamma(p)}\int_{\R}\bigg(\int_{\R}|W(x_1,x_2)|^{1/p}f(x_1,x_2)|^{p}\mathd x_1\bigg)\mathd x_2\\
&=[W]\ci{A_p,\bfD}^{\frac{p}{p-1}+p\gamma(p)}\Vert f\Vert\ci{L^{p}(W)}^{p},
\end{align*}
concluding the proof.

\item[(2)] We have
\begin{equation*}
\Vert [\wt{S}M]\ci{\bfD,W}f\Vert\ci{L^{p}}^{p}=\int_{\R}\bigg(\int_{\R}\bigg(\sum_{I\in\cD^1}(\sup_{J\in\cD^2}|\cW\ci{x_2,I}\La f\ci{I}^1\Ra\ci{J}|\1\ci{J}(x_2))^2\frac{\1\ci{I}(x_1)}{|I|}\bigg)^{p/2}\mathd x_1\bigg)\mathd x_2.
\end{equation*}
For a.e. $x_2\in\R$, using first Lemma \ref{l: second scalar} and then Lemma \ref{l: uniform A_p in each coordinate implies biparameter A_p} we have
\begin{align*}
\sup_{J\in\cD^2}|\cW\ci{x_2,I}\La f\ci{I}^1\Ra\ci{J}|\1\ci{J}(x_2)&\lesssim_{p,d}\sup_{J\in\cD^2}[W_{x_2}]\ci{A_p,\cD^1}^{1/p}(|W_{x_2}^{1/p}\La f\ci{I}^1\Ra\ci{J}|)\ci{I}\1\ci{J}(x_2)\\
&\lesssim_{p,d}[W]\ci{A_p,\bfD}^{1/p}(\sup_{J\in\cD^2}|W_{x_2}^{1/p}\La f\ci{I}^1\Ra\ci{J}|\1\ci{J}(x_2))\ci{I}.
\end{align*}
In the last $\lesssim_{p,d}$, we used \eqref{uniform domination of characteristics of averages}, as well as the fact that the supremum of the integrals is dominated by the integral of the supremum. So, for a.e. $x_2\in\R$ we have
\begin{align*}
&\int_{\R}\bigg(\sum_{I\in\cD^1}(\sup_{J\in\cD^2}|\cW\ci{x_2,I}\La f\ci{I}^1\Ra\ci{J}|\1\ci{J}(x_2))^2\frac{\1\ci{I}(x_1)}{|I|}\bigg)^{p/2}\mathd x_1\\
&\lesssim_{p,d}[W]\ci{A_p,\bfD}\int_{\R}\bigg(\sum_{I\in\cD^1}(\sup_{J\in\cD^2}|W_{x_2}^{1/p}\La f\ci{I}^1\Ra\ci{J}|\1\ci{J}(x_2))\ci{I}^2\frac{\1\ci{I}(x_1)}{|I|}\bigg)^{p/2}\mathd x_1.
\end{align*}
For a.e. $x_2\in\R$, applying the same duality trick as in the proof of Lemma \ref{l: dominate TL square function by square function} we obtain
\begin{align*}
&\int_{\R}\bigg(\sum_{I\in\cD^1}\La\sup_{J\in\cD^2}|W_{x_2}^{1/p}\La f\ci{I}^1\Ra\ci{J}|\1\ci{J}(x_2)\Ra\ci{I}^2\frac{\1\ci{I}(x_1)}{|I|}\bigg)^{p/2}\mathd x_1\\
&\lesssim_{p,d}\int_{\R}\bigg(\sum_{I\in\cD^1}(\sup_{J\in\cD^2}|W_{x_2}(x_1)^{1/p}\La f\ci{I}^1\Ra\ci{J}|\1\ci{J}(x_2))^2\frac{\1\ci{I}(x_1)}{|I|}\bigg)^{p/2}\mathd x_1.
\end{align*}
Thus, by Fubini--Tonelli we have
\begin{align*}
&\Vert[\wt{S}M]\ci{\bfD,W}f\Vert\ci{L^{p}}^{p}\\
&\lesssim_{p,d}[W]\ci{A_p,\bfD}\int_{\R}\bigg(\int_{\R}\bigg(\sum_{I\in\cD^1}(\sup_{J\in\cD^2}|W_{x_2}(x_1)^{1/p}\La f\ci{I}^1\Ra\ci{J}|\1\ci{J}(x_2))^2\frac{\1\ci{I}(x_1)}{|I|}\bigg)^{p/2}\mathd x_2\bigg)\mathd x_1\\
&=[W]\ci{A_p,\bfD}\int_{\R}\bigg(\int_{\R}\bigg(\sum_{I\in\cD^1}(\sup_{J\in\cD^2}|W_{x_1}(x_2)^{1/p}\La f\ci{I}^1\Ra\ci{J}|\1\ci{J}(x_2))^2\frac{\1\ci{I}(x_1)}{|I|}\bigg)^{p/2}\mathd x_2\bigg)\mathd x_1\\
&=[W]\ci{A_p,\bfD}\int_{\R}\bigg(\int_{\R}\bigg(\sum_{I\in\cD^1}(M\ci{\cD^2,W(x_1,\fdot)}(Q^{1}\ci{I}f)(x_2))^2\bigg)^{p/2}\mathd x_2\bigg)\mathd x_1.
\end{align*}
Clearly, one would follow a similar strategy to bound $[M\wt{S}]\ci{\bfD,W}$.
\end{proof}

\begin{rem}
\label{r: incomplete sm bound}
To extend part (1) of Lemma \ref{l: SM bound} to the whole range $1<p<\infty$, we would need to answer Question \ref{quest: vector valued maximal function} to the positive. Moreover, to deduce actual matrix-weighted bounds from part (2) of Lemma \ref{l: SM bound}, we would need to answer Question \ref{quest: vector valued maximal function-not modified} to the positive.
\end{rem}

\subsubsection{Mixed shifted square function-maximal functions operators}

Let $i$ be any nonnegative integer. For a.e.~$x=(x_1,x_2)\in\R\times\R$, we define
\begin{equation*}
[\wt{S}^{i}M]\ci{\bfD,W}f(x):=\bigg(\sum_{R_1\in\cD^1}\bigg(\sum_{P_1\in\ch_{i}(R_1)}(\sup_{R_2\in\cD^2}|\cW\ci{x_2,R_1}\La f^{1}\ci{P_1}\Ra\ci{R_2}|\1\ci{R_2}(x_2))\bigg)^2\frac{\1\ci{R_1}(x_1)}{|R_1|}\bigg)^{1/2},
\end{equation*}
\begin{equation*}
[M\wt{S}^{i}]\ci{\bfD,W}f(x):=\bigg(\sum_{R_2\in\cD^2}\bigg(\sum_{P_2\in\ch_{i}(R_2)}(\sup_{R_1\in\cD^1}|\cW\ci{x_1,R_2}\La f^{2}\ci{P_2}\Ra\ci{R_1}|\1\ci{R_1}(x_1))\bigg)^2\frac{\1\ci{R_2}(x_2)}{|R_2|}\bigg)^{1/2}.
\end{equation*}
Intuitively, $[\wt{S}^{i}M]\ci{\bfD,W}$ acts like a one-parameter matrix-weighted shifted square function of complexity $i$ (with the weight introduced in terms of its reducing operators) on the first variable, and as a one-parameter matrix-weighted maximal function (with the weight introduced in a pointwise fashion) on the second variable. A similar symmetric interpretation holds for $[M\wt{S}^{i}]\ci{\bfD,W}$.

Estimating these operators for $p=2$ is straightforward.

\begin{lm}
\label{l: shifted square-maximal bound p = 2}
Assume $p=2$. Then, there holds
\begin{equation*}
\Vert [\wt{S}^{i}M]\ci{\bfD,W}\Vert\ci{L^2(W)\rightarrow L^2}\lesssim_{d}2^{i/2}[W]\ci{A_2,\bfD}^{5/2}.
\end{equation*}
The same bound is true for the symmetric counterpart $[M\wt{S}^{i}]\ci{\bfD,W}$.
\end{lm}

\begin{proof}
We will apply the same trick as in the proof of Lemma \ref{l: tl shifted-square-function-p=2}. We have
\begin{align*}
&\Vert [\wt{S}^{i}M]\ci{\bfD,W}f\Vert\ci{L^2}^2=
\int_{\R}\sum_{R_1\in\cD^1}\bigg(\sum_{P_1\in\ch_{i}(R_1)}(\sup_{R_2\in\cD^2}|\cW\ci{x_2,R_1}\La f^{1}\ci{P_1}\Ra\ci{R_2}|\1\ci{R_2}(x_2))\bigg)^2\mathd x_2\\
&\leq\int_{\R}\sum_{R_1\in\cD^1}\bigg(\sum_{P_1\in\ch_{i}(R_1)}(\sup_{R_2\in\cD^2}(|\cW\ci{x_2,R_1}\cW\ci{x_2,P_1}^{-1}|\cdot|\cW\ci{x_2,P_1}\La f^{1}\ci{P_1}\Ra\ci{R_2}|\1\ci{R_2}(x_2)))\bigg)^2\mathd x_2\\
&\leq\int_{\R}\sum_{R_1\in\cD^1}\bigg(\sum_{P_1\in\ch_{i}(R_1)}|\cW\ci{x_2,R_1}\cW\ci{x_2,P_1}^{-1}|^2\bigg)\bigg(\sum_{P_1\in\ch_{i}(R_1)}(\sup_{R_2\in\cD^2}|\cW\ci{x_2,P_1}\La f^{1}\ci{P_1}\Ra\ci{R_2}|\1\ci{R_2}(x_2))^2\bigg)\mathd x_2
\end{align*}
where we applied the Cauchy--Schwarz inequality in the last step. As in the proof of Lemma \ref{l: tl shifted-square-function-p=2} we have
\begin{equation*}
\sum_{P_1\in\ch_{i}(R_1)}|\cW\ci{x_2,R_1}\cW\ci{x_2,P_1}^{-1}|^2\lesssim_{d}2^{i}[W_{x_2}]\ci{A_2,\cD^1}\lesssim_{d}2^{i}[W]\ci{A_2,\bfD},
\end{equation*}
where we used Lemma \ref{l: two-weight-biparameter A_p implies uniform A_p in each coordinate} in the last $\lesssim_{d}$. Recall that $W\ci{P_1}(x_2)=\cW\ci{x_2,P_1}^{2}$, for a.e.~$x_2\in\R$, and that by \eqref{uniform domination of characteristics of averages} we have $[W\ci{P_1}]\ci{A_2,\cD^2}\lesssim_{d}[W]\ci{A_2,\bfD}$. Thus
\begin{align*}
\Vert [\wt{S}^{i}M]f\ci{\bfD,W}\Vert\ci{L^2}^2&\lesssim_{d}2^{i}[W]\ci{A_2,\bfD}\int_{\R}\sum_{R_1\in\cD^1}\sum_{P_1\in\ch_{i}(R_1)}(\sup_{R_2\in\cD^2}|\cW\ci{x_2,P_1}\La f^{1}\ci{P_1}\Ra\ci{R_2}|\1\ci{R_2}(x_2))^2\mathd x_2\\
&=2^{i}[W]\ci{A_2,\bfD}\sum_{P_1\in\cD^1}\Vert M\ci{\cD^2,W\ci{P_1}}(f\ci{P_1}^1)\Vert\ci{L^2(\R)}^2\\
&\lesssim_{d}2^{i}[W]\ci{A_2,\bfD}\sum_{P_1\in\cD^1}[W\ci{P_1}]\ci{A_2,\cD^2}^2\Vert f\ci{P_1}^1\Vert\ci{L^2(W\ci{P_1})}^2\\
&\lesssim_{d}2^{i}[W]\ci{A_2,\bfD}^{3}\sum_{P_1\in\cD^1}\int_{\R}|\cW_{x_2,P_1}f\ci{P_1}^1(x_2)|^2\mathd x_2\\
&=2^{i}[W]\ci{A_2,\bfD}^{3}\int_{\R}\bigg(\int_{\R}\sum_{P_1\in\cD^1}|\cW_{x_2,P_1}f\ci{P_1}^1(x_2)|^2\frac{\1\ci{P_1}(x_1)}{|P_1|}\mathd x_1\bigg)\mathd x_2\\
&=2^{i}[W]\ci{A_2,\bfD}^{3}\int_{\R}\Vert \wt{S}\ci{\cD^1,W_{x_2}}(f(\fdot,x_2))\Vert\ci{L^2(\R)}^2\mathd x_2\\
&\lesssim_{d}2^{i}[W]\ci{A_2,\bfD}^{5}\int_{\R}\Vert f(\fdot,x_2)\Vert\ci{L^2(W(\fdot,x_2))}^2\mathd x_2=[W]\ci{A_2,\bfD}^{5}\Vert f\Vert\ci{L^2(W)}^2.
\end{align*}
Thus
\begin{equation*}
\Vert [\wt{S}^{i}M]\ci{\bfD,W}\Vert\ci{L^2(W)\rightarrow L^2}\lesssim_{d}2^{i/2}[W]\ci{A_2,\bfD}^{5/2}.
\end{equation*}
The proof for $[M\wt{S}^{i}]\ci{\bfD,W}$ is symmetric.
\end{proof}

Bounding these operators for general $1<p<\infty$ is much harder.

\begin{lm}
\label{l: shifted square-maximal bound}
There holds
\begin{align*}
&\Vert [\wt{S}^{i}M]\ci{\bfD,W}f\Vert\ci{L^{p}}^{p}\\
&\lesssim_{p,d} i^{p}2^{pi/2}[W]\ci{A_p,\bfD}^{1+p\gamma(p)+\frac{p\gamma(p')}{p-1}}\int_{\R}\bigg(\int_{\R}\bigg(\sum_{P_1\in\cD^1}(M\ci{\cD^2,W(x_1,\fdot)}(Q^{1}\ci{P_1}(f))(x_2))^{2}\bigg)^{p/2}\mathd x_2\bigg)\mathd x_1.
\end{align*}
A symmetric estimate holds for the symmetric counterpart $[M\wt{S}^{i}]\ci{\bfD,W}$.
\end{lm}

\begin{proof}
By the Monotone Convergence Theorem, it suffices for all nonempty finite subsets $\cF_1,\cF_2$ of $\cD^1,\cD^2$ respectively to obtain bounds for the map
\begin{equation*}
[\wt{S}^{i}M]\ci{\cF_1,\cF_2,W}f(x):=\bigg(\sum_{R_1\in\cF_1}\bigg(\sum_{P_1\in\ch_{i}(R_1)}(\max_{R_2\in\cF_2}|\cW\ci{x_2,R_1}\La f^{1}\ci{P_1}\Ra\ci{R_2}|\1\ci{R_2}(x_2))\bigg)^2\frac{\1\ci{R_1}(x_1)}{|R_1|}\bigg)^{1/2}.
\end{equation*}
that are independent of $\cF_1,\cF_2$.

For every $P_1\in\cF_1$, it is easy to see that there exists an a.e.~defined map $J\ci{P_1}:\R\rightarrow\cF_2$, such that:
\begin{itemize}

\item the set $\lbrace x_2\in\R:~J\ci{P_1}(x_2)=R_2\rbrace$ is measurable, for all $R_2\in\cF_2$

\item there holds
\begin{equation*}
|\cW\ci{x_2,P_1^{(i)}}\La f^{1}\ci{P_1}\Ra\ci{J\ci{P_1}(x_2)}|=\max_{R_2\in\cF_2}|\cW\ci{x_2,P_1^{(i)}}\La f^{1}\ci{P_1}\Ra\ci{R_2}|,
\end{equation*}
for a.e.~$x_2\in\R$, where $P_1^{(i)}$ is the unique interval in $\cD^1$ containing $P_1$ such that $P_1\in\ch_{i}(P_1^{(i)})$.
\end{itemize}
For the sake of completeness, we give the details in the appendix. Thus
\begin{align*}
[\wt{S}^{i}M]\ci{\cF_1,\cF_2,W}f(x)&\lesssim\bigg(\sum_{R_1\in\cF_1}\bigg(\sum_{P_1\in\ch_{i}(R_1)}|\cW\ci{x_2,P_1^{(i)}}\La f^{1}\ci{P_1}\Ra\ci{J\ci{P_1}(x_2)}|\bigg)^2\frac{\1\ci{R_1}(x_1)}{|R_1|}\bigg)^{1/2}\\
&=\wt{S}^{i}\ci{\cD^1,W_{x_2}}(F_{x_2})(x_1),
\end{align*}
where
\begin{equation*}
F_{x_2}(y_1):=\sum_{R_1\in\cF_1}\sum_{P_1\in\ch_{i}(R_1)}\La f^{1}\ci{P_1}\Ra\ci{J\ci{P_1}(x_2)}h\ci{P_1}(y_1).
\end{equation*}
Thus, for a.e.~$x_2\in\R$, we have
\begin{align*}
\int_{\R}|[\wt{S}^{i}M]\ci{\cF_1,\cF_2,W}f(x_1,x_2)|^{p}\mathd x_1&\lesssim_{p,d}\Vert \wt{S}^{i}\ci{\cD^1,W_{x_2}}(F_{x_2})\Vert\ci{L^{p}(\R)}^{p}\\
&\lesssim_{p,d}i^{p}2^{ip/2}[W_{x_2}]\ci{A_p,\cD^1}^{1+p\gamma(p)}\Vert F_{x_2}\Vert\ci{L^{p}(W_{x_2})}^{p}\\
&\lesssim_{p,d}i^{p}2^{ip/2}[W]\ci{A_p,\bfD}^{1+p\gamma(p)}\Vert F_{x_2}\Vert\ci{L^{p}(W_{x_2})}^{p}.
\end{align*}
where in the second $\lesssim_{p,d}$ we used Lemma \ref{l: upper bound TL shifted square function one parameter}. It follows that
\begin{align*}
&\Vert [\wt{S}^{i}M]\ci{\cF_1,\cF_2,W}f\Vert\ci{L^{p}}^{p}\lesssim_{p,d}i^{p}2^{ip/2}[W]\ci{A_p,\bfD}^{1+p\gamma(p)}\int_{\R}\Vert F\ci{x_2}\Vert\ci{L^{p}(W_{x_2})}^{p}\mathd x_2\\
&\lesssim_{p,d}i^{p}2^{pi/2}[W]\ci{A_p,\bfD}^{1+p\gamma(p)+\frac{p\gamma(p')}{p-1}}\int_{\R}\Vert S\ci{\cD^1,W_{x_2}}(F\ci{x_2})\Vert\ci{L^{p}(\R)}^{p}\mathd x_2\\
&=i^{p}2^{pi/2}[W]\ci{A_p,\bfD}^{1+p\gamma(p)+\frac{p\gamma(p')}{p-1}}\int_{\R}\bigg(\int_{\R}\bigg(\sum_{P_1\in\cD^1}|W_{x_2}(x_1)^{1/p}(F\ci{x_2})\ci{P_1}|^{2}\frac{\1\ci{P_1}(x_1)}{|P_1|}\bigg)^{p/2}\mathd x_1\bigg)\mathd x_2\\
&=i^{p}2^{pi/2}[W]\ci{A_p,\bfD}^{1+p\gamma(p)+\frac{p\gamma(p')}{p-1}}\int_{\R}\bigg(\int_{\R}\bigg(\sum_{P_1\in\cD^1}|W(x_1,x_2)^{1/p}\La f^{1}\ci{P_1}\Ra\ci{J\ci{P_1(x_2)}}|^{2}\frac{\1\ci{P_1}(x_1)}{|P_1|}\bigg)^{p/2}\mathd x_1\bigg)\mathd x_2\\
&\leq i^{p}2^{pi/2}[W]\ci{A_p,\bfD}^{1+p\gamma(p)+\frac{p\gamma(p')}{p-1}}\int_{\R}\bigg(\int_{\R}\bigg(\sum_{P_1\in\cD^1}(M\ci{\cD^2,W(x_1,\fdot)}(Q^{1}\ci{P_1}(f))(x_2))^{2}\bigg)^{p/2}\mathd x_2\bigg)\mathd x_1.
\end{align*}
\end{proof}

\begin{rem}
\label{r: incomplete shifted square-maximal bounds}
To deduce actual weighted bounds from Lemma \ref{l: shifted square-maximal bound}, we would need to answer Question \ref{quest: vector valued maximal function-not modified} to the positive.
\end{rem}

\subsubsection{Mixed shifted square function-square function operators}

Let $i$ be any nonnegative integer. For a.e.~$x=(x_1,x_2)\in\R\times\R$, define
\begin{equation*}
[\wt{S}^{i}S]\ci{\bfD,W}f(x):=\bigg(\sum_{R_1\in\cD^1}\bigg(\sum_{P_1\in\ch_{i}(R_1)}\bigg(\sum_{R_2\in\cD^2}|\cW\ci{x_2,R_1}f\ci{P_1\times R_2}|^2\frac{\1\ci{R_2}(x_2)}{|R_2|}\bigg)^{1/2}\bigg)^2\frac{\1\ci{R_1}(x_1)}{|R_1|}\bigg)^{1/2},
\end{equation*}
\begin{equation*}
[S\wt{S}^{i}]\ci{\bfD,W}f(x):=\bigg(\sum_{R_2\in\cD^2}\bigg(\sum_{P_2\in\ch_{i}(R_2)}\bigg(\sum_{R_1\in\cD^1}|\cW\ci{x_1,R_2}f\ci{R_1\times P_2}|^2\frac{\1\ci{R_1}(x_1)}{|R_1|}\bigg)^{1/2}\bigg)^2\frac{\1\ci{R_2}(x_2)}{|R_2|}\bigg)^{1/2}.
\end{equation*}
Intuitively, $[\wt{S}^{i}S]\ci{\bfD,W}$ acts like a one-parameter matrix-weighted shifted square function of complexity $i$ (with the weight introduced in terms of its reducing operators) on the first variable, and as a one-parameter matrix-weighted square function (with the weight introduced in a pointwise fashion) on the second variable. A similar symmetric interpretation holds for $[S\wt{S}^{i}]\ci{\bfD,W}$.

Estimating these operators for $p=2$ is straightforward.

\begin{lm}
\label{l: shifted square-square bound p = 2}
Assume $p=2$. Then, there holds
\begin{equation*}
\Vert [\wt{S}^{i}S]\ci{\bfD,W}\Vert\ci{L^2(W)\rightarrow L^2}\lesssim_{d}2^{i/2}[W]\ci{A_2,\bfD}^{5/2}.
\end{equation*}
The same bound is true for the symmetric counterpart $[S\wt{S}^{i}]\ci{\bfD,W}$.
\end{lm}

The proof of Lemma \ref{l: shifted square-square bound p = 2} is almost identical to that of Lemma \ref{l: shifted square-maximal bound p = 2} (the only difference being that one obtains $S\ci{\cD^2,W\ci{P_1}}$ instead of $M\ci{\cD^2,W\ci{P_1}}$ in the course of the proof), so we omit it.

Bounding these operators for general $1<p<\infty$ is slightly harder.

\begin{lm}
\label{l: shifted square-square bound}
There holds
\begin{equation*}
\Vert [\wt{S}^{i}S]\ci{\bfD,W}\Vert\ci{L^{p}(W)\rightarrow L^{p}}\lesssim_{p,d}i2^{i/2}[W]\ci{A_p,\bfD}^{\frac{1}{p}+3\gamma(p)+\frac{\gamma(p')}{p-1}}.
\end{equation*}
The same bound is true for the symmetric counterpart $[S\wt{S}^{i}]\ci{\bfD,W}$.
\end{lm}

\begin{proof}
By the Monotone Convergence Theorem, it suffices for all nonempty finite subsets $\cF_1,\cF_2$ of $\cD^1,\cD^2$ respectively to obtain bounds for the map
\begin{equation*}
[\wt{S}^{i}S]\ci{\cF_1,\cF_2,W}f(x):=\bigg(\sum_{R_1\in\cF_1}\bigg(\sum_{P_1\in\ch_{i}(R_1)}\bigg(\sum_{R_2\in\cF_2}|\cW\ci{x_2,R_1}f\ci{P_1\times R_2}|^2\frac{\1\ci{R_2}(x_2)}{|R_2|}\bigg)^{1/2}\bigg)^2\frac{\1\ci{R_1}(x_1)}{|R_1|}\bigg)^{1/2}
\end{equation*}
that are independent of $\cF_1,\cF_2$.

For every $P_1\in\{P\in\ch_i(R_1)\colon R_1\in\cF_1\}$, by Khintchine's inequalities we have
\begin{equation*}
\bigg(\sum_{R_2\in\cF_2}|\cW\ci{x_2,P_1^{(i)}}f\ci{P_1\times R_2}|^2\frac{\1\ci{R_2}(x_2)}{|R_2|}\bigg)^{1/2}\sim\int_{\Omega}\bigg|\sum_{R_2\in\cF_2}\s\ci{R_2}(\omega)\cW\ci{x_2,P_1^{(i)}}f\ci{P_1\times R_2}h\ci{R_2}(x_2)\bigg|\mathd\bP(\omega),
\end{equation*}
where $P_1^{(i)}$ is the unique interval in $\cD^1$ such that $P_1\in\ch_{i}(P_1^{(i)})$. For every $P_1\in\cF_1$, it is easy to see that there exists an a.e.~defined measurable function $\f\ci{P_1}:\R\times\Omega\rightarrow\lbrace-1,1\rbrace$, such that
\begin{align*}
&\int_{\Omega}\bigg|\sum_{R_2\in\cF_2}\s\ci{R_2}(\omega)\cW\ci{x_2,P_1^{(i)}}f\ci{P_1\times R_2}h\ci{R_2}(x_2)\bigg|\mathd\bP(\omega)\\
&\sim_{d}\bigg|\cW\ci{x_2,P_1^{(i)}}\int_{\Omega}\f\ci{P_1}(x_2,\omega)\sum_{R_2\in\cF_2}\s\ci{R_2}(\omega)f\ci{P_1\times R_2}h\ci{R_2}(x_2)\mathd\bP(\omega)\bigg|.
\end{align*}
For the sake of completeness, we supply the details in the appendix. Therefore, setting
\begin{equation*}
F_{x_2}(y_1):=\sum_{R_1\in\cF_1}\sum_{P_1\in\ch_i(R_1)}\bigg(\int_{\Omega}\f\ci{P_1}(x_2,\omega)\sum_{R_2\in\cF_2}\s\ci{R_2}(\omega)f\ci{P_1\times R_2}h\ci{R_2}(x_2)\mathd\bP(\omega)\bigg) h\ci{P_1}(y_1),
\end{equation*}
we have
\begin{equation*}
[\wt{S}^{i}S]\ci{\bfD,W}f(x)\lesssim_{p,d}\wt{S}^{i}\ci{\cD^1,W_{x_2}}(F_{x_2})(x_1).
\end{equation*}
It follows that for a.e. $x_2\in\R$, we have
\begin{align*}
\int_{\R}|[\wt{S}^{i}S]\ci{\cF_1,\cF_2,W}f(x_1,x_2)|^{p}\mathd x_1&\lesssim_{p,d}\Vert \wt{S}^{i}\ci{\cD^1,W_{x_2}}(F_{x_2})\Vert\ci{L^{p}(\R)}^{p}\\
&\lesssim_{p,d}i^{p}2^{ip/2}[W_{x_2}]\ci{A_p,\cD^1}^{1+p\gamma(p)}\Vert F_{x_2}\Vert\ci{L^{p}(W_{x_2})}^{p}\\
&\lesssim_{p,d}i^{p}2^{ip/2}[W]\ci{A_p,\bfD}^{1+p\gamma(p)}\Vert F_{x_2}\Vert\ci{L^{p}(W_{x_2})}^{p},
\end{align*}
where in the second $\lesssim_{p,d}$ we used Lemma \ref{l: upper bound TL shifted square function one parameter}. Therefore
\begin{align*}
&\Vert [S^{i}\wt{S}]\ci{\cF_1,\cF_2,W}f\Vert\ci{L^{p}}^{p}
\lesssim_{p,d}[W]\ci{A_p,\bfD}^{1+p\gamma(p)}
i^{p}2^{pi/2}\int_{\R}\Vert F\ci{x_2}\Vert\ci{L^{p}(W_{x_2})}^{p}\mathd x_2\\
&\lesssim_{p,d}i^{p}2^{pi/2}[W]\ci{A_p,\bfD}^{1+p\gamma(p)+\frac{p\gamma(p')}{p-1}}\int_{\R}\Vert S\ci{\cD^1,W_{x_2}}(F\ci{x_2})(x_1)\Vert\ci{L^{p}(x_1)}^{p}\mathd x_2\\
&=i^{p}2^{pi/2}[W]\ci{A_p,\bfD}^{1+p\gamma(p)+\frac{p\gamma(p')}{p-1}}\int_{\R}\bigg(\int_{\R}\bigg(\sum_{P_1\in\cD^1}|W_{x_2}(x_1)^{1/p}(F\ci{x_2})\ci{P_1}|^{2}\frac{\1\ci{P_1}(x_1)}{|P_1|}\bigg)^{p/2}\mathd x_1\bigg)\mathd x_2\\
&\leq i^{p}2^{pi/2}[W]\ci{A_p,\bfD}^{1+p\gamma(p)+\frac{p\gamma(p')}{p-1}}\\
&\int_{\R}\bigg(\int_{\R}\bigg(\sum_{P_1\in\cD^1}\sum_{R_2\in\cD^2}|W_{x_2}(x_1)^{1/p}f\ci{P_1\times R_2}|^{2}\frac{\1\ci{P_1\times R_2}(x_1,x_2)}{|P_1|\cdot|R_2|}\bigg)^{p/2}\mathd x_1\bigg)\mathd x_2\\
&=i^{p}2^{i/2}[W]\ci{A_p,\bfD}^{1+p\gamma(p)+\frac{p\gamma(p')}{p-1}}\Vert S\ci{\bfD,W}f\Vert\ci{L^{p}(\R^2)}^{p}\lesssim_{p,d}i^{p}2^{pi/2}[W]\ci{A_p,\bfD}^{1+3p\gamma(p)+\frac{p\gamma(p')}{p-1}},
\end{align*}
concluding the proof.
\end{proof}

\subsection{Estimating non-cancellative terms}

Here we attempt to estimate the various non-cancellative terms in Martikainen's representation \eqref{martikainen representation}.
Note that all of these terms are of paraproduct form.
Also, similar two-weighted bounds for one-parameter matrix weights appear in \cite{isralowitz-commutators}.
We split them into different classes, following \cite[Subsection 7.4]{holmes-petermichl-wick}.

\subsubsection{Full standard paraproducts}
\label{s: pure biparameter paraproduct}

Let $\bfD=\cD^1\times\cD^2$ be a product dyadic grid in $\R\times\R$.
First of all, observe that all operators in this section (and in the following ones) are linear
with respect to their symbol $a.$
In particular, all estimates that we will show are homogeneous with respect to $\Vert a\Vert\ci{\text{BMO}\ci{\text{prod},\bfD}},$
and so we can consider without loss of generality that $\Vert a\Vert\ci{\text{BMO}\ci{\text{prod},\bfD}}\leq 1$ (see Subsection \ref{s:ProductSpaces} for the definition of the space $\text{BMO}\ci{\text{prod},\bfD}$). Consider the paraproduct
\begin{equation*}
\Pi_{a}^{(11)}f:=\sum_{R\in\bfD}a\ci{R}\La f\Ra\ci{R}h\ci{R}.
\end{equation*}

\begin{lm}
\label{l: bound full paraproduct}
Let $1<p<\infty$, and let $W,U$ be $d\times d$-matrix valued $\bfD$-dyadic biparameter $A_p$ weights with $[W,U']\ci{A_p,\bfD}<\infty$. Then, we have the bounds
\begin{equation*}
\Vert \Pi_{a}^{(11)}\Vert\ci{L^{p}(W)\rightarrow L^{p}(W)}\lesssim_{p,d}[W]\ci{A_p,\bfD}^{\frac{p+1}{p(p-1)}+\frac{1}{p}+\frac{2\gamma(p')}{p-1}},
\end{equation*}
as well as the two-weight bound
\begin{equation*}
\Vert \Pi_{a}^{(11)}\Vert\ci{L^{p}(U)\rightarrow L^{p}(W)}\lesssim_{p,d}[W,U']\ci{A_p,\bfD}^{1/p}[U]\ci{A_p,\bfD}^{\frac{p+1}{p(p-1)}+\frac{1}{p}}[W]\ci{A_p,\bfD}^{\frac{1}{p}+\frac{2\gamma(p')}{p-1}}.
\end{equation*}
Moreover, for $p=2$ we have the better bounds
\begin{equation*}
\Vert \Pi_{a}^{(11)}\Vert\ci{L^{2}(W)\rightarrow L^{2}(W)}\lesssim_{d}[W]\ci{A_2,\bfD}^{7/2},
\end{equation*}
and
\begin{equation*}
\Vert \Pi_{a}^{(11)}\Vert\ci{L^{2}(U)\rightarrow L^{2}(W)}\lesssim_{d}[W,U']\ci{A_2,\bfD}^{1/2}[U]\ci{A_2,\bfD}^{3/2}[W]\ci{A_2,\bfD}^{2}.
\end{equation*}
\end{lm}

\begin{proof}
We have
\begin{equation*}
(\Pi^{(11)}_af,g)=\sum_{R\in\bfD}a\ci{R}\La \La f\Ra\ci{R},g\ci{R}\Ra
\end{equation*}
Notice that
\begin{align*}
&|(\cU\ci{R})^{-1}(\cW'\ci{R})^{-1}|\lesssim_{p,d}|\cU_{R}'\cW\ci{R}|\lesssim_{p,d}[W,U']\ci{A_p,\bfD}^{1/p},
\end{align*}
where in the first $\lesssim_{p,d}$ we applied Lemma \ref{l: replace inverse by prime}. Then, by the well-known (scalar) $H^1\ci{\bfD}$-$\text{BMO}\ci{\text{prod},\bfD}$ duality (see Subsection \ref{s:ProductSpaces}) we have
\begin{align*}
&|(\Pi^{(11)}_af,g)|\lesssim\int_{\R^2}\bigg(\sum_{R\in\bfD}|\La \La f\Ra\ci{R},g\ci{R}\Ra|^2\frac{\1\ci{R}(x)}{|R|}\bigg)^{1/2}\mathd x\\
&\lesssim_{p,d} [W,U']\ci{A_p,\bfD}^{1/p}\int_{\R^2}\bigg(\sum_{R\in\bfD}(|\cU\ci{R}\La f\Ra\ci{R}|\cdot|\cW'\ci{R}g\ci{R}|)^2\frac{\1\ci{R}(x)}{|R|}\bigg)^{1/2}\mathd x\\
&\leq[W,U']\ci{A_p,\bfD}^{1/p}\int_{\R^2}\wt{M}\ci{\bfD,U}f(x)\bigg(\sum_{R\in\bfD}(|\cW'\ci{R}g\ci{R}|)^2\frac{\1\ci{R}(x)}{|R|}\bigg)^{1/2}\mathd x\\
&=[W,U']\ci{A_p,\bfD}^{1/p}\int_{\R^2}\wt{M}\ci{\bfD,U}f(x)\wt{S}\ci{\bfD,W'}g(x)\mathd x\leq[W,U']\ci{A_p,\bfD}^{1/p}\Vert \wt{M}\ci{\bfD,U}f\Vert\ci{L^{p}}\Vert \wt{S}\ci{\bfD,W'}g\Vert\ci{L^{p'}}.
\end{align*}
Similarly of course we obtain
\begin{align*}
|(\Pi^{(11)}_af,g)|
&\lesssim_{p,d}\Vert \wt{M}\ci{\bfD,W}f\Vert\ci{L^{p}}\Vert \wt{S}\ci{\bfD,W'}g\Vert\ci{L^{p'}}.
\end{align*}
It suffices now to use the bounds for $\wt{M}\ci{\bfD,U},\wt{M}\ci{\bfD,W}$ and $\wt{S}\ci{\bfD,W'}$ obtained in the previous sections.
\end{proof}

The weighted estimates for $\Pi_{a}^{(11)}$ imply by duality analogous weighted estimates for the paraproduct $\Pi_{a}^{(00)}$ given by
\begin{equation*}
\Pi_{a}^{(00)}f:=\sum_{R\in\bfD}a\ci{R}f\ci{R}\frac{\1\ci{R}}{|R|},
\end{equation*}
since $(\Pi_{a}^{(00)})^{\ast}=\Pi_{\bar{a}}^{(11)}$ in the (unweighted) $L^2(\R^2;\C^d)$ sense, where $\bar{a}$ denotes the complex conjugate of $a$. In particular, for $p=2$ we have
\begin{align*}
&\Vert \Pi_{a}^{(00)}\Vert\ci{L^2(U)\rightarrow L^2(W)}=\Vert \Pi_{\bar{a}}^{(11)}\Vert\ci{L^2(U')\rightarrow L^2(W')}\\
&\lesssim_{d}[U',(W')']\ci{A_2,\bfD}^{1/2}[W']\ci{A_2,\bfD}^{3/2}[U']\ci{A_2,\bfD}^{2}\sim_{d}[W,U']\ci{A_2,\bfD}^{1/2}[U]\ci{A_2,\bfD}^{2}[W]\ci{A_2,\bfD}^{3/2},
\end{align*}
and similarly
\begin{equation*}
\Vert \Pi_{a}^{(00)}\Vert\ci{L^2(U)\rightarrow L^2(W)}\lesssim_{d}[W]\ci{A_2,\bfD}^{7/2}.
\end{equation*}

\subsubsection{Full mixed paraproducts}

Let $\bfD=\cD^1\times\cD^2$ be a product dyadic grid in $\R\times\R$. Let $a\in L^1\ti{loc}(\R^2)$ such that $\Vert a\Vert\ci{\text{BMO}\ci{\text{prod},\bfD}}\leq 1$. Consider the paraproduct
\begin{equation*}
\Pi^{(01)}_{a}f:=\sum_{R\in\bfD}a\ci{R}\La f^{1}\ci{R_1}\Ra\ci{R_2}\frac{\1\ci{R_1}}{|R_1|}\otimes h\ci{R_2}.
\end{equation*}

\begin{lm}
\label{l: full mixed paraproduct}
Let $1<p<\infty$, and let $W,U$ be $\bfD$-dyadic biparameter $d\times d$-matrix valued $A_p$ weights on $\R^2$ with $[W,U']\ci{A_p,\bfD}<\infty$. Then, we have the decompositions
\begin{align*}
&|(\Pi_{a}^{(01)}f,g)|\lesssim_{p,d}[W,U']\ci{A_p,\bfD}^{1/p}\Vert [\wt{S}M]\ci{\bfD,U}f\Vert\ci{L^p}\Vert[\wt{M}S]\ci{\bfD,W'}g\Vert\ci{L^{p'}},
\end{align*}
\begin{equation*}
|(\Pi_{a}^{(01)}f,g)|\lesssim_{p,d}\Vert[\wt{S}M]\ci{\bfD,W}f\Vert\ci{L^{p}}\Vert[\wt{M}S]\ci{\bfD,W'}g\Vert\ci{L^{p'}},
\end{equation*}
\begin{align*}
&|(\Pi_{a}^{(01)}f,g)|\lesssim_{p,d} [W,U']\ci{A_p,\bfD}^{1/p}\Vert [S\wt{M}]\ci{\bfD,U}f\Vert\ci{L^p}\Vert[M\wt{S}]\ci{\bfD,W'}g\Vert\ci{L^{p'}},
\end{align*}
and
\begin{equation*}
|(\Pi_{a}^{(01)}f,g)|\lesssim_{p,d}\Vert[S\wt{M}]\ci{\bfD,W}f\Vert\ci{L^{p}}\Vert[M\wt{S}]\ci{\bfD,W'}g\Vert\ci{L^{p'}}.
\end{equation*}
In particular, for $p=2$ we deduce the one-weight bound
\begin{equation*}
\Vert \Pi_{a}^{(01)}\Vert\ci{L^{2}(W)\rightarrow L^{2}(W)}\lesssim_{d}[W]\ci{A_2,\bfD}^{4},
\end{equation*}
as well as the two-weight bound
\begin{equation*}
\Vert \Pi_{a}^{(01)}\Vert\ci{L^{2}(U)\rightarrow L^{2}(W)}\lesssim_{d}[W,U']\ci{A_2,\bfD}^{1/2}[U]\ci{A_2,\bfD}^{2}[W]\ci{A_2,\bfD}^{2}.
\end{equation*}
\end{lm}

\begin{proof}
We have
\begin{equation*}
(\Pi^{(01)}_af,g)=\sum_{R\in\bfD}a\ci{R}\La \La f^1\ci{R_1}\Ra\ci{R_2},\La g^2\ci{R_2}\Ra\ci{R_1}\Ra.
\end{equation*}
Thus, by the well-known (scalar) $H^1\ci{\bfD}$-$\text{BMO}\ci{\text{prod},\bfD}$ duality (see Subsection \ref{s:ProductSpaces}) we have
\begin{align*}
&|(\Pi_{a}^{(01)}f,g)|\lesssim\int_{\R^2}\bigg(\sum_{R\in\bfD}|\La \La f^1\ci{R_1}\Ra\ci{R_2},\La g^2\ci{R_2}\Ra\ci{R_1}\Ra|^2\frac{\1\ci{R}(x)}{|R|}\bigg)^{1/2}\mathd x.
\end{align*}
For a.e.~$x_2\in\R$, define the reducing operators $\cW'\ci{x_2,I},\cU\ci{x_2,I},\cU'\ci{x_2,I}$ corresponding to the weights $W',U,U'$ respectively in the same way as $\cW\ci{x_2,I}$. Notice that for a.e.~$x_2\in\R$, we have
\begin{align*}
&|(\cU\ci{x_2,R_1})^{-1}(\cW'\ci{x_2,R_1})^{-1}|\lesssim_{p,d}|\cU_{x_2,R_1}'\cW\ci{x_2,R_1}|\lesssim_{p,d}[W_{x_2},U'_{x_2}]\ci{A_p,\bfD}^{1/p}\lesssim_{p,d}[W,U']\ci{A_p,\bfD}^{1/p},
\end{align*}
where in the first $\lesssim_{p,d}$ we applied Lemma \ref{l: replace inverse by prime}, and in the third $\lesssim_{p,d}$ we applied Lemma \ref{l: two-weight-biparameter A_p implies uniform A_p in each coordinate}. It follows that
\begin{align*}
&|(\Pi_{a}^{(01)}f,g)|\lesssim_{p,d}[W,U']\ci{A_p,\bfD}^{1/p}\int_{\R^2}\bigg(\sum_{R\in\bfD}(|\cU\ci{x_2, R_1}\La f^1\ci{R_1}\Ra\ci{R_2}|\cdot|\cW'\ci{x_2,R_1}\La g^2\ci{R_2}\Ra\ci{R_1}|)^2\frac{\1\ci{R}(x)}{|R|}\bigg)^{1/2}\mathd x\\
&\leq[W,U']\ci{A_p,\bfD}^{1/p}\int_{\R^2}\bigg(\sum_{R\in\bfD}(F\ci{R_1}(x_2) G\ci{x_2,R_2}(x_1))^2\frac{\1\ci{R}(x)}{|R|}\bigg)^{1/2}\mathd x\\
&=[W,U']\ci{A_p,\bfD}^{1/p}\bigg(\int_{\R^2}[\wt{S}M]\ci{\bfD,U}f(x)[\wt{M}S]\ci{\bfD,W'}g(x)\mathd x\bigg)^{1/2}\\
&\leq [W,U']\ci{A_p,\bfD}^{1/p}\Vert [\wt{S}M]\ci{\bfD,U}f\Vert\ci{L^p}\Vert[\wt{M}S]\ci{\bfD,W'}g\Vert\ci{L^{p'}},
\end{align*}
where
\begin{equation*}
F\ci{R_1}(x_2):=\sup_{R_2\in\cD^2}|\cU\ci{x_2,R_1}\La f\ci{R_1}^1\Ra\ci{R_2}|\1\ci{R_2}(x_2),\qquad x_2\in\R,~R_1\in\cD^1,
\end{equation*}
\begin{equation*}
G\ci{x_2,R_2}(x_1):=\sup_{R_1\in\cD^2}|\cW'\ci{x_2,R_1}\La g\ci{R_2}^2\Ra\ci{R_1}|\1\ci{R_1}(x_1),\qquad x_1\in\R,~R_2\in\cD^2.
\end{equation*}
The other decompositions are proved similarly, and the bounds for the case $p=2$ follow from Lemma \ref{l: SM bound p = 2}.
\end{proof}

\begin{rem}
\label{r: incomplete mixed paraproducts}
Note that for $p'<2$, that is $p>2$, one can use the set of decompositions
\begin{align*}
|(\Pi_{a}^{(01)}f,g)|\lesssim_{p,d}[W,U']\ci{A_p,\bfD}^{1/p}\Vert [\wt{S}M]\ci{\bfD,U}f\Vert\ci{L^p}\Vert[\wt{M}S]\ci{\bfD,W'}g\Vert\ci{L^{p'}},
\end{align*}
\begin{equation*}
|(\Pi_{a}^{(01)}f,g)|\lesssim_{p,d}\Vert[\wt{S}M]\ci{\bfD,W}f\Vert\ci{L^{p}}\Vert[\wt{M}S]\ci{\bfD,W'}g\Vert\ci{L^{p'}},
\end{equation*}
while for $p<2$ one can use the set of decompositions
\begin{align*}
&|(\Pi_{a}^{(01)}f,g)|\lesssim_{p,d} [W,U']\ci{A_p,\bfD}^{1/p}\Vert [S\wt{M}]\ci{\bfD,U}f\Vert\ci{L^p}\Vert[M\wt{S}]\ci{\bfD,W'}g\Vert\ci{L^{p'}},
\end{align*}
\begin{equation*}
|(\Pi_{a}^{(01)}f,g)|\lesssim_{p,d}\Vert[S\wt{M}]\ci{\bfD,W}f\Vert\ci{L^{p}}\Vert[M\wt{S}]\ci{\bfD,W'}g\Vert\ci{L^{p'}}.
\end{equation*}
Therefore, in view of Lemma \ref{l: SM bound} and Remark \ref{r: incomplete sm bound}, in order to estimate $\Pi_a^{(01)}$ we only need to answer Question \ref{quest: vector valued maximal function-not modified} to the positive in the range $2<p<\infty$.
\end{rem}

Observe that weighted estimates for $\Pi_{a}^{(01)}$ imply weighted estimates for the paraproduct $\Pi_{a}^{(10)}$ given by
\begin{equation*}
\Pi_{a}^{(10)}f:=\sum_{R\in\bfD}a\ci{R}\La f^2\ci{R_2}\Ra\ci{R_1}h\ci{R_1}\otimes\frac{\1\ci{R_2}}{|R_2|},
\end{equation*}
since $(\Pi_{a}^{(10)})^{\ast}=\Pi_{\bar{a}}^{(01)}$ in the (unweighted) $L^2(\R^2;\C^d)$ sense, where $\bar{a}$ denotes again the complex conjugate of $a$.

\subsubsection{Partial paraproducts}

Let $\bfD=\cD^1\times\cD^2$ be any product dyadic grid in $\R^2$. Let $i,j$ be nonnegative integers. Consider the shifted paraproduct
\begin{equation*}
\shi\ci{\bfD}^{i,j}f=\sum_{R\in\bfD}\sum_{\substack{P_1\in\ch_{i}(R_1)\\ Q_1\in\ch_{j}(R_1)}}a^{P_1Q_1R_1}\ci{R_2}f\ci{P_1\times R_2}h\ci{Q_1}\otimes\frac{\1\ci{R_2}}{|R_2|},
\end{equation*}
where for every $P_1,Q_1,R_1$, $a^{P_1Q_1R_1}$ is a function in $L^1\ti{loc}(\R)$ with Haar coefficients $a^{P_1Q_1R_1}\ci{R_2},~R_2\in\cD^2$ such that
\begin{equation*}
\Vert a^{P_1Q_1R_1}\Vert\ci{\text{BMO}\ci{\cD^2}}\leq 2^{-(i+j)/2}.
\end{equation*}

\begin{lm}
\label{l: partial paraproduct}
Let $1<p<\infty$, and let $W,U$ be $d\times d$-matrix valued $\bfD$-dyadic biparameter $A_p$ weights on $\R^2$ with $[W,U']\ci{A_p,\bfD}<\infty$. Then, we have the decompositions
\begin{equation*}
|(\shi\ci{\bfD}^{i,j}f,g)|\lesssim_{p,d}[W,U']\ci{A_p,\bfD}^{1/p}2^{-(i+j)/2}\Vert [\wt{S}^{i}S]\ci{\bfD,U}f\Vert\ci{L^{p}}\Vert[\wt{S}^{j}M]\ci{\bfD,W'}g\Vert\ci{L^{p'}},
\end{equation*}
and
\begin{equation*}
|(\shi\ci{\bfD}^{i,j}f,g)|\lesssim_{p,d}2^{-(i+j)/2}\Vert [\wt{S}^{i}S]\ci{\bfD,W}f\Vert\ci{L^{p}}\Vert[\wt{S}^{j}M]\ci{\bfD,W'}g\Vert\ci{L^{p'}}.
\end{equation*}
In particular, for $p=2$ we deduce the bounds
\begin{equation*}
\Vert \shi^{i,j}\ci{\bfD}\Vert\ci{L^{2}(U)\rightarrow L^{2}(W)}\lesssim_{d}[W,U']\ci{A_2,\bfD}^{1/2}[W]\ci{A_2,\bfD}^{5/2}[U]\ci{A_2,\bfD}^{5/2}
\end{equation*}
and
\begin{equation*}
\Vert \shi^{i,j}\ci{\bfD}\Vert\ci{L^{2}(W)\rightarrow L^{2}(W)}\lesssim_{d}[W]\ci{A_2,\bfD}^{5}.
\end{equation*}
\end{lm}

\begin{proof}
By the classical $H^1\ci{\cD^2}$-$\text{BMO}\ci{\cD^2}$ duality we have
\begin{align*}
|(\shi^{i,j}\ci{\bfD}f,g)|&\leq\sum_{R\in\bfD}\sum_{\substack{P_1\in\ch_{i}(R_1)\\ Q_1\in\ch_{j}(R_1)}}|a^{P_1Q_1R_1}\ci{R_2}|\cdot|\La f\ci{P_1\times R_2},\La g^{1}\ci{Q_1}\Ra\ci{R_2}\Ra|\\
&\lesssim \sum_{R_1\in\cD^1}\sum_{\substack{P_1\in\ch_{i}(R_1)\\ Q_1\in\ch_{j}(R_1)}}\Vert a^{P_1Q_1R_1}\Vert\ci{\text{BMO}\ci{\cD^2}}\int_{\R}\bigg(\sum_{R_2\in\cD^2}|\La f\ci{P_1\times R_2},\La g^{1}\ci{Q_1}\Ra\ci{R_2}\Ra|^2\frac{\1\ci{R_2}(x_2)}{|R_2|}\bigg)^{1/2}\mathd x_2\\
&\leq 2^{-(i+j)/2} \sum_{R_1\in\cD^1}\sum_{\substack{P_1\in\ch_{i}(R_1)\\ Q_1\in\ch_{j}(R_1)}}\int_{\R}\bigg(\sum_{R_2\in\cD^2}|\La f\ci{P_1\times R_2},\La g^{1}\ci{Q_1}\Ra\ci{R_2}\Ra|^2\frac{\1\ci{R_2}(x_2)}{|R_2|}\bigg)^{1/2}\mathd x_2\\
&\lesssim_{p,d}2^{-(i+j)/2}[W,U']\ci{A_p,\bfD}^{1/p}\\
&\sum_{R_1\in\cD^1}\sum_{\substack{P_1\in\ch_{i}(R_1)\\ Q_1\in\ch_{j}(R_1)}}\int_{\R}\bigg(\sum_{R_2\in\cD^2}|\cU\ci{x_2,R_1}f\ci{P_1\times R_2}|^2\cdot|\cW'\ci{x_2,R_1}\La g^{1}\ci{Q_1}\Ra\ci{R_2}|^2\frac{\1\ci{R_2}(x_2)}{|R_2|}\bigg)^{1/2}\mathd x_2.
\end{align*}
We have
\begin{align*}
&\sum_{R_1\in\cD^1}\sum_{\substack{P_1\in\ch_{i}(R_1)\\ Q_1\in\ch_{j}(R_1)}}\int_{\R}\bigg(\sum_{R_2\in\cD^2}|\cU\ci{x_2,R_1}f\ci{P_1\times R_2}|^2\cdot|\cW'\ci{x_2,R_1}\La g^{1}\ci{Q_1}\Ra\ci{R_2}|^2\frac{\1\ci{R_2}(x_2)}{|R_2|}\bigg)^{1/2}\mathd x_2\\
&\leq\int_{\R}\sum_{R_1\in\cD^1}\sum_{\substack{P_1\in\ch_{i}(R_1)\\ Q_1\in\ch_{j}(R_1)}}(\sup_{R_2\in\cD^2}|\cW'\ci{x_2,R_1}\La g^{1}\ci{Q_1}\Ra\ci{R_2}|\1\ci{R_2}(x_2))\\
&\bigg(\sum_{R_2\in\cD^2}|\cU\ci{x_2,R_1}f\ci{P_1\times R_2}|^2\frac{\1\ci{R_2}(x_2)}{|R_2|}\bigg)^{1/2}\mathd x_2\\
&=\int_{\R^2}\sum_{R_1\in\cD^1}\bigg(\sum_{\substack{Q_1\in\ch_{j}(R_1)}}(\sup_{R_2\in\cD^2}|\cW'\ci{x_2,R_1}\La g^{1}\ci{Q_1}\Ra\ci{R_2}|\1\ci{R_2}(x_2))\bigg)\\
&\bigg(\sum_{\substack{P_1\in\ch_{i}(R_1)}}\bigg(\sum_{R_2\in\cD^2}|\cU\ci{x_2,R_1}f\ci{P_1\times R_2}|^2\frac{\1\ci{R_2}(x_2)}{|R_2|}\bigg)^{1/2}\bigg)\frac{\1\ci{R_1}(x_1)}{|R_1|}\mathd x\\
&\leq\int_{\R^2}[\wt{S}^{i,0}S]\ci{\bfD,U}f(x)[\wt{S}^{j}M]\ci{\bfD,W'}g(x)\mathd x\leq\Vert [\wt{S}^{i}S]\ci{\bfD,U}f\Vert\ci{L^{p}}\Vert[\wt{S}^{j}M]\ci{\bfD,W'}g\Vert\ci{L^{p'}}.
\end{align*}
The second decomposition is proved similarly, and the bounds for the case $p=2$ follow immediately by applying Lemmas \ref{l: shifted square-maximal bound p = 2} and \ref{l: shifted square-square bound p = 2}. 
\end{proof}

\begin{rem}
\label{r: incomplete partial paraproduct}
In view of Remark \ref{r: incomplete shifted square-maximal bounds}, we see that in order to deduce actual weighed bounds from Lemma \ref{l: partial paraproduct}, we need to answer Question \ref{quest: vector valued maximal function-not modified} to the positive.
\end{rem}

Working symmetrically, we obtain analogous decompositions for shifted paraproducts $\shi^{i,j,\ast}\ci{\bfD}$ of the form
\begin{equation*}
\shi\ci{\bfD}^{i,j,\ast}f=\sum_{R\in\bfD}\sum_{\substack{P_2\in\ch_{i}(R_2)\\ Q_2\in\ch_{j}(R_2)}}a^{P_2Q_2R_2}\ci{R_1}f\ci{R_1\times P_2}h\ci{Q_2}\otimes\frac{\1\ci{R_1}}{|R_1|},
\end{equation*}
where $i,j$ are nonnegative integers, and for every $P_2,Q_2,R_2$, $a^{P_2Q_2R_2}$ is a function in $L^1\ti{loc}(\R)$ with Haar coefficients $a^{P_2Q_2R_2}\ci{R_1},~R_1\in\cD^1$ such that
\begin{equation*}
\Vert a^{P_2Q_2R_2}\Vert\ci{\text{BMO}\ci{\cD^1}}\leq 2^{-(i+j)/2}.
\end{equation*}

\subsection{\texorpdfstring{$L^2$}{L2} matrix-weighted bounds for general Journ\'e operators}

Notice that for $p=2$ we have estimated every piece in Martikainen's representation \eqref{martikainen representation}. It follows that if $T$ is any (not necessarily paraproduct-free) Journ\'e operator on $\R\times\R$, then
\begin{equation*}
\Vert T\Vert\ci{L^2(U)\rightarrow L^2(W)}\lesssim_{d,T}[W,U']\ci{A_2(\R\times\R)}^{1/2}[W]\ci{A_2(\R\times\R)}^{5/2}[U]\ci{A_2(\R\times\R)}^{5/2}
\end{equation*}
and
\begin{equation*}
\Vert T\Vert\ci{L^2(W)\rightarrow L^2(W)}\lesssim_{d,T}[W]\ci{A_2(\R\times\R)}^{5}.
\end{equation*}

\begin{rem}
\label{r: incomplete general Journe}
In order to obtain $L^p$ matrix-weighted bounds for general Journ\'e operators for any $1<p<\infty$ using our methods, we see in view of Remarks \ref{r: incomplete mixed paraproducts} and \ref{r: incomplete partial paraproduct} that we only need to answer Question \ref{quest: vector valued maximal function-not modified} to the positive.
\end{rem}

\section{Appendix}

\subsection{Reverse H\ddoto lder inequality for biparameter dyadic \texorpdfstring{$A_p$}{Ap} weights}

For the reader's convenience, we give here the statement and proof of the reverse H\ddoto lder inequality for scalar biparameter dyadic $A_p$ weights due to Holmes--Petermichl--Wick \cite[Proposition 2.2]{holmes-petermichl-wick}.

First of all, if $\cD$ is any dyadic grid in $\R^n$ and $w$ is any weight on $\R^n$, we define the \emph{dyadic $A_{\infty}$ characteristic}
\begin{equation*}
[w]\ci{A_{\infty},\cD}:=\sup_{Q\in\cD}\frac{\La M\ci{\cD,Q}w\Ra\ci{Q}}{\La w\Ra\ci{Q}},
\end{equation*}
where $M\ci{\cD,Q}$ denotes the dyadic Hardy--Littlewood maximal function adapted to $Q$,
\begin{equation*}
M\ci{\cD,Q}f:=\sup_{\substack{P\in\cD\\P\subseteq Q}}\La f\Ra\ci{P}\1\ci{P}.
\end{equation*}
A tiny modification of \cite[Lemma 4.1]{convex body} shows that
\begin{equation*}
[w]\ci{A_{\infty},\cD}\leq 4[w]\ci{A_p,\cD},\qquad 1<p<\infty.
\end{equation*}
We now give the statement of the sharp reverse H\ddoto lder inequality for one-parameter dyadic $A_p$ weights from \cite{convex body} (attributed therein to \cite{hytonen-perez-rela}, \cite{vasyunin}).

\begin{lm}
Let $\cD$ be any dyadic grid in $\R^n$, and let $w$ be a weight on $\R^{n}$ with $[w]\ci{A_{\infty},\cD}<\infty$. Then, for all 
\begin{equation*}
\delta\in\bigg(0,\frac{1}{2^{n+1}[w]\ci{A_{\infty},\cD}}\bigg),
\end{equation*}
there holds
\begin{equation*}
\La w^{1+\delta}\Ra\ci{Q}\leq 2\La w\Ra\ci{Q}^{1+\delta},\qquad\forall Q\in\cD.
\end{equation*}
\end{lm}

We now state and prove the reverse H\ddoto lder inequality for scalar biparameter dyadic $A_p$ weights due to Holmes--Petermichl--Wick \cite[Proposition 2.2]{holmes-petermichl-wick}, with emphasis on explicitly tracking constants.

\begin{lm}[\cite{holmes-petermichl-wick}]
Let $\bfD:=\cD^1\times\cD^2$ be any product dyadic grid in $\R^n\times\R^m$. Let $1<p<\infty$, and let $w$ be a weight on $\R^{n+m}$ with $[w]\ci{A_p,\bfD}<\infty$. Then, for all
\begin{equation*}
\delta\in\bigg(0,\frac{1}{2^{\max(m,n)+3}[w]\ci{A_p,\bfD}}\bigg),
\end{equation*}
there holds
\begin{equation*}
\La w^{1+\delta}\Ra\ci{R}\leq 4\La w\Ra\ci{R}^{1+\delta},\qquad\forall R\in\bfD.
\end{equation*}
\end{lm}

\begin{proof}
We follow the proof of \cite[Proposition 2.2]{holmes-petermichl-wick}. For a.e. $x_1\in\R^n$, define
\begin{equation*}
w_{x_1}(x_2):=w(x_1,x_2),\qquad x_2\in\R^m.
\end{equation*}
For a.e. $x_1\in\R^n$, we have that $w_{x_1}$ is a weight on $\R^{m}$, and since $w$ is scalar valued it is trivial by classical dyadic Lebesgue differentiation to see that $[w_{x_1}]\ci{A_p,\cD^2}\leq[w]\ci{A_p,\bfD}$. Therefore, for a.e. $x_1\in\R^{1}$ we have
\begin{equation*}
\delta<\frac{1}{2^{m+3}[w_{x_1}]\ci{A_p,\cD^2}}\leq\frac{1}{2^{m+1}[w_{x_1}]\ci{A_\infty,\cD^2}}
\end{equation*}
and therefore
\begin{align*}
\La w^{1+\delta}\Ra\ci{R}=\strokedint_{R_1}\La w_{x_1}^{1+\delta}\Ra\ci{R_2}\mathd x_1
\leq 2\strokedint_{R_1}\La w_{x_1}\Ra^{1+\delta}\ci{R_2}\mathd x_1=2\La v\ci{R_2}^{1+\delta}\Ra\ci{R_1},
\end{align*}
where $v\ci{R_2}(x_1):=\La w(x_1,\fdot)\Ra\ci{R_2}$, for a.e. $x_1\in\R^n$. We notice that for all $P\in\cD^1$, we have $\La v\ci{R_2}\Ra\ci{P}=\La w\Ra\ci{P\times R_2}$, and also by Jensen's inequality (since $-1/(p-1)<0$)
\begin{equation*}
v\ci{R_2}(x_1)^{-1/(p-1)}=\La w(x_1,\fdot)\Ra\ci{R_2}^{-1/(p-1)}\leq\La w(x_1,\fdot)^{-1/(p-1)}\Ra\ci{R_2},
\end{equation*}
therefore $\La v\ci{R_2}\Ra\ci{P}\leq\La w^{-1/(p-1)}\Ra\ci{P\times R_2}$. We deduce $[v]\ci{A_p,\cD^1}\leq[w]\ci{A_p,\bfD}$. Therefore
\begin{equation*}
\delta<\frac{1}{2^{n+3}[v\ci{R_2}]\ci{A_p,\cD^2}}\leq\frac{1}{2^{n+1}[v\ci{R_2}]\ci{A_\infty,\cD^2}}
\end{equation*}
and therefore
\begin{align*}
\La v\ci{R_2}^{1+\delta}\Ra\ci{R_1}\leq 2\La v\ci{R_2}\Ra\ci{R_1}^{1+\delta}=2\La w\Ra^{1+\delta}\ci{R},
\end{align*}
concluding the proof.
\end{proof}

\subsection{``Slicing" a finite supremum of measurable functions}

Let $X$ be a measurable space. Let $N$ be any positive integer. Let $f_{n}:X\rightarrow[0,\infty)$, $n=1,\ldots,N$ be a sequence of measurable functions on $X$. Set $f:=\sup_{n=1,\ldots,N}f_n$. Then, there exists a measurable function $J:X\rightarrow\lbrace1,\ldots,N\rbrace$ (where $\lbrace1,\ldots,N\rbrace$ is equipped with the power-set $\sigma$-algebra), such that
\begin{equation*}
f\ci{J(x)}(x)=f(x),\qquad\forall x\in X.
\end{equation*}
Indeed, set
\begin{equation*}
J(x):=\min\lbrace n\in\lbrace1,\ldots,N\rbrace:~f(x)=f_{n}(x)\rbrace,\qquad x\in X.
\end{equation*}
Then, $J$ is measurable and $f\ci{J(x)}=f(x)$, for all $x\in X$.

\subsection{``Passing" the Euclidean norm outside the integral}

Let $X$ be a measurable space, and let $(Y,\mu)$ be a measure space. Let $f:X\times Y\rightarrow\C^d$ be a measurable function, such that $f(x,\fdot)\in L^1(Y,\mu)$, for all $x\in X$. Then, there exists a measurable function $\f:X\times Y\rightarrow\lbrace-1,1\rbrace$, such that
\begin{equation*}
\int_{Y}|f(x,y)|\mathd\mu(y)\lesssim_{d}\bigg|\int_{Y}\phi(x,y)f(x,y)\mathd\mu(x)\bigg|,\qquad\forall x\in X.
\end{equation*}
Indeed, let $f_1,\ldots,f_n$ be the coordinate projections of $f$. Set $g_i^{j}:=\text{Re}(f_i)$ and $g_i^{j}=\text{Im}(f_i)$, for all $i=1,\ldots,d$ and $j=1,2$. Then, we have
\begin{equation*}
\int_{Y}|f(x,y)|\mathd\mu(y)\leq C(d)\sum_{i=1}^{d}\sum_{j=1}^{2}\int_{Y}|g^{j}_i(x,y)|\mathd\mu(y),\qquad\forall x\in X.
\end{equation*}
For all $x\in X$, set
\begin{align*}
I(x):=\min\bigg\lbrace i\in\lbrace1,\ldots,d\rbrace:  ~&\max\bigg(\int_{Y}|g^{1}_i(x,y)|\mathd\mu(y),\int_{Y}|g^{2}_i(x,y)|\mathd\mu(y)\bigg)\\
&\geq\frac{1}{2dC(d)}\int_{Y}|f(x,y)|\mathd\mu(y)\bigg\rbrace
\end{align*}
and
\begin{align*}
J(x):=\min\bigg\lbrace j\in\lbrace1,2\rbrace:  ~\int_{Y}|g^{j}\ci{I(x)}(x,y)|\mathd\mu(y)
\geq\frac{1}{2dC(d)}\int_{Y}|f(x,y)|\mathd\mu(y)\bigg\rbrace.
\end{align*}
Clearly, the functions $I:X\rightarrow\lbrace1,\ldots,d\rbrace,J:X\rightarrow\lbrace1,2\rbrace$ are well-defined measurable. Set then
\begin{equation*}
\f(x,y):=\text{sgn}(g^{J(x)}\ci{I(x)}(x,y)),\qquad(x,y)\in X\times Y,
\end{equation*}
where $\text{sgn}(t):=t/|t|$ for $t\in\R\setminus\lbrace0\rbrace$, and $\text{sgn}(0):=1$. Then, the function $\f:X\times Y\rightarrow\lbrace-1,1\rbrace$ is well-defined measurable and
\begin{equation*}
\int_{Y}|f(x,y)|\mathd\mu(y)\lesssim_{d}\int_{Y}|g^{J(x)}\ci{I(x)}(x,y)|\mathd\mu(y)\lesssim_{d}\bigg|\int_{Y}\f(x,y)f(x,y)\mathd\mu(y)\bigg|,
\end{equation*}
for all $x\in X$, concluding the proof.

\subsection{Measurable choice of reducing matrix}
\label{s:MeasurableReducingOp}

Here we show that if a family of norms is measurably parametrized, then one can choose a reducing matrix for each of them in a way that is measurable with respect to the parameter.

We fix a positive integer $d$ throughout this subsection. All metric spaces will be assumed to be equipped with the Borel $\sigma$-algebra with respect to the topology induced by their metric, unless explicitly stated otherwise.

\subsubsection{Hausdorff distance}

Let $\F=\R$ of $\F=\C$. Recall for any $x\in\F^d$ and for any nonempty $K\subseteq\F^d$ the notation
\begin{equation*}
\text{dist}(x,K):=\inf_{y\in K}|x-y|.
\end{equation*}
Let $\cX(\F^d)$ be the set of all nonempty compact subsets of $\F^d$. For all $K_1,K_2\in\cX(\F^d)$, we define the Hausdorff distance of $K_1,K_2$ by
\begin{equation*}
\delta(K_1,K_2):=\max(\sup_{y\in K_1}\text{dist}(y,K_2),~\sup_{z\in K_2}\text{dist}(z,K_1)).
\end{equation*}
Then, it is very well-known that $(\cX(\F^d),\delta)$ is a metric space. Moreover, it is easy to see that $(\cX(\F^d),\delta)$ is separable; for example, the family of all finite subsets of $(\Q^{d}+i\Q^d)\cap\F^d$ is dense in $(\cX(\F^d),\delta)$. See \cite{aamari} for a more general version of that, as well as for additional properties of the Hausdorff distance.

\subsubsection{Ellipsoids}

Let $\F=\R$ or $\F=\C$. Let $\B_{d}(\F):=\lbrace v\in\F^d:~|v|\leq1\rbrace$ be the unit ball in $\F^d$. We identify linear maps $A:\F^d\rightarrow\F^d$ with $d\times d$-matrices with entries in $\F$ in the natural way.

A subset $E$ of $\F^d$ is said to be a centrally symmetric ellipsoid in $\F^d$ if there exists a $\F$-linear map $A:\F^d\rightarrow\F^d$ such that
\begin{equation*}
E=A\B_{d}(\F):=\lbrace Av:~v\in\B_{d}(\F)\rbrace.
\end{equation*}
Simple facts of linear algebra, see for instance \cite[page 304]{goffin-hoffman}, imply that if $A,B\in M_{d}(\F)$ are such that $E=A\B_{d}(\F)=B\B_{d}(\F)$, then $AA^{\ast}=BB^{\ast}$ and $E=(AA^{\ast})^{1/2}\B_{d}(\F)$. In particular, if $E$ is any centrally symmetric ellipsoid in $\F^d$, then there exists a unique positive semidefinite matrix $A\in M_{d}(\F)$ such that $E=A\B_{d}(\F)$. Denote by $\text{PS}_{d}(\F)$ the set of all positive semidefinite matrices in $M_{d}(\F)$, and by $\cE_{d}(\F)$ the set of all centrally symmetric ellipsoids in $\F^d$. Then, the map $\cM\ci{\F,d}:\cE_{d}(\F)\rightarrow\text{PS}_{d}(\F)$ given by $\cM\ci{\F,d}(A\B_{d}(\F)):=A$, for all $A\in\text{PS}_{d}(\F)$, is well-defined bijective.

A centrally symmetric ellipsoid $E$ in $\F^d$ is said to be nondegenerate if $E=A\B_{d}(\F)$ for some invertible linear map $A:\F^{d}\rightarrow\F^{d}$. Let $\cE^{\ast}_{d}(\F)$ be the set of all nondegenerate centrally symmetric ellipsoids in $\F^d$. Notice that any $E\in\cE_{d}(\F)$ is nondegenerate if and only if $\cM\ci{\F,d}(E)$ is positive definite. We denote by $\text{P}_{d}(\F)$ the set of all positive definite matrices in $M_{d}(\F)$.

\subsubsection{John ellipsoids}

Let $K$ be any convex compact subset of $\R^d$ such that $0\in\text{Int}(K)$ and
\begin{equation*}
-K:=\lbrace -v:~v\in K\rbrace\subseteq K.
\end{equation*}
Then, it is well-known that there exists a unique $E\in\cE^{\ast}_{d}(\R)$, such that $E\subseteq K$ and
\begin{equation*}
|E|=\max\lbrace|F|:~F\in\cE^{\ast}_{d}(\R),~F\subseteq K\rbrace.
\end{equation*}
The ellipsoid $E$ is usually called the John ellipsoid of $K$. Then, it is well-known, see \cite{goldberg}, \cite{convex body}, that
\begin{equation}
\label{John ellipsoid containment}
E\subseteq K\subseteq\sqrt{d}E.
\end{equation}

Let $\cC_{d}(\R)$ be the set of all convex compact subsets of $\R^d$ such that $0\in\text{Int}(K)$ and $-K\subseteq K$. Then, it is well-known, see \cite{mordhorst}, that the map $J_{d}:\cC_{d}(\R)\rightarrow\cE^{\ast}_{d}(\R)$ sending each $K\in\cC_{d}(\R)$ to its John ellipsoid is continuous with respect to the Hausdorff distance.

\subsubsection{John ellipsoids for unit balls of norms on \texorpdfstring{$\R^d$}{Rd}}

Let $r$ be any norm on $\R^d$. Let $K:=\lbrace v\in\R^{d}:~r(v)\leq 1\rbrace$ be the unit ball of $r$. Clearly, $K\in\cC_{d}(\R)$. Let $E$ be the John ellipsoid of $K$. Set $A:=\cM_{\R,d}(E)$. Then, \eqref{John ellipsoid containment} implies that for all $e\in\R^d$, we have
\begin{equation*}
|e|\leq1\Rightarrow r(Ae)\leq1
\end{equation*}
and
\begin{equation*}
r(e)\leq1\Rightarrow|A^{-1}e|\leq\sqrt{d}.
\end{equation*}
It follows that
\begin{equation*}
r(e)\leq|A^{-1}e|\leq\sqrt{d}r(e),\qquad\forall e\in\R^d.
\end{equation*}

\subsubsection{Measurable choice of reducing matrix for measurably parametrized norms on \texorpdfstring{$\R^d$}{Rd}}

Let $(X,\cF)$ be any measurable space. Let $r:X\times\R^d\rightarrow[0,\infty)$ be a function such that $r(x,\fdot)$ is a norm on $\R^d$, for all $x\in X$, and $r(\fdot,y)$ is a measurable function on $X$, for all $y\in\R^d$. For all $x\in X$, we let $K_{x}$ be the unit ball of $r(x,\fdot)$, $E_{x}$ be the John ellipsoid of $K_{x}$, and we set $A_{x}:=\cM\ci{\R,d}(E_{x})$.

\begin{prop}
The map $\cR:X\rightarrow\emph{P}_{d}(\R)$ given by $\cR(x):=A_{x}$, for all $x\in X$, is measurable.
\end{prop}

The proof will be accomplished in several steps. We first consider the map $\cK:X\rightarrow\cC_{d}(\R)$ given by $\cK(x):=K_{x}$, for all $x\in X$. Then, it is clear that
\begin{equation*}
\cR=\cM\ci{\R,d}\circ J_{d}\circ\cK.
\end{equation*}
Recall that $J_{d}:\cC_{d}(\R)\rightarrow\cE_{d}^{\ast}(\R)$ is continuous, thus measurable.

\begin{lm}[\cite{goffin-hoffman}]
The map $\cM\ci{\R,d}:\cE_{d}(\R)\rightarrow\emph{PS}_{d}(\R)$ is continuous. In fact, there holds
\begin{equation*}
\delta(E,F)\leq|A-B|\lesssim_{d}\delta(E,F),
\end{equation*}
where $E=A\B_{d}(\R)$ and $F=B\B_{d}(\R)$, for all $A,B\in\emph{PS}_{d}(\R)$, so $\cM\ci{\R,d}:\cE_{d}(\R)\rightarrow\emph{PS}_{d}(\R)$ is a topological homeomorphism.
\end{lm}

Therefore, $\cM\ci{\R,d}|\ci{\cE_{d}^{\ast}(\R)}:\cE_{d}^{\ast}(\R)\rightarrow\text{P}_{d}(\R)$ is continuous, thus measurable. Thus, it suffices to prove that the map $\cK:X\rightarrow\cC_{d}(\R)$ is measurable.

\begin{lm}
Let $K$ be any compact convex subset of $\R^{d}$ with $\emph{Int}(K)\neq\emptyset$. Then, $K\cap\Q^d$ is dense in $K$. In particular
\begin{equation*}
\emph{dist}(x,K)=\inf_{y\in K\cap\Q^d}|y-x|,\qquad\forall x\in\R^d.
\end{equation*}
\end{lm}

\begin{proof}
It is a well-known fact of convex geometry, see e.g. \cite[Proposition 3.1]{gruber}, that $K=\text{Cl}(\text{Int}(K))$. Since $H:=\text{Int}(K)$ is an open subset of $\R^d$, we have that $H\cap\Q^{d}$ is dense in $H$. The desired result follows.
\end{proof}

Since the function $\text{dist}(\fdot,K)$ on $\R^d$ is continuous, for any nonempty $K\subseteq\R^d$, the previous lemma implies immediately the following corollary.

\begin{cor}
Let $K_1,K_2$ be any two convex compact subsets of $\R^d$ with $\emph{Int}(K_i)\neq\emptyset$, $i=1,2$. Then
\begin{equation*}
\delta(K_1,K_2)=\max(\sup_{y\in K_1\cap\Q^{d}}\emph{dist}(y,K_2),~\sup_{z\in K_2\cap\Q^{d}}\emph{dist}(z,K_1)).
\end{equation*}
\end{cor}

We can now prove that the map $\cK$ is measurable.

\begin{lm}
The map $\cK:X\rightarrow\cC_{d}(\R)$ given by $\cK(x):=K_{x}$, for all $x\in X$, is measurable.
\end{lm}

\begin{proof}
Since $(\cX(\R^d),\delta)$ is a separable metric space, it follows that $(\cC_{d}(\R),\delta)$ is also a separable metric space. Therefore, it suffices to prove that for all $K\in\cC_{d}(\R)$ and for all $\e>0$, the inverse image along $\cK$ of the closed $\delta$-ball of radius $\e$ and center $K$,
\begin{equation*}
\lbrace x\in X:~\delta(K_x,K)\leq\e\rbrace,
\end{equation*}
belongs to $\cF$. We have
\begin{align*}
&\lbrace x\in X:~\delta(K_x,K)\leq\e\rbrace\\
&=\lbrace x\in X:~\sup_{y\in K\cap\Q^{d}}\text{dist}(y,K_x)\leq\e\rbrace\cap\lbrace x\in X:~\sup_{z\in K_x\cap\Q^{d}}\text{dist}(z,K)\leq\e\rbrace.
\end{align*}
Notice that
\begin{align*}
\lbrace x\in X:~\sup_{y\in K\cap\Q^{d}}\text{dist}(y,K_x)\leq\e\rbrace=\bigcap_{y\in K\cap\Q^{d}}\bigcap_{n=1}^{\infty}\bigg\lbrace x\in X:~\text{dist}(y,K_x)<\e+\frac{1}{n}\bigg\rbrace,
\end{align*}
and
\begin{align*}
&\bigg\lbrace x\in X:~\text{dist}(y,K_x)<\e+\frac{1}{n}\bigg\rbrace=
\bigg\lbrace x\in X:~\inf_{z\in K_{x}\cap\Q^{d}}|y-z|<\e+\frac{1}{n}\bigg\rbrace\\
&=\bigcup_{\substack{z\in\Q^{d}\\|y-z|<\e+\frac{1}{n}}}\lbrace x\in X:~z\in K_{x}\rbrace
=\bigcup_{\substack{z\in\Q^{d}\\|y-z|<\e+\frac{1}{n}}}\lbrace x\in X:~r(x,z)\leq 1\rbrace,
\end{align*}
for all $y\in\R^d$ and for all $n=1,2,\ldots$. Thus $\lbrace x\in X:~\sup_{y\in K\cap\Q^{d}}\text{dist}(y,K_x)\leq\e\rbrace\in\cF$. Moreover, we have
\begin{align*}
\lbrace x\in X:~\sup_{z\in K_x\cap\Q^{d}}\text{dist}(z,K)\leq\e\rbrace&=\bigcap_{\substack{z\in\Q^{d}\\\text{dist}(z,K)>\e}}\lbrace x\in X:~z\notin K_{x}\rbrace\\
&=\bigcap_{\substack{z\in\Q^{d}\\\text{dist}(z,K)>\e}}\lbrace x\in X:~r(x,z)>1\rbrace.
\end{align*}
It follows that $\lbrace x\in X:~\sup_{z\in K_x\cap\Q^{d}}\text{dist}(z,K)\leq\e\rbrace\in\cF$ as well, concluding the proof.
\end{proof}

\subsubsection{Identifying \texorpdfstring{$\C^d$}{Cd} with \texorpdfstring{$\R^{2d}$}{R2d}}

Recall that we identify complex $d\times d$-matrices with $\C$-linear maps on $\C^d$ in the natural way, and that we identify $(2d)\times(2d)$-matrices in $\R^{2d}$ with linear maps on $\R^{2d}$ in the natural way.

Let $\cB=\lbrace e_1,e_2,\ldots,e_{2d-1},e_{2d}\rbrace$ be the standard $\R$-basis of $\C^d$, and let $\cB'=\lbrace e_1,\ldots,e_{2d}\rbrace$ be the standard basis of $\R^{2d}$. For all $A\in M_{d}(\C)$, it is easy to see by direct computation that $[A]\ci{\cB}^{\cB}=H(A)$, where $H:M_{d}(\C)\rightarrow M_{2d}(\R)$ is the natural ring monomorphism induced by the ring monomorphism $H:\C\rightarrow M_{2}(\R)$ given by
\begin{equation*}
H(z):=
\begin{bmatrix}
\rp(z)&-\ip(z)\\
\ip(z)&\rp(z)
\end{bmatrix}
,\qquad\forall z\in\C.
\end{equation*}

Consider the map $R:\C^{d}\rightarrow\R^{2d}$ given by
\begin{equation*}
R(z_1,\ldots,z_d):=(\rp(z_1),\ip(z_1),\ldots,\rp(z_d),\ip(z_d)),\qquad\forall z\in\C^d.
\end{equation*}
Clearly, $R$ is an $\R$-linear isometric isomorphism, and it is also measure-preserving. In particular, $R(\B_{d}(\C))=\B_{2d}(\R)$. Note that $R$ induces the map $\cX(R):\cX(\C^d)\rightarrow\cX(\R^{2d})$ given by $\cX(R)(K):=R(K)$, for all $K\in\cX(\C^d)$, which is an isometric isomorphism with respect to the Hausdorff distance. 

Notice that for all $\C$-linear maps $A:\C^d\rightarrow\C^d$, the map $RAR^{-1}:\R^{2d}\rightarrow\R^{2d}$ is linear, and
\begin{equation*}
[RAR^{-1}]\ci{\cB'}^{\cB'}=[R]\ci{\cB}^{\cB'}[A]\ci{\cB}^{\cB}[R^{-1}]\ci{\cB'}^{\cB}=I_{2d}\cdot H(A)\cdot I_{2d}^{-1}=H(A).
\end{equation*}

Notice also that for all $A\in M_{d}(\C)$, we have $H(A)^{T}=H(\bar{A}^{T})=H(A^{\ast})$, where $\bar{A}$ denotes the complex conjugate of $A$, $A^{T}$ denotes the transpose of $A$, and $A^{\ast}$ denotes the adjoint of $A$. Therefore, if $A\in M_{d}(\C)$ is Hermitian, then $H(A)$ is symmetric. 

Let now $A\in\text{PS}_{d}(\C)$. Then, we can write $A=UDU^{\ast}$, where $U$ is a complex $d\times d$-unitary matrix and $D$ is a $d\times d$-diagonal matrix with real non-negative entries. Then, $H(D)$ is a $(2d)\times(2d)$-diagonal matrix with real non-negative entries. Moreover, we have
\begin{equation*}
H(U)^{T}H(U)=H(U^{\ast}U)=H(I_{d})=I_{2d},
\end{equation*}
so $H(U)$ is a $(2d)\times(2d)$-orthogonal matrix. Thus
\begin{equation*}
H(A)=H(U)DH(U)^{T}\in\text{PS}_{2d}(\R).
\end{equation*}

\subsubsection{Complex John ellipsoids}

Let $\cC_{d}(\C)$ be the set of all compact convex subsets $K$ of $\C^d$ such that $0\in\text{Int}(K)$ and
\begin{equation*}
zK:=\lbrace zv:~v\in K\rbrace\subseteq K,
\end{equation*}
for all $z\in\T:=\lbrace w\in\C:~|w|=1\rbrace$. Clearly, $\wt{K}:=R(K)\in\cC_{2d}(\R)$. Let $\wt{E}$ be the John ellipsoid of $\wt{K}$ in $\R^{2d}$. Set $E:=R^{-1}(\wt{E})$. We claim that $E$ is an ellipsoid in $\C^d$.

Set $\wt{A}:=\cM\ci{\R,d}(\wt{E})$ and $A:=R^{-1}\wt{A}R$. It is obvious that $A:\C^d\rightarrow\C^d$ is an $\R$-linear map. Moreover, we have
\begin{equation*}
E=R^{-1}(\wt{E})=R^{-1}\wt{A}\B_{2d}(\R)=R^{-1}\wt{A}R\B_{d}(\C)=A\B_{d}(\C).
\end{equation*}
Therefore, it suffices to prove that $A:\C^d\rightarrow\C^d$ is $\C$-linear.

For every $z\in\T$, consider the $\C$-linear map $L_{z}:\C^d\rightarrow\C^d$ given by $L_{z}(v):=zv$, for all $v\in\C^d$. Then it is easy to see that $RL_{z}=\wt{L}_{z}R$, where the linear map $\wt{L}_{z}:\R^{2d}\rightarrow\R^{2d}$ is given by
\begin{equation*}
\wt{L}_{z}:=H(L_{z})=
\begin{bmatrix}
\rp(z)&-\ip(z)&&&&&&&&0\\
\ip(z)&\rp(z)&&&&&&&\\
&&&&\ddots&&&&&\\
&&&&&&&&\rp(z)&-\ip(z)\\
0&&&&&&&&\ip(z)&\rp(z)
\end{bmatrix}
,
\end{equation*}
for all $z\in\T$. Clearly, for all $z\in\T$, $\wt{L}_{z}:\R^{2d}\rightarrow\R^{2d}$ is an orthogonal map, because $|z|=1$. It follows that for all $z\in\T$, $\wt{L}_{z}(\wt{E})$ is the John ellipsoid in $\R^{2d}$ of $\wt{L}_{z}(\wt{K})$. Notice that by assumption
\begin{equation*}
\wt{L}_{z}(\wt{K})=\wt{L}_{z}R(K)=RL_{z}(K)=R(K)=\wt{K},\qquad\forall z\in\T.
\end{equation*}
Therefore, by uniqueness of the John ellipsoid we deduce $\wt{L}_{z}(\wt{E})=\wt{E}$, for all $z\in\T$. Notice that $\wt{L}_{z}^{T}=\wt{L}_{\bar{z}}=\wt{L}_{z}^{-1}$, for all $z\in\T$. Then, for all $z\in\T$, $\wt{L}_{z}\wt{A}(\wt{L}_{z})^{T}\in\text{PS}_{2d}(\R)$ (because $\wt{A}\in \text{PS}_{2d}(\R)$) and
\begin{equation*}
\wt{L}_{z}\wt{A}(\wt{L}_{z})^{T}\B_{2d}(\R)=\wt{L}_{z}\wt{A}\wt{L}_{\bar{z}}\B_{2d}(\R)=\wt{L}_{z}\wt{A}\B_{2d}(\R)=\wt{L}_{z}\wt{E}=\wt{E}.
\end{equation*}
By uniqueness of $\wt{A}$ we deduce $\wt{L}_{z}\wt{A}\wt{L}_{\bar{z}}=\wt{A}$, therefore deduce $\wt{L}_{z}\wt{A}=\wt{A}\wt{L}_{z}$, for all $z\in\T$. Thus
\begin{equation*}
AL_{z}=R^{-1}\wt{A}RL_{z}=R^{-1}\wt{A}\wt{L}_{z}R=R^{-1}\wt{L}_{z}\wt{A}R=L_{z}R^{-1}\wt{A}R=L_{z}A,\qquad\forall z\in\T,
\end{equation*}
proving that $A:\C^d\rightarrow\C^d$ is $\C$-linear (since $A:\C^d\rightarrow\C^d$ is already known to be $\R$-linear). In the sequel, we will say that $E$ is the \emph{complex} John ellipsoid of $K$.

Let $J\ci{\C,d}:\cC_{d}(\C)\rightarrow\cE^{\ast}_{d}(\C)$ be the map sending each $K\in\cC_{d}(C)$ to its complex John ellipsoid. Since $R$ is an isometry, we have
\begin{equation*}
\delta(E,F)=\delta(R(E),R(F)),
\end{equation*}
for every nonempty compact subsets $E,F$ of $\C^d$. Therefore, since $J_{2d}:\cC_{2d}(\R)\rightarrow\cE^{\ast}_{2d}(\R)$ is continuous with respect to the Hausdorff distance, it follows that $J\ci{\C,d}:\cC_{d}(\C)\rightarrow\cE_{d}^{\ast}(\C)$ is also continuous with respect to the Hausdorff distance.

\subsubsection{From ellipsoids in \texorpdfstring{$\C^d$}{Cd} to ellipsoids in \texorpdfstring{$\R^{2d}$}{R2d}}

Let $E\in\cE_{d}(\C)$, and let $A:=\cM\ci{\C,d}(E)$, so $E=A\B_{d}(\C)$. Set $\wt{E}:=R(E)$. Then
\begin{equation*}
\wt{E}=RAR^{-1}\B_{2d}(\R),
\end{equation*}
so since $R:\C^{d}\rightarrow\R^{2d}$ is an $\R$-linear isometric isomorphism, it follows that $\wt{E}$ is an ellipsoid in $\R^{2d}$. Since $A\in\text{PS}_{d}(\C)$, by the above we have $[RAR^{-1}]\ci{\cB'}^{\cB'}=H(A)\in\text{PS}_{2d}(\R)$. Thus
\begin{equation*}
\cM\ci{\R,2d}(\wt{E})=H(A)=H(\cM\ci{\C,d}(E)).
\end{equation*}
This implies that $\cM\ci{\C,d}=H^{-1}\circ\cM\ci{\R,2d}\circ \cX(R)|\ci{\cE_{d}(\C)}$. Since $H:M_{d}(\C)\rightarrow M_{2d}(\R)$ is a topological homeomorphism onto its image, $\cX(R):\cX(\C^{d})\rightarrow\cX(\R^{2d})$ is an isometric isomorphism with respect to the Hausdorff distance, and $\cM\ci{\R,2d}:\cE_{2d}(\R)\rightarrow\text{PS}_{2d}(\R)$ is continuous, it follows that $\cM\ci{\C,d}:\cE_{d}(\C)\rightarrow\text{PS}_{d}(\C)$ is continuous. In fact, since $\cM\ci{\R,2d}:\cE_{2d}(\R)\rightarrow\text{PS}_{2d}(\R)$ is a topological homeomorphism, it follows that $\cM\ci{\C,d}:\cE_{d}(\C)\rightarrow\text{PS}_{d}(\C)$ is also a topological homeomorphism.

\subsubsection{John ellipsoids for unit balls of norms on \texorpdfstring{$\C^d$}{Cd}}

Let $r$ be any norm on $\C^d$. Let $K:=\lbrace v\in\C^{d}:~r(v)\leq 1\rbrace$ be the unit ball of $r$. Clearly, $K\in\cC\ci{d}(\C)$. Let $E$ be the complex John ellipsoid of $K$. Set $A:=\cM\ci{\C,d}(E)$. Then, \eqref{John ellipsoid containment} implies that
\begin{equation*}
R(E)\subseteq R(K)\subseteq \sqrt{2d}R(E),
\end{equation*}
therefore
\begin{equation*}
E\subseteq K\subseteq \sqrt{2d}E.
\end{equation*}
In fact, it follows from the above that $E$ is a maximum volume centrally symmetric nondegenerate complex ellipsoid contained in $K$, therefore by the proof of \cite[Proposition 1.2]{goldberg} we have
\begin{equation*}
E\subseteq K\subseteq \sqrt{d}E.
\end{equation*}
Thus, for all $e\in\C^d$, we have
\begin{equation*}
|e|\leq1\Rightarrow r(Ae)\leq1
\end{equation*}
and
\begin{equation*}
r(e)\leq1\Rightarrow|A^{-1}e|\leq\sqrt{d}.
\end{equation*}
It follows that $r(e)\leq |A^{-1}e|\leq \sqrt{d}r(e)$, for all $e\in\C^d$.

\subsubsection{Measurable choice of reducing matrix for measurably parametrized norms on \texorpdfstring{$\C^d$}{Cd}}

Let $(X,\cF)$ be any measurable space. Let $r:X\times\C^d\rightarrow[0,\infty)$ be a function such that $r(x,\fdot)$ is a norm on $\C^d$, for all $x\in X$, and $r(\fdot,y)$ is a measurable function on $X$, for all $y\in\C^d$. For all $x\in X$, we let $K_{x}$ be the unit ball of $r(x,\fdot)$, $E_{x}$ be the complex John ellipsoid of $K_{x}$, and we set $A_{x}:=\cM\ci{\C,d}(E_{x})$.

\begin{prop}
The map $\cR:X\rightarrow \emph{P}_{d}(\C)$ given by $\cR(x):=A_{x}$, for all $x\in X$, is measurable.
\end{prop}

Consider the map $\cK:X\rightarrow\cC_{d}(\C)$ given by $\cK(x):=K_{x}$, for all $x\in X$. Then, it is clear that
\begin{equation*}
\cR=\cM\ci{\C,d}\circ J\ci{\C,d}\circ\cK.
\end{equation*}
We already saw above that $\cM_{\C,d}|\ci{\cE_d^{\ast}(\C)}:\cE_{d}^{\ast}(\C)\rightarrow\text{P}_{d}(\C)$ and $J\ci{\C,d}:\cC_{d}(\C)\rightarrow\cE^{\ast}_{d}(\C)$ are continuous, therefore measurable. Moreover, similarly to the real case we have that $\cK:X\rightarrow\cC_{d}(\C)$ is measurable, yielding the desired result.

\Addresses

\end{document}